\documentclass[reqno,10pt]{amsart}
\usepackage{amsthm,amsfonts,amssymb,euscript,bbm,amsmath}
\usepackage{empheq}
\usepackage{parskip}
\usepackage{hyperref}
\newcommand{\bea}{\begin{eqnarray}}
\newcommand{\eea}{\end{eqnarray}}
\def\beaa{\begin{eqnarray*}}
\def\eeaa{\end{eqnarray*}}
\def\ba{\begin{array}}
\def\ea{\end{array}}
\def\be#1{\begin{equation} \label{#1}}
\def \eeq{\end{equation}}

\def\be{{\beta}}

\def\al{\alpha}

\def\FF{{\mathcal F}}

\def\OO{{\mathcal O}}

\def\R{{\mathbb{R}}}
\def\C{{\mathbb{C}}}

\def\N{{\bf N}}

\def\Z{{\mathbb{Z}}}

\def\T{{\mathbb{T}}}

\newtheorem{theorem}{Theorem}[section]
\newtheorem{lemma}[theorem]{Lemma}
\newtheorem{proposition}[theorem]{Proposition}
\newtheorem{corollary}[theorem]{Corollary}
\newtheorem{definition}[theorem]{Definition}
\newtheorem{remark}[theorem]{Remark}
\newtheorem{conjecture}[theorem]{Conjecture}

\numberwithin{equation}{section}

\begin{document}

\title[Quintic NLS on $\mathbb{R}\times\mathbb{T}^2$]{On scattering for the quintic defocusing nonlinear Schr\"odinger equation on $\mathbb{R}\times\mathbb{T}^2$.}

\author{Zaher Hani and Benoit Pausader}
\address{Courant Institute of Mathematical Sciences\\ 251 Mercer Street, New York NY 10012}
\email{hani@cims.nyu.edu}
\email{pausader@cims.nyu.edu}

\thanks{The first author was supported in part by a Simons Postdoctoral Fellowship. The second author was supported in part by NSF grant DMS-$1069243$.}

\maketitle

\begin{abstract}

We consider the problem of large data scattering for the quintic nonlinear Schr\"odinger equation on $\mathbb{R}\times\mathbb{T}^2$. This equation is critical both at the level of energy and mass. Most notably, we exhibit a new type of profile (a ``large scale profile'') that controls the asymptotic behavior of the solutions.
\end{abstract}

\section{Introduction}

The purpose of this manuscript is to study the asymptotic behavior of the defocusing quintic nonlinear Schr\"odinger equation on $\R\times \T^2$ given by\footnote{Since our main goal is to attack the \emph{large-data long-time} theory, the regularity of our initial data is dictated by the conservation laws. For large data global regularity and scattering, considering data in $H^1$ or in any smoother space is effectively equivalent.}:
\begin{equation}\label{NLS}
\left(i\partial_t +\Delta_{\mathbb{R}\times\mathbb{T}^2}\right) u =|u|^4u,\quad
u(t=0)= u_0 \in H^1(\R\times \T^2)
\end{equation}
Our main motivation is to better understand the broad question of the effect of the geometry of the domain on the asymptotic behavior of large solutions to nonlinear dispersive equations. While scattering holds for the quintic equation on $\mathbb{R}^3$, it is not expected to hold (apart from trivial cases) on $\mathbb{T}^3$ (cf. Appendix to \cite{CKSTT}). As we will argue below, the situation on $\mathbb{R}\times\mathbb{T}^2$ seems to be a borderline case for this question. 

The study of solutions of the nonlinear Schr\"odinger equation on compact or partially compact domains has been the subject of many works, dating back at least to \cite{BreGal} and then systematically developed by Bourgain \cite{Bo2,Bo1} on tori and Burq-G\'erard-Tzvetkov \cite{BuGeTz,BuGeTz3,BuGeTz2} on compact manifolds, with recent key developments in the energy-critical setting following the work of Herr-Tataru-Tzvetkov \cite{HeTaTz,HeTaTz2,He} and extended to global existence in Ionescu-Pausader \cite{IoPa,IoPa2}. We also refer to \cite{Bouclet,BoTz,GePi,GreTho,Ha,IoPaSt,TaTz} and especially to \cite{BaCaDu,TzVi} for previous works on the relation between scattering and geometry. Such equations have also been extensively studied in applied sciences on various backgrounds. While we will not attempt to make any fairly exhaustive list of those works, we should point out that their study on ``wave-guide" manifolds like $\R\times \T^2$ seems to be of particular interest especially in non-linear optics of telecommunications \cite{Schn, LPA, DFRMLCM}.

The studies on global wellposedness for energy critical and subcritical equations\footnote{That is when the potential part of the energy can be bounded in terms of the kinetic part. We remark that the energy-supercritical case (the complementary case which we do not discuss) is completely unknown, even on Euclidean spaces $\mathbb{R}^d$, except for some ill-posedness results \cite{AlCa,BuGeTz3,Tho}, or some results on the ``barely supercritical'' case \cite{Roy,TaoSuperCrit}.} so far seem to point to the absence of any geometric obstruction to global existence. Indeed in the cases known so far, the solutions  to the defocusing problem all extend globally. In contrast, it is clear that the geometry influences strongly the asymptotic dynamics of solutions but the precise relationship remains poorly understood.

A first step in this direction is to explore when one can obtain the simplest asymptotic behavior, namely scattering, which means that all nonlinear solutions asymptotically resemble linear solutions. From the heuristic that linear solutions with frequency $\sim N$ initially localized around the origin will disperse at time $t$ in the ball of radius $\sim Nt$, one can hope that scattering is partly determined by the asymptotic volume growth of balls with respect to their radius. In fact, if 
\begin{equation*}
V( r):=\inf_{q\in M^d}\{\hbox{Vol}_{M^d}(B(q,r))\}\sim_{r\to\infty} r^g,
\end{equation*}
then one would expect that linear solutions decay at a rate $\sim t^{-g/2}$ and based on the Euclidean theory on $\R^g$, the equation
\begin{equation}\label{PNLS}
\left(i\partial_t+\Delta_{M^d}\right)u=\vert u\vert^{p-1}u,\quad u(0)\in H^1(M^d)
\end{equation}
would scatter in the range $1+4/g\le p\le 1+4/(d-2)$, while one might expect more exotic behavior, at least when $p\le 1+2/g$. Note that this heuristics is consistent with the results in \cite{BaCaDu}.

We don't know whether such a simple picture is accurate but testing this hypothesis motivated us to study the asymptotic behavior for $\eqref{NLS}$ in the case $g=1$ and $d=3$, which seems to be the hardest case that can be addressed in light of the recent developments in \cite{HeTaTz, Dod}. Indeed, as we will argue later, this problem is both mass-critical and energy-critical ($1+4/g=1+4/(d-2)=5$).

Our two main results tend to confirm the picture above about scattering, at least in the case of quotients of Euclidean spaces and insofar as one can parallel the case $p=1+4/g$ and the mass-critical problem in $\mathbb{R}^g$. The first result asserts that small initial data lead to solutions which are global {\it and scatter}.

\begin{theorem}\label{Thm1}
There exists $\delta>0$ such that any initial data $u_0\in H^1(\mathbb{R}\times\mathbb{T}^2)$ satisfying
\begin{equation*}
\Vert u_0\Vert_{H^1(\mathbb{R}\times\mathbb{T}^2)}\le\delta
\end{equation*}
leads to a unique global solution $u\in X^1_c(\mathbb{R})$ of \eqref{NLS} which scatters in the sense that there exists $v^{\pm}\in H^1(\mathbb{R}\times\mathbb{T}^2)$ such that
\begin{equation}\label{Scat}
\Vert u(t)-e^{it\Delta_{\mathbb{R}\times\mathbb{T}^2}}v^{\pm}\Vert_{H^1(\mathbb{R}\times\mathbb{T}^2)}\to0\quad\hbox{as}\quad t\to\pm\infty.
\end{equation}
\end{theorem}
The uniqueness space $X^1_c\subset C_t(\R:H^1(\mathbb{R}\times\mathbb{T}^2))$ was essentially introduced by Herr-Tataru-Tzvetkov \cite{HeTaTz}. 
The main novelty here is the scattering statement on a manifold with such little volume (and so many trapped geodesics\footnote{The presence of trapped geodesics is known to have nontrivial effects on the linear flow and could be expected to also affect the asymptotic behavior of nonlinear solutions.}). Using the time-reversal symmetry, a similar statement holds for negative times as well.
One key fact about Theorem \ref{Thm1} is that it requires only a control provided by the conserved mass and energy of the solution defined respectively by
\begin{equation}\label{mass and energy}
M(u):=\Vert u(0)\Vert_{L^2(\mathbb{R}\times\mathbb{T}^2)}^2,\quad E(u):=\frac{1}{2}\Vert\nabla u(0)\Vert_{L^2(\mathbb{R}\times\mathbb{T}^2)}^2+\frac{1}{6}\Vert u(0)\Vert_{L^6(\mathbb{R}\times\mathbb{T}^2)}^6.
\end{equation}
We will refer to the quantity $L(u):=\frac{1}{2}M(u)+E(u)$ as the ``full energy'' of $u$. We should mention the work of Tzvetkov-Visciglia \cite{TzVi} for a previous scattering statement involving a norm which is not bounded by the full energy.
The main advantage of having a control at the level of a conserved quantity is that one can then hope to extend a small data result to a global result. This is precisely the question we tackle in the main part of this paper.

In order to extend our analysis to large data, we use a method formalized in Kenig-Merle \cite{KeMe} following previous works on critical nonlinear dispersive equations \cite{BaGe,B,CKSTTcrit,Ker}, see also \cite{KiTaVi,Dod3}. One key ingredient is a linear and nonlinear profile decomposition for solutions with bounded energy. The so-called profiles correpond to sequences of solutions exhibiting an extreme behavior (in fact a defect of compactness) and possibly ``leaving'' the geometry. It is there that the ``energy-critical'' and ``mass-critical'' nature of our equation become manifest.

In order to understand the appearance of the profiles, one can argue as follows: in view of the scaling invariance of \eqref{NLS} under
\begin{equation*}
\mathbb{R}_x\times\mathbb{T}^2_y\to M_\lambda:=\R_x \times (\lambda^{-1}\T^2)_y,\quad u\to \tilde u(x,y,t)=\lambda^{1/2}u( \lambda x, \lambda y,\lambda^2t),
\end{equation*}
a uniform control on the solutions on $M_1$ in terms of their full energy yields a control on solutions on $M_\lambda$ uniformly in $\lambda$. Therefore, we would expect a solution initially at ``scale one'' on $M_\lambda$ to remain so. We then observe the two natural scaling limits:

$i)$ {\bf Small-scale limit:} when $\lambda \to 0$, a ``scale-one" solution on $M_\lambda$ should not ``sense'' the distinction between $M_\lambda$ and $\mathbb{R}^3$ and should therefore behave as a solution to the energy-critical NLS on $\mathbb{R}^3$. This corresponds in $M_1$ to solutions with initial data
\begin{equation*}
u^\lambda(x,y,0)=\lambda^{-1/2} \phi(\lambda^{-1}(x,y)),\quad\phi\in C^\infty_c(\mathbb{R}^3), \lambda \to 0.
\end{equation*}
We call such profiles Euclidean profiles. Their appearance is a manifestation of the energy-critical nature of the nonlinearity\footnote{In the sense that it corresponds to the critical exponent with respect to Sobolev embedding; in other words, with the notations in \eqref{PNLS}, $p=1+4/(d-2)$.}.

$ii)$ {\bf Large-scale limit: }when $\lambda \to+\infty$ the manifold $M_\lambda$ becomes thinner and thinner and ``resembles $\mathbb{R}$''. It is then tempting to guess that scale-one solutions on $M_\lambda$ will evolve, on time-scales of size $\lambda^{2}$, as the solutions to the quintic (mass-critical) nonlinear Schr\"odinger equation on $\mathbb{R}$:
\begin{equation}\label{MC}
\left(i\partial_t+\partial_{xx}\right)u=\vert u\vert^4u,\quad u(0)\in H^1(\mathbb{R})
\end{equation}
and this is indeed the case when the initial data is constant in the compact ($y$) variable. However, the general situation is more complicated as different Fourier modes will interact in a nontrivial manner and this ``fast dynamics'' has a profound effect on the slow dynamics coming from \eqref{MC}. Indeed the rigorous analysis of this ``multi-scale dynamics'' is one of the main component of this paper.

These solutions on $M_\lambda$ correspond to what we call the ``large-scale profiles'' on $M_1$ which are initial data given by:
\begin{equation*}
u^\lambda(x,y,0)=\lambda^{-1/2}\phi(\lambda^{-1} x,y),\quad\phi\in C^\infty_c(\mathbb{R}\times\mathbb{T}^2), \lambda \to +\infty
\end{equation*}
and their appearance seems to be a manifestation of the mass-critical nature of the equation. We will argue below that these solutions should follow the flow associated to a more general flow than that of \eqref{MC}, namely the one given by quintic resonant system on $\R$ that we now describe.

The quintic resonant system is given by
\begin{equation}\label{SS}
\begin{split}
\left(i\partial_t +\partial_{xx}\right) u_j =&\sum_{\mathcal{R}(j)} u_{j_1}\overline{u_{j_2}}u_{j_3}\overline{u_{j_4}}u_{j_5} \quad j \in \Z^2\\
\mathcal{R}(j)=&\{(j_1, j_2, j_3, j_4, j_5) \in (\Z^2)^5: j_1-j_2+j_3-j_4+j_5=j \hbox{ and }\\
 & |j_1|^2-|j_2|^2+|j_3|^2-|j_4|^2+|j_5|^2=|j|^2\}
\end{split}
\end{equation}
with unknown ${\vec u}=\{u_j\}_{j\in\mathbb{Z}^2}$, where $u_j:\mathbb{R}_x\times\mathbb{R}_t\to\mathbb{C}$. In the special case when $u_j=0$ for $j\ne 0$, we recover \eqref{MC}, but in general, this is a new equation. Similar finite or infinite systems of nonlinear Schr\"odinger equations arise independently in the study of nonlinear optics in waveguides and are the object of several previous studies\footnote{They sometimes go by the name of ``vector nonlinear Schr\"odinger equation'' or VNLS.} (see e.g. \cite{Cole,Gerd,MeSoSt}, the books \cite{AbPrTr,SuSu} and references therein).

As we show in an appendix, the system \eqref{SS} is Hamiltonian, has a nice local theory and retains many properties of \eqref{MC}. In view of this and of the result of Dodson \cite{Dod}, it seems reasonable to formulate the following conjecture:
\begin{conjecture}\label{SSConj}
Let $E \in (0, \infty)$. For any smooth initial data $\vec u_0$ satisfying:
\begin{equation*}
E_{ls}(\vec u_0):=\frac{1}{2}\sum_{j\in \Z} \langle j\rangle^2 \|u_{0,j}\|_{L^2(\mathbb{R})}^2\leq E
\end{equation*}
there exists a global solution of \eqref{SS}, $\vec u(t)$, ${\vec u}(t=0)={\vec u_0}$ with conserved $E_{ls}(\vec u(t))=E_{ls}(\vec u_0)$  satisfying:
\begin{equation}\label{SSBd}
\Vert {\vec u}\Vert_{\vec W}^2:=\sum_{j\in\mathbb{Z}}\langle j\rangle^2\Vert u_j\Vert_{L^6_{x,t}(\mathbb{R}_x\times\mathbb{R}_t)}^2\leq \Lambda_{ls}(E_{ls}({\vec u_0})).
\end{equation}
for some finite non-decreasing function $\Lambda_{ls}(E)$.
\end{conjecture}

We can now give the main result of this paper which asserts large data scattering for \eqref{NLS} conditioned on Conjecture \ref{SSConj}.

\begin{theorem}\label{MainThm}
Assume that Conjecture \ref{SSConj} holds for all $E\le E_{max}^{ls}$, then any initial data $u_0\in H^1(\mathbb{R}\times\mathbb{T}^2)$ satisfying
\begin{equation*}
L(u_0)=\int_{\mathbb{R}\times\mathbb{T}^2}\left\{\frac{1}{2}\vert u_0\vert^2+\frac{1}{2}\vert\nabla u_0\vert^2+\frac{1}{6}\vert u_0\vert^6\right\}dx\le E_{max}^{ls}
\end{equation*}
leads to a solution $u\in X^1_c(\R)$ which is global, and scatters in the sense that there exists $v^{\pm}\in H^1(\mathbb{R}\times\mathbb{T}^2)$ such that \eqref{Scat} holds. In particular, if $E_{max}^{ls}=+\infty$, then all solutions of \eqref{NLS} with finite energy and mass scatter.
\end{theorem}

Before we go into the details of the proof, a few remarks about the above theorem are in order. First, we should point out that the global regularity part holds for all solutions of finite energy, unconditional on Conjecture \ref{SSConj}. Second, as a consequence of the local theory for the system \eqref{SS}, one has that Conjecture \ref{SSConj} holds below a nonzero threshold $ E_{max}^{ls}>0$, so Theorem \ref{MainThm} is non-empty and actually strengthens Theorem \ref{Thm1}. Actually, by simple modifications of the proof, one can find regimes with large mass (for example using the main result in \cite{Dod}) where one can obtain nontrivial large data that scatter. Another point worth mentioning is that while Theorem \ref{MainThm} is stated as an implication, it is actually an equivalence as it is easy to see that one can reverse the analysis needed to understand the behavior of large-scale profile initial data for \eqref{NLS} in order to control general solutions of \eqref{SS} and prove Conjecture \ref{SSConj} assuming that Theorem \ref{MainThm} holds (cf. Appendix). Finally, we note that the full resolution of Conjecture \ref{SSConj} seems to require considerable additional work that is completely independent of the analysis on $\mathbb{R}\times\mathbb{T}^2$, so we choose to leave it for a later work.

The proof of Theorem \ref{MainThm} has a fairly standard skeleton based on the Kenig-Merle machinery \cite{KeMe} on one hand, and the recent works devoted to understanding inhomogeneous critical equations \cite{IoPaSt, IoPa, IoPa2, KiKwShVi, KiStVi} on the other. Nonetheless, the global geometry, the two levels of criticality (mass and energy), and the emergence of large-scale profiles require some novel ideas and nonstandard adaptations. Most notably: i) proving good global Strichartz estimates not only to prove Theorem \ref{Thm1} but also to obtain an $L^2-$profile decomposition suitable for the large data theory, ii) the analysis of the large-scale profile initial data that appear in the profile decomposition, understanding their ``two time-scale'' behavior in terms of the quintic system \eqref{SS} via a normal form transformation, and iii) a final nonlinear profile recomposition similar to that in \cite{IoPa} but with many more cases. We elaborate briefly on those three points:

\subsection{Global Strichartz estimates} While local-in-time Strichartz estimates are sufficient for the purposes of local existence and even global regularity \cite{IoPa, IoPa2}, they are not useful to obtain information about the asymptotic behavior, even for small data. On the other hand, in order to prove scattering with critical control (at the level of $H^1$), one cannot afford to work with Strichartz estimates that lose too many derivatives. These two limitations (good global-in-time integrability and derivative loss) encapsulate the main difficulties at this level. 

By using the Strichartz estimates on $\R$ and Sobolev embedding, one can obtain global-in-time bounds with optimal scaling such as:
$$
\|e^{it\Delta_{\R\times \T^2}}u_0\|_{L^6_{x,y,t}(\R\times \T^2\times\mathbb{R})} \lesssim \|\langle \nabla\rangle^{\frac{2}{3}}u_0\|_{L_{x,y}^2(\R\times \T^2)}.
$$
However, this derivative loss of $2/3$ derivatives does not allow for a local theory at the level of $H^1$, which seems to require estimates with less than $1/2$ derivatives lost. On the other hand, one can also easily obtain \emph{local-in-time} Strichartz estimates with satisfactory derivative loss from the corresponding estimates on $\T^2$, but with no obvious way to extend them globally. In addition, application of the Hardy-Littlewood circle method as in \cite{Bo2,IoPa} requires the use of the same $L^p$ norms in $x,y,t$ and no such $p$ can give both sufficiently good derivative loss and global-in-time integrability.

This suggests using norms that distinguish between local-in-time and global-in-time integrability. We are thus lead to the following Strichartz estimates:
$$
\|e^{it\Delta_{\R\times \T^2}}u_0\|_{l^q_\gamma L^r_{x,y,t}(\R\times \T^2\times [\gamma, \gamma+1])}\lesssim \|\langle \nabla \rangle^{\frac{3}{2}-\frac{5}{r}} u_0\|_{L^2_{x,y}(\R\times \T^2)}, \quad \frac{2}{q}+\frac{1}{r}=\frac{1}{2}, \quad 4<q,r<+\infty,
$$
where we take the $l^q$ sum in $\gamma$ of the $L^p_{x,y,t}(\R\times \T^2\times [\gamma, \gamma+1])$ norm on the left hand side. Notice that the condition on $q$ is exactly the $1d$ Schr\"odinger admissibility condition on $\R$ whereas the derivative loss of $3/2-5/r$ is exactly that on $\T^3$ (or $\R^3$). The proof of this estimate follows a $TT^*$ argument that allows to decompose the relevant inner product into a diagonal part (the solution at time $t$ interacting with itself at a similar time) and a non-diagonal part (the solution at time $t$ interacting with itself at time $s$ with $|s-t|\gg 1$). The diagonal component leads to the loss of derivatives, but gives a contribution that has better time-integrability (like $l_\gamma^2$). The non-diagonal part loses fewer derivatives but forces the above-mentioned slower $l_\gamma^q$ time-integrability and requires a nontrivial adaptation of the Hardy-Littlewood circle method in a way that incorporates the decay in $|s-t|$ coming from the $\R$ direction. The fact that we can decouple those two difficulties (loss of derivatives and slow integrability) is crucial in obtaining a good profile decomposition in this mass-critical context.

\subsection{Large-scale profiles and the resonant quintic system} The treatment of the large-scale profiles alluded to in our previous heuristic and formalized using the profile decomposition leads to the quintic resonant system \eqref{SS}. This system is not completely unexpected as it can be derived from \eqref{NLS} by removing all non-resonant interactions between the periodic Fourier modes of $u$. In general, one would expect non-resonant interactions to have important large-time contributions especially for large data. However, in the \emph{special limit} given by those large-scale profile initial data, one can integrate out the non-resonant dynamics by using a normal form transformation and this yields \eqref{SS}.

Among the three types of sequences appearing in the profile decomposition, the large-scale profiles correspond to sequences of solutions of \eqref{NLS} with initial data roughly of the form:
$$
\psi_k(x,y)=M_k^{1/2}\psi(M_k x,y), \quad \hbox{ with } \psi\in C^\infty_c(\mathbb{R}\times\mathbb{T}^2),\hbox{ and }M_k \to 0.
$$
Predicting that this scaling is conserved leads one to consider $u(x,y,t) = M_k^{1/2}\tilde u ( M_k x, y,t)$ where $\tilde u$ then satisfies the equation:
$$
(i\partial_t + \Delta_{\T^2})\tilde u =M_k^2\left(|\tilde u|^4 \tilde u-\partial_{xx}\tilde{u}\right),  \quad \tilde u(0)=\psi.
$$
This suggests conjugating with respect to the corresponding semigroup $e^{it\Delta_{\mathbb{T}^2}}$ and, after a rescaling,
we see that $v$ satisfies an equation that is essentially:
$$
\left(i\partial_t +\partial_{xx} \right) v=\mathcal{N}( v)=\mathcal O( v^5)
$$
Since one would expect solutions of this equation to disperse after some large time $T_0$ (thanks to the dispersion of $e^{it\Delta_\R}$), we only need to control the dynamics up to  $T_0$. But over this time interval, the contribution of the non-resonant interactions is $O(T_0M_k^{2})$ and hence can be ignored in the limit as $M_k \to 0$. Doing this leads to the quintic resonant system \eqref{SS}. The main conclusion of this analysis is that large-scale profiles admit the following ``multi-scale'' description\footnote{Indeed, this describes $u$ in terms of two ``times'' $t$ and $T$.}
\begin{equation*}
u(x,y,t)=e^{it\Delta_{\mathbb{T}^2}}\left[M_k^\frac{1}{2}v(X,y,T)\right],\quad X=M_kx,\,\,T=M_k^2t,
\end{equation*}
where $v$ solves \eqref{SS}.

\subsection{Profile recomposition and end of the proof} To end the argument, one needs to be able to write an approximate solution of $\eqref{NLS}$ with initial data given as a sum of orthogonal profiles and dispersed initial data (i.e.~one whose linear development is small in the relevant norm). A first guess for this approximate solution would be to sum the nonlinear profiles (solution of \eqref{NLS} with initial data given by the profiles) and the linear propagation of the perturbation. As in \cite{IoPa}, this doesn't quite work because of terms that are linear in the perturbation. We follow \cite{IoPa} and focus on the second iterate of the Duhamel formula. While this does indeed give a good approximate solution it leaves us with proving that the newly defined perturbation is still small in the appropriate norm.

The main problem we face comes from perturbations that are concentrated at high frequencies and pointing only in the non-dispersive direction (near the zero frequency of $\FF_x$). These solutions are close to concentrate on the trapped geodesics of $\R\times \T^2$ and they linger for a long time in compact regions of the manifold before they disperse, thereby interacting with solutions at larger scales for a long time. The redeeming features of such solutions are that i) their initial Fourier support makes them far from saturating the Strichartz estimates, so at least their linear development is small in a weak norm which makes their interaction with smaller scale profiles ineffective; ii) the qualitative fact that they are located in a small angular sector is conserved when they interact with larger scale nonlinear profiles, so the output of this interaction has similar properties. Implementing this requires introducing norms that are sensitive to the special Fourier support properties of the initial data and obtaining a good well-posedness theory to propagate the initial Fourier information effectively to later times\footnote{Note that just putting the gradient of the highest frequency in $L^2$ loses any such information.}. 

The paper is organized as follows: in Section \ref{Not} we fix our notations and recall the function spaces we use; in Section \ref{Strichartz section} we prove the global Strichartz estimates we need for the rest of our work. Section \ref{Section-LWP} contains the local well-posedness and small-data scattering (Theorem \ref{Thm1}) along with the associated stability theory. In Section \ref{Sec-ProfileDec}, we obtain a good linear profile decomposition that leads us to the types of profiles we need to analyze, which is what we do in advance in Section \ref{Sec-DesPro} for the Euclidean and large-scale profiles. In Section \ref{Sec-IOE} we prove the contradiction argument leading to Theorem \ref{MainThm}. Finally, we develop the local existence theory for \eqref{SS} in Section \ref{Sec-AppSS}.


\section{Notation and Function Spaces}\label{Not}

\subsection{Notations} 

We write $A\lesssim B$ to signify that there is a constant $C>0$ such that $A\le C B$. We write $A\simeq B$ when $A\lesssim B\lesssim A$. If the constant $C$ involved has some explicit dependency, we emphasize it by a subscript. Thus $A\lesssim_uB$ means that $A\le C(u)B$ for some constant $C(u)$ depending on $u$.
For $p\in\mathbb{N}^n$ a multi-index we denote by $\mathfrak{O}_{p_1,\dots ,p_n}(a_1,\dots,a_n)$ a $\vert p\vert$-linear expression which is a product of $p_1$ terms which are either equal to $a_1$ or its complex conjugate $\overline{a}_1$ and similarly for $p_j$, $a_j$, $2\le j\le n$.

We label the coordinates on $\mathbb{R}\times\mathbb{T}^2$ as $(x,y_1,y_2)\in\mathbb{R}\times\mathbb{T}\times\mathbb{T}$.
We fix $\eta^1\in C^\infty(\mathbb{R})$ such that $\eta^1(x)=1$ if $\vert x\vert\le 1$ and $\eta^1(x)=0$ when $\vert x\vert\ge 2$. We also let $\eta^3(a,b,c)=\eta^1(a)\eta^1(b)\eta^1( c)$, so that for $x\in \R^k$ and $k \in \{1,3\}$:
\begin{equation*}
\begin{split}
\eta^k_{\le N}(x)=\eta^k(x/N),\quad \eta^k_N(x)=\eta^k_{\le N}(x)-\eta^k_{\le N/2}(x),\quad \eta^k_{\ge N}(x)=1-\eta^k_{\le N/2}(x).
\end{split}
\end{equation*}

For any dyadic number $N \geq 1$, we define the Littlewood-Paley projectors:
$$P_{\le N}:=\eta^3_{\leq N}(i\nabla), \quad P_{N}:=P_{\leq N} -P_{\leq N/2} \hbox{ for }N\geq 2, \hbox{ and }P_1:=P_{\leq 1}.
$$ 
Since we sometimes need to distinguish the regularity in the different directions, we similarly define $P^x_{\le l}$ to be the Fourier multiplier associated to $\eta^1_{\le M}(i\partial_x)$ and $P^y_{\le N}$ to be the multiplier associated to $\eta^1_{\le N}(i\partial_{y_1})\eta^1_{\le N}(i\partial_{y_2})$. We define $P^x_M$ and $P^y_M$ similarly as before. We will also need the following angular frequency localization: for $\delta>0$, we define:
\begin{equation}\label{def of tilde P}
\tilde{P}^1_\delta=\sum_{N\ge 1}P_NP_{\ge\delta N}^x.
\end{equation}

In addition to the usual isotropic Sobolev spaces $H^s(\R\times \T^2)$, we will need non-isotropic versions. For $s_1, s_2 \in \R$ we define:
\begin{equation*}\label{DefHH}
H^{s_1, s_2}(\R\times \T^2)= \{u: \R\times \T^2 \to \C: \langle \xi\rangle^{s_1}\langle n \rangle^{s_2} \widehat u(\xi,n) \in L^2_{\xi, n}(\R \times \Z^2)\}.
\end{equation*}

We will be particularly interested in the space $H^{0,1}(\R\times \T^2)$. This is the set of functions $\psi\in L^2_{x,y}(\mathbb{R}\times\mathbb{T}^2)$ such that $\nabla_y\psi\in L^2_{x,y}(\mathbb{R}\times\mathbb{T}^2)$. We can also define a discrete analogue. For $\vec \phi=\{\phi_p\}_{p\in \Z^2}$ a sequence of real-variable functions, we let
\begin{equation}\label{DefhH}
h^{s_1}H^{s_2}:=\left\{{\vec \phi}=\{\phi_p\}:\Vert{\vec \phi}\Vert_{h^{s_1}H^{s_2}}^2=\sum_{p\in\mathbb{Z}^2}\langle p\rangle^{2s_1}\Vert\phi_p\Vert_{H^{s_2}}^2<+\infty\right\}.
\end{equation}
We can naturally identify $H^{0,1}(\mathbb{R}\times\mathbb{T}^2)$ and $h^1L^2$ by via the Fourier transform in the periodic variable $y$ as follows:
for $\psi\in H^{0,1}(\mathbb{R}\times\mathbb{T}^2)$ we define the vector ${\vec \psi}=\{\psi_p\}_p$ where
\begin{equation}\label{DiscFourier}
\psi_p(x):=\frac{1}{(2\pi)^2}\int_{\mathbb{T}^2}\psi(x,y)e^{-i\langle p,y\rangle} dy
\end{equation}
is the sequence of periodic Fourier coefficients of $\psi$. Conversely, for any ${\vec v}\in h^1L^2$, we define
\begin{equation*}
\underline{{\vec v}}(x,y):=\sum_{p\in\mathbb{Z}^2}v_p(x)e^{i\langle y,p\rangle}.
\end{equation*}
Clearly the two operations invert one another and are isometries\footnote{Up to the normalizing constants in the definition of the Fourier transform.}.

\subsection{Function spaces}

For $C=[-\frac{1}{2}, \frac{1}{2})^3\subset \R^3$ and $z\in \R^3$, we denote by $C_z=z+C$ the translate of $C$ by $z$ and define the sharp projection operator $P_{C_z}$ as follows: 
$$
\FF \left(P_{C_z}f\right)=\chi_{C_z}(\xi)\FF(f)(\xi).
$$
We use the same modifications of the atomic and variation space norms that were employed in \cite{HeTaTz, HeTaTz2, IoPa}.  Namely, for $s\in \R$, we define:
\begin{equation*}
\begin{split}
\|u\|_{X_0^s(\R)}^2&=\sum_{z\in \Z^3}\langle z\rangle^{2s}\|P_{C_z}u\|_{U^2_\Delta(\R;L^2)}^2,\quad \|u\|_{Y^s(\R)}^2=\sum_{z\in \Z^3}\langle z\rangle^{2s}\|P_{C_z}u\|_{V^2_\Delta(\R;L^2)}^2.
\end{split}
\end{equation*}
We refer to \cite{HaHeKo,HeTaTz, HeTaTz2} for the description and properties of the spaces $U^p_\Delta$ and $V^p_\Delta$. The norms $X^s_0$ and $Y^s$ are both stronger than the $L^\infty(\R; H^s)$ norm and weaker than the $U_\Delta^2(\mathbb{R}:H^s)$ norm. In fact, they satisfy the following nesting property:
\begin{equation*}
U^2_\Delta(\R; H^s) \hookrightarrow X_0^s(\R) \hookrightarrow Y^s(\R) \hookrightarrow V^2_\Delta(\R; H^s)\hookrightarrow U^p_\Delta(\R;H^s) \hookrightarrow L^\infty(\R; H^s)
\end{equation*}
for any $p>2$.
For a compact interval $I\subset \R$, we also need the restriction norms $X^s(I)$ and $Y^s(I)$ defined in the usual way:
\begin{equation*}
\|u\|_{X^s(I)}=\inf\{\|v\|_{X_0^s(\R)}: v\in X_0^s(\R)\text{ satisfying } v_{|I}=u_{|I}\}
\end{equation*}
and similarly for $Y^s(I)$.
Our main solution space is then
\begin{equation*}
\begin{split}
X^s_c(\R):=\{u\in C(\mathbb{R}:H^s(\R\times \T^2)):\,&\phi_{-\infty}:=\lim_{t\to -\infty}e^{-it\Delta}u(t) \hbox{ exists in $H^s$, and }\,u(t)-e^{it\Delta}\phi_{-\infty} \in X^s_0(\R)\}
\end{split}
\end{equation*}
equipped with the norm:
\begin{equation}\label{def of X^1+}
\|u\|_{X^s(\R)}^2:=\|\phi\|_{H^s(\R\times \T^2)}^2+\|u-e^{it\Delta}\phi_{-\infty}\|_{X^s_0(\R)}^2\simeq \sup_{K\subset \mathbb{R},K\text{ compact}}\Vert u\Vert_{X^s(K)}^2
\end{equation}
and we can extend the second definition of the $X^s$-norm to arbitrary subintervals of $\mathbb{R}$.
We also consider $X^1_{c,\operatorname{loc}}(I)$ to be the set of all functions in $C_{\operatorname{loc}}(I:H^1)$ whose $X^1(J)$-norm is finite for every compact subset $J\subset I$. 

To control the nonlinearity on an interval $I=(a,b)$ we use the norm:
\begin{equation}\label{def of N norm}
\|h\|_{N^s(I)}:= \left \| \int_a^t e^{i(t-s)\Delta} h(s) ds\right\|_{X^s(I)}. 
\end{equation}
We mostly use this norm for $s=1$ and in this case, we omit the $s$ in the notation.

%

In addition to the above norms, we need the following time-divisible norm:
\begin{equation}\label{def of Z}
\|u\|_{Z(I)}=\sum_{p_0=9/2, 18} \left(\sum_{N\geq 1}N^{p_0(\frac{5}{p_0}-\frac{1}{2})}\|\mathbbm{1}_I(t)P_N u\|^{p_0}_{l^{\frac{4p_0}{p_0-2}}_\gamma L_{x,y,t}^{p_0}( \R \times \T^2\times I_\gamma)}\right)^{\frac{1}{p_0}}.
\end{equation}
where $I_\gamma=[2\pi\gamma,2\pi(\gamma+1)]$.
$Z$ is a weaker norm than $X^1$. In fact, from Lemma \ref{U^p lemma} below,
\begin{equation*}
\|u\|_{Z(I)}\lesssim \|u\|_{X^1(I)}.
\end{equation*}

We will also need the following proposition which is the analogue of Poposition 2.11 of \cite{HeTaTz} or Proposition 2.10 of \cite{HeTaTz2}:

\begin{proposition}[\cite{HeTaTz, HeTaTz2}]\label{N prop}
If $f\in L_t^1(I, H^1(\mathbb{R}\times \T^2))$, then 
$$
\|f\|_{N(I)}\lesssim \sup_{\substack{v\in Y^{-1}(I)\\ \|v\|_{Y^{-1}(I)}\leq 1}} \int_{I \times (\R\times \T^2)} f(x,t) \overline{v(x,t)} dxdt 
$$
In particular, the following estimate holds for any smooth function $g$ on an interval $I=[a,b]$:
\begin{equation}\label{X1SimpleEst}
\|g\|_{X^1(I)} \lesssim \|g(0)\|_{H^1(\mathbb{R}\times\mathbb{T}^2)}+\left(\sum_{N}\|P_N(i\partial_t +\Delta)g\|^2_{L_t^1(I,H_x^1(\R\times \T^2))}\right)^{1/2}.
\end{equation}
\end{proposition}


\section{Global Strichartz estimates }\label{Strichartz section}

We recall the Strichartz estimates from \cite{Bo2}: for any $p>4$, there holds that
\begin{equation}\label{StricT}
\begin{split}
\Vert e^{it\Delta_{\mathbb{T}^3}}P_N f\Vert_{L^p_{x,t}(\mathbb{T}^3_x\times\mathbb{T}_t)}&\lesssim N^{\left(\frac{3}{2}-\frac{5}{p}\right)}\Vert f\Vert_{L^2(\mathbb{T}^3)}\\
\end{split}
\end{equation}
This bound can of course be adapted to the case of $\mathbb{R}\times\mathbb{T}^2$. However, in order to prove scattering estimates, we need to obtain global in time bounds. This is only possible by using the $\mathbb{R}$ component of our manifold. But this only supports linear solutions that decay much slower than on $\mathbb{R}^3$. As a result, we need to distinguish between the ``local in time'' integrability, similar to the one in $\mathbb{R}^3$ (or $\mathbb{T}^3$) and the ``global in time'' integrability, similar to the one on $\mathbb{R}$. This is done by decomposing
\begin{equation*}
\mathbb{R}_t=\bigcup_{\gamma\in\mathbb{Z}}2\pi[\gamma,\gamma+1).
\end{equation*}
In this section, we prove the following result:
\begin{theorem}\label{Striclem}
Let $N\geq 1$ be dyadic, then
\begin{equation}\label{Stric1}
\Vert e^{it\Delta_{\mathbb{R}\times\mathbb{T}^2}}P_{\le N} u_0\Vert_{l^q_\gamma L^{p}_{x,y,t}(\mathbb{R}\times\mathbb{T}^2\times [2\pi\gamma,2\pi(\gamma+1)])}\lesssim N^{\left(\frac{3}{2}-\frac{5}{p}\right)}\Vert u_0\Vert_{L^2(\mathbb{R}\times\mathbb{T}^2)}.
\end{equation}
whenever
\begin{equation}\label{Adm}
p>4\quad\hbox{ and }\quad \frac{2}{q}+\frac{1}{p}=\frac{1}{2}.
\end{equation}
\end{theorem}
In the estimate above, we see that the loss in derivatives is dictated by the local ($3d$) geometry, while the global integrability condition is dictated by the global ($1d$) geometry. Theorem \ref{Striclem} follows directly by duality from the following more precise estimate
\begin{lemma}\label{Stric}
For any $h\in C^\infty_c(\mathbb{R}_x\times\mathbb{T}^2_y\times\mathbb{R}_t)$, there holds that
\begin{equation}\label{Adj}
\begin{split}
&\left \Vert \int_{s\in\mathbb{R}}e^{-is\Delta_{\mathbb{R}\times\mathbb{T}^2}}P_{\le N}h(x,y,s)ds\right \Vert_{L^2_{x,y}(\mathbb{R}\times\mathbb{T}^2)}\\
&\lesssim N^{\left(\frac{3}{2}-\frac{5}{p}\right)}\Vert h\Vert_{l^{2}_\gamma L^{p^\prime}_{x,y,t}(\mathbb{R}\times\mathbb{T}^2\times [2\pi\gamma,2\pi(\gamma+1)])}+N^{\left(1-\frac{3}{p}\right)}\Vert h\Vert_{l^{q^\prime}_\gamma L^{p^\prime}_{x,y,t}(\mathbb{R}\times\mathbb{T}^2\times [2\pi\gamma,2\pi(\gamma+1)])}
\end{split}
\end{equation}
for any $(q,p)$ satisfying \eqref{Adm}.
\end{lemma}

\begin{proof}[Proof of Lemma \ref{Striclem}]
In order to distinguish between the large and small time scales, we choose a smooth partition of unity
\begin{equation*}
1=\sum_{\gamma \in \mathbb{Z}} \chi(t -2\pi\gamma)
\end{equation*}
with $\chi$ supported in $[-2\pi,2\pi]$. We also denote by $h_\alpha(t):=\chi(t)h(2\pi\alpha+t)$.
Using the semigroup property and the unitarity of $e^{it\Delta_{\mathbb{R}\times\mathbb{T}^2}}$, we compute directly that
\begin{equation*}
\begin{split}
&\Vert \int_{s\in\mathbb{R}}e^{-is\Delta_{\mathbb{R}\times\mathbb{T}^2}}P_{\le N}h(x,y,s)ds\Vert_{L^2_{x,y}(\mathbb{R}\times\mathbb{T}^2)}^2\\
&=\iint_{s,t\in\mathbb{R}}\langle e^{-is\Delta_{\mathbb{R}\times\mathbb{T}^2}}P_{\le N}h(s),e^{-it\Delta_{\mathbb{R}\times\mathbb{T}^2}}P_{\le N}h(t)\rangle_{L^2_{x,y}\times L^2_{x,y}}dsdt\\
&=\sum_{\alpha,\beta}\iint_{s,t\in\mathbb{R}}\langle \chi(s-2\pi\alpha)e^{-is\Delta_{\mathbb{R}\times\mathbb{T}^2}}P_{\leq N}h(s),\chi(t-2\pi\beta)e^{-it\Delta_{\mathbb{R}\times\mathbb{T}^2}}P_{\leq N}h(t)\rangle_{L^2_{x,y}\times L^2_{x,y}}dsdt\\
&=\sum_{\alpha,\beta}\iint_{s,t\in[-2\pi,2\pi]}\langle e^{-i(2\pi(\alpha-\beta)+s)\Delta_{\mathbb{R}\times\mathbb{T}^2}}P_{\leq N}h_\alpha(s),e^{-it\Delta_{\mathbb{R}\times\mathbb{T}^2}}P_{\leq N}h_\beta(t)\rangle_{L^2_{x,y}\times L^2_{x,y}}dsdt\\
&=\sigma_{\operatorname{diag}}+\sigma_{\operatorname{non-diag}}
\end{split}
\end{equation*}
where
\begin{equation*}
\begin{split}
\sigma_{\operatorname{diag}}&=\sum_{\vert\alpha-\beta\vert\le 9}\iint_{s,t}\langle e^{-i(2\pi(\alpha-\beta)+s)\Delta_{\mathbb{R}\times\mathbb{T}^2}}P_{\leq N}h_\alpha(s),e^{-it\Delta_{\mathbb{R}\times\mathbb{T}^2}}P_{\leq N}h_\beta(t)\rangle_{L^2_{x,y}\times L^2_{x,y}}dsdt\\
\sigma_{\operatorname{non-diag}}&=\sum_{\alpha,\gamma\in\mathbb{Z},\vert\gamma\vert\ge 10}\iint_{s,t}\langle e^{-i(s-2\pi\gamma)\Delta_{\mathbb{R}\times\mathbb{T}^2}}P_{\leq N}h_\alpha(s),e^{-it\Delta_{\mathbb{R}\times\mathbb{T}^2}}P_{\leq N}h_{\alpha+\gamma}(t)\rangle_{L^2_{x,y}\times L^2_{x,y}}dsdt.
\end{split}
\end{equation*}
The diagonal sum dictates the loss of derivative, while the non-diagonal sum dictates the decay.

\medskip

{\bf Diagonal Estimates.} Here, we use the following result from \cite{Bo2}: for any $f\in L^2(\mathbb{T}^2)$ and any $p>4$, there holds that
\begin{equation*}
\Vert e^{it\Delta_{\mathbb{T}^2}}P^y_{\le N}f\Vert_{L^p(\mathbb{T}^2_y\times\mathbb{T}_t)}\lesssim N^{\left(1-\frac{4}{p}\right)}\Vert f\Vert_{L^2(\mathbb{T}^2_y)}.
\end{equation*}
This implies that for $p>4$,
\begin{equation*}
\Vert e^{it\Delta_{\mathbb{R}\times\mathbb{T}^2}}P_{\le N}f\Vert_{L^p_{x,y,t}(\mathbb{R}\times\mathbb{T}^2\times [0,2\pi])}
\lesssim N^{\left(\frac{3}{2}-\frac{5}{p}\right)}\Vert f\Vert_{L^2_{x,y}(\mathbb{R}\times\mathbb{T}^2)} .
\end{equation*}
Indeed, using the Hausdorff-Young, H\"older, Minkowski and Plancherel inequalities, we have that
\begin{equation*}
\begin{split}
\Vert e^{it\Delta_{\mathbb{R}\times\mathbb{T}^2}}P_{\le N}f\Vert_{L^p_{x,y,t}(\mathbb{R}\times\mathbb{T}^2\times [-2\pi,2\pi])}
&=\Vert e^{it\Delta_{\mathbb{R}\times\mathbb{T}^2}}P_{\le N}f\Vert_{L^p_x(\mathbb{R};L^p_{y,t}(\mathbb{T}^2\times [-2\pi,2\pi]))}\\
&\lesssim N^{\left(\frac{1}{2}-\frac{1}{p}\right)}\Vert e^{-it\vert\xi\vert^2}\mathcal{F}_\xi e^{it\Delta_{\mathbb{T}^2}}P_{\le N}f\Vert_{L^2_\xi(\mathbb{R};L^p_{y,t}(\mathbb{T}^2\times [-2\pi,2\pi]))}\\
&\lesssim N^{\left(\frac{1}{2}-\frac{1}{p}\right)}\Vert e^{it\Delta_{\mathbb{T}^2}}P_{\le N}^y\mathcal{F}_\xi f\Vert_{L^2_\xi(\mathbb{R};L^p_{y,t}(\mathbb{T}^2\times [-2\pi,2\pi]))}\\
&\lesssim N^{\left(\frac{3}{2}-\frac{5}{p}\right)}\Vert \mathcal{F}_\xi f\Vert_{L^2_\xi L^2_y}=N^{\left(\frac{3}{2}-\frac{5}{p}\right)}\Vert f\Vert_{L^2_{x,y}(\mathbb{R}\times\mathbb{T}^2)}.
\end{split}
\end{equation*}
Again, by duality, this implies that
\begin{equation}\label{DiagDual}
\Vert \int_{s\in\mathbb{R}}e^{-is\Delta_{\mathbb{R}\times\mathbb{T}^2}} P_{\leq N}h(s)ds \Vert_{L^2_{x,y}(\mathbb{R}\times\mathbb{T}^2)}
\lesssim N^{\left(\frac{3}{2}-\frac{5}{p}\right)}\Vert h\Vert_{L^{p^\prime}_{x,y,t}(\mathbb{R}\times\mathbb{T}^2\times [-2\pi,2\pi])}
\end{equation}
whenever $h$ is supported in $[-2\pi,2\pi]$.
And consequently,
\begin{equation}\label{Sigma1}
\begin{split}
\sigma_{\operatorname{diag}}&=\sum_{\alpha\in\mathbb{Z},\vert\gamma\vert\le 9}\iint_{s,t\in\mathbb{R}}\langle e^{-i(s-2\pi\gamma)\Delta_{\mathbb{R}\times\mathbb{T}^2}}P_{\leq N}h_\alpha(s),e^{-it\Delta_{\mathbb{R}\times\mathbb{T}^2}}P_{\leq N} h_{\alpha+\gamma}(t)\rangle_{L^2_{x,y}\times L^2_{x,y}}dsdt\\
&\le\sum_{\alpha,\vert\gamma\vert\le 9} \left \Vert \int_{s\in\mathbb{R}}e^{-is\Delta_{\mathbb{R}\times\mathbb{T}^2}}P_{\leq N}h_\alpha(2\pi\gamma+s)ds\right \Vert_{L^2_{x,y}(\mathbb{R}\times\mathbb{T}^2)}\left \Vert \int_{s\in\mathbb{R}}e^{-is\Delta_{\mathbb{R}\times\mathbb{T}^2}}P_{\leq N}h_{\alpha+\gamma}(s)ds \right\Vert_{L^2_{x,y}(\mathbb{R}\times\mathbb{T}^2)}\\
&\lesssim N^{2\left(\frac{3}{2}-\frac{5}{p}\right)}
\sum_{\alpha} \Vert h_\alpha\Vert_{L^{p^\prime}_{x,y,t}(\mathbb{R}\times\mathbb{T}^2\times [-2\pi,2\pi])}^2.
\end{split}
\end{equation}
This is controlled by the first term in the righthand side of \eqref{Adj}.

\medskip

{\bf Nondiagonal Estimates.}
We can estimate the non diagonal sum as follows using \eqref{SE} proved below:
\begin{equation}\label{NDS}
\begin{split}
\sigma_{\operatorname{non-diag}}&=\sum_{\alpha,\gamma\in\mathbb{Z},\vert\gamma\vert\ge 10}\int_{t}\langle \int_{s}e^{-i(s-t-2\pi\gamma)\Delta_{\mathbb{R}\times\mathbb{T}^2}}P_{\leq N}h_\alpha(s)ds,P_{\leq N}h_{\alpha+\gamma}(t)\rangle_{L^2_{x,y}\times L^2_{x,y}}dt\\
&\lesssim N^{\left(2-\frac{6}{p}\right)}\sum_{\alpha,\gamma\in\mathbb{Z},\vert\gamma\vert\ge 3}\vert\gamma\vert^{\frac{1}{p}-\frac{1}{2}}\Vert h_\alpha\Vert_{L^{p^\prime}_{x,y,t}}\Vert h_{\alpha+\gamma}\Vert_{L^{p^\prime}_{x,y,t}}\\
&\lesssim N^{\left(2-\frac{6}{p}\right)}\Vert \{h_\alpha\}_\alpha\Vert^2_{l^{q^\prime}_\alpha L^{p^\prime}_{x,y,t}(\mathbb{R}\times\mathbb{T}^2\times [-2\pi,2\pi])}, \quad (\frac{2}{q}=\frac{1}{2}-\frac{1}{p})
\end{split}
\end{equation}
where we have used the discrete Hardy-Littlewood-Sobolev inequality in the last step. Together with \eqref{Sigma1}, this finishes the proof of \eqref{Adj} once we prove Lemma \ref{SElem1}.
\end{proof}

\begin{lemma}\label{SElem1}
Suppose that $\gamma\in\mathbb{Z}$ satisfies $\vert\gamma\vert\ge 3$ and that $p>4$. For any $h \in L^{p^\prime}_{x,y,s}(\mathbb{R}\times\mathbb{T}^2\times [-2\pi,2\pi])$, there holds that:
\begin{equation}\label{SE}
\Vert \int_{s\in\mathbb{R}}\chi(s)e^{i(t-s+2\pi\gamma)\Delta_{\mathbb{R}\times\mathbb{T}^2}}P_{\leq N}h(s)ds\Vert_{L^p_{x,y,t}(\mathbb{R}\times\mathbb{T}^2\times [-2\pi,2\pi])}\lesssim \vert\gamma\vert^{\frac{1}{p}-\frac{1}{2}} N^{\left(2-\frac{6}{p}\right)}\Vert h\Vert_{L^{p^\prime}_{x,y,s}(\mathbb{R}\times\mathbb{T}^2\times [-2\pi,2\pi])}.
\end{equation}
\end{lemma}

\begin{proof}
Without loss of generality, we may assume that: 
\begin{equation}\label{BdOnh}
h=\chi(s)P_{\le N} h,\quad \hbox{ and } \Vert h\Vert_{L^{p^\prime}_{x,y,t}(\mathbb{R}\times\mathbb{T}^2\times[-2\pi,2\pi])}=1
\end{equation}
and we define
\begin{equation*}
g(x,y,t)=\int_{s\in\mathbb{R}}e^{i(t-s+2\pi \gamma)\Delta_{\mathbb{R}\times\mathbb{T}^2}}h(x,y,s)ds.
\end{equation*}
We also define the Kernel
\begin{equation}\label{Ker}
\begin{split}
K_N(x,y,t)&=\sum_{k\in \Z^2}\int_{\R_\xi} [\eta^1_{\le N}(\xi)]^2[\eta^1_{\le N}(k_1) \eta^1_{\leq N}(k_2)]^2 e^{i\left[x\cdot\xi+y\cdot k+t(\vert k\vert^2+\vert \xi\vert^2)\right]} d\xi\\
&=\left[\int_{\R_\xi} [\eta^1_{\le N}(\xi)]^2e^{i\left(x\cdot\xi+t\vert\xi\vert^2\right)}d\xi\right]\cdot\left[\sum_{k\in \Z^2}[\eta^1_{\le N}(k_1) \eta^1_{\leq N}(k_2)]^2e^{i\left(y\cdot k+t\vert k\vert^2\right)}\right]\\
&=K_N^{\mathbb{R}}(x,t)\otimes K^{\mathbb{T}^2}_N(y,t).
\end{split}
\end{equation}
and introduce $K_{N,\gamma}(x,y,t):=K_N(x,y,2\pi\gamma+t)$, so that $g(x,y,t)=K_{N,\gamma}\ast_{x,y,t}h$. We have the following straightforward bounds:
\begin{equation}\label{BdOnK}
\begin{split}
\Vert K_{N,\gamma}\Vert_{L^\infty_{x,y,t}}&\lesssim\vert\gamma\vert^{-\frac{1}{2}}N^{2}\\
\Vert \mathcal{F}_{x,y,t}K_{N,\gamma}\Vert_{L^\infty_{\xi,k,\tau}}&\lesssim 1.
\end{split}
\end{equation}
The first estimates follows from stationary phase estimates in $x$ and direct counting in $k$ (see also $\eqref{Weyl}$). The second estimate is direct.

\medskip

For $\alpha$ a dyadic number, we define $g^\alpha(x,y,t)=\alpha^{-1}g(x,y,t)\mathbbm{1}_{\{\alpha/2\le \vert g\vert< \alpha\}}$ which has modulus in $[1/2,1]$. We define\footnote{This is not be confused with $h_\alpha$ defined previous to Lemma \ref{SElem1}.} $h^\beta$ similarly for $\beta \in 2^\Z$. Clearly, $g=\sum_{\alpha}\alpha g^\alpha$. Besides,
\begin{equation}\label{G}
\begin{split}
\Vert g\Vert_{L^p_{x,y,t}}^p&=\langle\vert g\vert^{p-2}g,g\rangle\\
&=\sum_{\alpha,\beta}\alpha^{p-1}\beta\langle \vert g^\alpha\vert^{p-2}g^\alpha,K_{N,\gamma}\ast h^\beta\rangle\\
&=\left[\sum_{\mathcal{S}_1}+\sum_{\mathcal{S}_2}+\sum_{\mathcal{S}_3}\right]\alpha^{p-1}\beta\langle \vert g^\alpha\vert^{p-2}g^\alpha,K_{N,\gamma}\ast h^\beta\rangle=\Sigma_1+\Sigma_2+\Sigma_3,
\end{split}
\end{equation}
where
\begin{equation*}
\begin{split}
\mathcal{S}_1&=\{(\alpha,\beta): C \vert\gamma\vert^{-1/2}N^{2}\le\alpha\beta^{p^\prime-1}\},\\
 \mathcal{S}_2&=\{(\alpha,\beta): \alpha\beta^{p^\prime-1}\le CN\vert\gamma\vert^{-1/2}\},\\
 \mathcal{S}_3&=\{(\alpha,\beta): CN\vert\gamma\vert^{-1/2}\le\alpha\beta^{p^\prime-1}\le C\vert\gamma\vert^{-1/2}N^{2}\}
 \end{split}
 \end{equation*}
 for $C$ a large constant to be fixed later.
In each of the three cases defined above and for any fixed $\alpha, \beta$, we will decompose $K_{N,\gamma}=K^1_{N,\gamma;\alpha,\beta}+K^2_{N,\gamma;\alpha,\beta}$ and estimate the corresponding contributions as follows:
\begin{equation}\label{DecK}
\begin{split}
\langle \vert g^\alpha\vert^{p-2}g^\alpha,K^1_{N,\gamma;\alpha,\beta}\ast h^\beta\rangle&\lesssim \Vert K^1_{N,\gamma;\alpha,\beta}\Vert_{L^\infty_{x,y,t}}\Vert g^\alpha\Vert_{L^1}\Vert h^\beta\Vert_{L^1},\\
\langle \vert g^\alpha\vert^{p-2}g^\alpha,K^2_{N,\gamma;\alpha,\beta}\ast h^\beta\rangle&\lesssim \Vert \mathcal{F}_{x,y,t}K^2_{N,\gamma;\alpha,\beta}\Vert_{L^\infty_{\xi,k,\tau}}\Vert g^{\alpha}\Vert_{L^2}\Vert h^\beta\Vert_{L^2}.
\end{split}
\end{equation}

In what follows, we denote by $S_\alpha:=\hbox{supp}(g_\alpha)$ and $E_\beta:=\hbox{supp}(h_\beta)$. Then, by assumption
\begin{equation*}
\begin{split}
\Vert g\Vert_{L^p_{x,y,t}}^p&\simeq\sum_\alpha\alpha^{p}\vert S_\alpha\vert,\quad \sum_{\beta}\beta^\frac{p}{p-1}\vert E_\beta\vert\simeq1,\quad \Vert g^\alpha\Vert_{L^p_{x,y,t}}\simeq\vert S_\alpha\vert^\frac{1}{p}.
\end{split}
\end{equation*}

\medskip

We will bound $\Sigma_1, \Sigma_2, \Sigma_3$ in \eqref{G} separately:

\medskip
{\bf $\Sigma_1$, the large levels:} $C\vert\gamma\vert^{-1/2}N^{2}\le\alpha\beta^{p^\prime-1}$. In this case, we use the decomposition \eqref{DecK} with $K^2=0$ and the bound \eqref{BdOnK}. This gives
\begin{equation*}
\begin{split}
\Sigma_1\le\sum_{\alpha,\beta;C\vert\gamma\vert^{-1/2}N^{2}\le\alpha\beta^{p^\prime-1}}\alpha^{p-1}\beta\Vert K_{N,\gamma}\Vert_{L^\infty_{x,y,t}}\vert S_\alpha\vert \vert E_\beta\vert\lesssim C^{-1}\sum_{\alpha,\beta}\alpha^p\vert S_\alpha\vert \beta^{p^\prime}\vert E_\beta\vert\lesssim C^{-1}\Vert g\Vert_{L^p_{x,y,t}}^p
\end{split}
\end{equation*}

\medskip

{\bf $\Sigma_2$, the small levels:} $\alpha\beta^{p^\prime-1}\le CN\vert\gamma\vert^{-1/2}$. In this case, we use the decomposition \eqref{DecK} with $K^1=0$ and \eqref{BdOnK}. This gives
\begin{equation*}
\begin{split}
\Sigma_2&\le \sum_{(\alpha,\beta)\in\mathcal{S}_2}\alpha^{p-1}\beta\vert S_\alpha\vert^\frac{1}{2}\vert E_\beta\vert^\frac{1}{2}\le \sum_{\alpha,\beta; \alpha\beta^{p^\prime-1}\le CN\vert\gamma\vert^{-1/2}}\left[\alpha^{p}\vert S_\alpha\vert\right]^\frac{1}{2}\left[\beta^{p^\prime}\vert E_\beta\vert\right]^\frac{1}{2}\left[\alpha\beta^{p^\prime-1}\right]^\frac{p-2}{2}.\\
\end{split}
\end{equation*}
Using Schur's test, we can sum in $\alpha,\beta$ and this gives
\begin{equation*}
\Sigma_2\lesssim C^{\frac{p-2}{2}}\Vert g\Vert_{L^p_{x,y,t}}^\frac{p}{2} N^{\frac{p-2}{2}}\vert\gamma\vert^{-\frac{p-2}{4}}.
\end{equation*}

\medskip

{\bf $\Sigma_3$, the medium levels:} $CN\vert\gamma\vert^{-1/2}\le\alpha\beta^{p^\prime-1}\le C\vert\gamma\vert^{-1/2}N^{2}$. In this case, we will decompose the kernel into two components as in \eqref{DecK} which satisfy
\begin{equation}\label{DecK3}
\begin{split}
\Vert K^1_{N,\gamma;\alpha,\beta}\Vert_{L^\infty_{x,y,t}}&\lesssim C^{-1}\alpha\beta^{p^\prime-1}\\
\Vert \mathcal{F}_{t,x,y}K^2_{N,\gamma;\alpha,\beta}\Vert_{L^\infty_{\xi,k,\tau}}&\lesssim_\varepsilon N^{\varepsilon}C(\alpha\beta^{p^\prime-1})^{-1}\vert\gamma\vert^{-1/2}.
\end{split}
\end{equation}
Assuming this, we finish the proof as follows:
\begin{equation*}
\begin{split}
\Sigma_3&\lesssim C^{-1}\sum_{\mathcal{S}_3}\alpha^p\beta^{p^\prime}\vert S_\alpha\vert \vert E_\beta\vert+N^{\varepsilon }C\sum_{\mathcal{S}_3}\alpha^{p-2}\beta^{2-p^\prime}\vert S_\alpha\vert^\frac{1}{2}\vert E_\beta\vert^\frac{1}{2}\vert\gamma\vert^{-1/2}\\
&\lesssim C^{-1}\Vert g\Vert_{L^p_{x,y,t}}^p+N^{\varepsilon }C\sum_{\mathcal{S}_3}\left[\alpha^p\vert S_\alpha\vert\right]^\frac{1}{2}\left[\beta^{p^\prime}\vert E_\beta\vert\right]^\frac{1}{2}\left[\alpha\beta^{p^\prime-1}\right]^\frac{p-4}{2}\vert\gamma\vert^{-1/2}\\
&\lesssim C^{-1}\Vert g\Vert_{L^p_{x,y,t}}^p+\Vert g\Vert_{L^p_{x,y,t}}^\frac{p}{2}C^{\frac{p-2}{2}}N^{(p-4+\varepsilon)}\vert\gamma\vert^{-\frac{p-2}{4}}.
\end{split}
\end{equation*}
As a result we finally obtain
\begin{equation*}
\Vert g\Vert_{L^p_{x,y,t}}^p=\Sigma_1+\Sigma_2+\Sigma_3\lesssim C^{-1}\Vert g\Vert_{L^p_{x,y,t}}^p+C^\frac{p-2}{2}\Vert g\Vert_{L^p_{x,y,t}}^\frac{p}{2}\left[N^{\frac{p-2}{2}}\vert\gamma\vert^{-\frac{p-2}{4}}+N^{p-4+\varepsilon}|\gamma|^{-\frac{p-2}{4}}\right].
\end{equation*}
We now fix the universal constant $C$ so that it is large enough to absorb the first term on the RHS in the LHS and we get:
\begin{equation}\label{possimp01}
\Vert g\Vert_{L^p_{x,y,t}}^p \lesssim_C \Vert g\Vert_{L^p_{x,y,t}}^\frac{p}{2}\vert\gamma\vert^{-\frac{p-2}{4}}\max(N^{(p-4+\varepsilon)}, N^{\frac{p-2}{2}})\lesssim \Vert g\Vert_{L^p_{x,y,t}}^\frac{p}{2}\vert\gamma\vert^{-\frac{p-2}{4}}N^{p-3}
\end{equation}
if $p>4$. This finishes the proof of \eqref{SE} once we justify the decomposition \eqref{DecK3}. 

\medskip

In order to get the decomposition \eqref{DecK3}, we need a minor/major arc decomposition of the Kernel $K_N^{\T^2}$. We recall the following bound from \cite[Lemma 3.18]{Bo2}: assume that $1\le a\le q\le N$, $(a,q)=1$ and $\vert t/2\pi-a/q\vert\le 1/Nq$, then
\begin{equation*}
\Vert K_N^{\mathbb{T}^2}(\cdot,t)\Vert_{L^\infty_y(\mathbb{T}^2)}\lesssim \frac{N^{2}}{q(1+N^{2}\vert t/2\pi-a/q\vert)}.
\end{equation*}
From \eqref{Ker} and the explicit formula in $\mathbb{R}$, we directly get that, under the above conditions, the first estimate in \eqref{BdOnK} can be further refined into:
\begin{equation}\label{Weyl}
\Vert K_{N,\gamma}(t)\Vert_{L^\infty_{x,y}(\mathbb{R}\times\mathbb{T}^2)}\lesssim \vert\gamma\vert^{-1/2}\frac{N^{2}}{q(1+N^{2}\vert t/2\pi-a/q\vert)}.
\end{equation}

For $L \in 2^\Z$, we will need the following fine-scale time-cutoff functions $\phi_{\geq L}(\cdot):=\eta^1(L\cdot)$ and $\phi_{L}=\phi_{\geq L}-\phi_{\geq 2L}$. Let $n$ be such that $N=2^n$ and for $0\le k\le n-10$, $j+k\le n$, define the time projectors 
\begin{equation}\label{Proj}
\begin{split}
p_{k,j}(s):=&\sum_{2^k \leq q \leq 2^{k+1}}\sum_{(a,q)=1, 1\leq a\leq q} \phi_{N2^{k+j+10}}(s/2\pi-a/q)\quad \text{ for } j\leq n-k-1,\\
p_{k, n-k}(s):=& \sum_{2^k \leq q\leq 2^{k+1}}\sum_{(a,q)=1, 1 \leq a \leq q} \phi_{\geq 2^{10}N^2 }(s/2\pi-a/q)
\end{split}
\end{equation}
and let $e$ be such that
\begin{equation*}
1=\sum_{k=0}^{n-10}\sum_{j;j+k\le n}p_{k,j}(t)+e(t).
\end{equation*}
We now define
\begin{equation}\label{DefK1}
K^1_{N,\gamma;\alpha,\beta}(x,y,t):=\left[e(t)+\sum_{\mathcal{S}_M}p_{k,j}(t)+\sum_{\mathcal{S}_m}\rho_{j,k}p_{k,0}(t)\right]K_{N,\gamma}(x,y,t).
\end{equation}
where we used the following notations:
\begin{equation}\label{def of rho_kj}
\rho_{k,j}:=2^{-j} \;\text{ if } j<n-k \;\;\text{ and } \rho_{k, n-k}:=2^{-n+k+1},
\end{equation}
\begin{equation*}
\begin{split}
\mathcal{S}_M&:=\{(j,k):0\le k\le n-10, 0\le j+k\le n;\,\, 2^{n+j}\le C^{-1}\alpha\beta^{p^\prime-1}|\gamma|^{1/2}\}\\
\mathcal{S}_m&:=\{(j,k):0\le k\le n-10, 0\le j+k\le n;\,\, 2^{n+j}\ge C^{-1}\alpha\beta^{p^\prime-1}|\gamma|^{1/2}\}
\end{split}
\end{equation*}
the defining sets of the major and minor arcs. The $\rho_{j,k}$ are introduced to cancel the large zero frequency mode of $K^2_{n,\gamma;\alpha,\beta}(x,y,t)$ (cf. $\eqref{def of K^2}$ and $\eqref{estimate on b_m}$). 

By direct inspection, we see that the major arcs are disjoint:
\begin{equation*}
\phi_{N2^{k+j+10}}(s/2\pi-a/q)\phi_{N2^{k^\prime+j^\prime+10}}(s/2\pi-a^\prime/q^\prime)\ne 0\Rightarrow k=k^\prime,\,(a,q)=(a^\prime,q^\prime),\,\vert j-j^\prime\vert\le 3.
\end{equation*}
In particular,
\begin{equation*}
0\le p_{k,j}\le 1\quad\text{and}\quad p_{k,j}(t)p_{k^\prime,j^\prime}(t)\ne 0\Rightarrow k=k^\prime;\,\vert j-j^\prime\vert\le 3.
\end{equation*}
If $e(t)\ne 0$, then either $\vert t/2\pi-a/q\vert\ge 2^{-n}/q$ or $q\gtrsim 2^n$. In both cases, from \eqref{Weyl}, we see that
\begin{equation*}
\Vert e(t)K_{N,\gamma}\Vert_{L^\infty_{x,y,t}}\lesssim 2^n\vert\gamma\vert^{-1/2}.
\end{equation*}
Invoking \eqref{Weyl} again, we deduce that
\begin{equation*}
\Vert K_{N,\gamma}(x,t)\sum_{\mathcal{S}_M}p_{k,j}(t)\Vert_{L^\infty_{x,y,t}}\lesssim \sup_{\mathcal{S}_M}\Vert p_{k,j}K_{N,\gamma}\Vert_{L^\infty_{x,y,t}}\lesssim C^{-1}\alpha\beta^{p^\prime-1}.
\end{equation*}
Similarly using the fact that $2^n\vert\gamma\vert^{-1/2}\le\alpha\beta^{p^\prime-1}$, we get that
\begin{equation*}
\Vert K_{N,\gamma}(x,t)\sum_{\mathcal{S}_m}2^{-j}p_{k,0}(t)\Vert_{L^\infty_{x,y,t}}\lesssim \sup_{k}\sum_{j:(j,k)\in\mathcal{S}_m}2^{-j-k}2^{2n}\vert\gamma\vert^{-1/2}\lesssim C^{-1}\alpha\beta^{p'-1}.
\end{equation*}
This proves that $K^1_{n,\gamma;\alpha,\beta}$ satisfies the right bound in \eqref{DecK3}.

\medskip

It remains to prove that
\begin{equation}\label{def of K^2}
K^2_{N,\gamma;\alpha,\beta}(x,y,t)=\chi(t)\sum_{\mathcal{S}_m}\left[p_{k,j}(t)-\rho_{k,j}p_{k,0}(t)\right]K_{N,\gamma}(x,y,t)
\end{equation}
satisfies the appropriate estimates in \eqref{DecK3}. Here we need the following lemma from \cite{IoPa}:

\begin{lemma}\label{decomposition}
Assume $\kappa>0$, $Q,M\in\mathbb{N}^\ast$, $M\geq 8Q$, $S\subseteq\{1,\ldots,Q\}$, and $\eta:\mathbb{R}\to[0,1]$ is a smooth function supported in $[-2,2]$. Then, for any $t\in\mathbb{R}$,
\begin{equation}\label{bo1}
\sum_{q\in S,\,a\in\mathbb{Z}\,(a,q)=1}\eta(MQ(t-a/q))=\sum_{m\in\mathbb{Z}}(MQ)^{-1}\widehat{\eta}(2\pi m/(MQ))c_m e^{2\pi imt},
\end{equation}
and the coefficients $c_m$ have the properties
\begin{equation}\label{bo2}
c_m=\sum_{q\in S,\,0\le a\le q-1,\,(a,q)=1}e^{-2\pi im\cdot a/q},\qquad\sup_{m\in\mathbb{Z}}|c_m|\leq 4Q^2,\qquad |c_m|\leq C_\kappa d(m,Q)Q^{1+\kappa},
\end{equation}
where $d(m,Q)$ denotes the number of divisors of $m$ less than or equal to $Q$.
\end{lemma}

Using this lemma and because of our definition of $\rho_{k,j}$ in $\eqref{def of rho_kj}$, we can write
\begin{equation*}
p_{k,j}(t)-\rho_{k,j}p_{k,0}(t)=\sum_{m\in\mathbb{Z}}c_me^{i2\pi mt}2^{-j-k-n-10}\left[\hat{\phi_1}(2^{-j-k-n-10}2\pi m)-\hat{\phi_1}(2^{-k-n-10}2\pi m)\right]
\end{equation*}
if $k < n-k$, and 
\begin{equation*}
p_{k,n-k}(t)-\rho_{k,n-k} p_{k,0}(t)=\sum_{m\in\mathbb{Z}}c_me^{i2\pi mt}2^{-2n-10}\left[\widehat{\phi_{\geq 1}}(2^{-2n-10}2\pi m)-2\hat{\phi_1}(2^{-k-n-10}2\pi m)\right]
\end{equation*}

In either case, we have that for $0\leq k \leq n-10, \;0\leq j \leq n-k$:

\begin{equation*}
p_{k,j}(t)-\rho_{k,j}p_{k,0}(t)=\sum_{m\in\mathbb{Z}}b_me^{i2\pi mt}
\end{equation*}
with
\begin{equation}\label{estimate on b_m}
b_0=0,\quad \vert b_m\vert\le C_\varepsilon d(m,2^{k+1})2^{-j-n+\varepsilon k}(1+2^{-j-k-n}\vert m\vert)^{-10},
\end{equation}
where $d(m,Q)$ denotes the number of divisors of $m$ less than $Q$. We remark that $\rho_{k,j}$ were specifically introduced to make $b_0=0$ (recall that $\widehat{\phi_{\geq 1}}(0)=2\widehat{\phi_1}(0)$).

Using the fact that $d(m,Q)\lesssim C_\varepsilon\vert m\vert^\varepsilon$, we finally get
\begin{equation*}
\begin{split}
\Vert \mathcal{F}_{x,y,t}K^2_{N,\gamma;\alpha,\beta}\Vert_{L^\infty_{\xi,z,\tau}}
&\lesssim\sum_{(k,j) \in \mathcal{S}_m}
\sup_{\xi,z,\tau}\left\vert\int_{\mathbb{R}}\chi(t)\left(p_{k,j}(t)-2^{-j}p_{k,0}(t)\right)e^{it\left[\vert\xi\vert^2+\vert z\vert^2-\tau-\gamma\right]}dt\right\vert\\
&\le \sum_{(k,j) \in \mathcal{S}_m}\sum_{m\in\mathbb{Z}} \vert b_m\vert\lesssim C_\varepsilon \sum_{(k,j) \in \mathcal{S}_m} 2^{\varepsilon n}2^{-j-n}.
\end{split}
\end{equation*}
which gives second bound in \eqref{DecK3}. This finishes the proof.
\end{proof}

\begin{remark}

One can sharpen the estimate \eqref{Stric1} to allow for different smoothness in both directions, namely, when $1\le M\le N$,
\begin{equation}\label{Strich M refined U^p lemma}
\Vert P_{\le N}P^x_{\le M}e^{it\Delta_{\mathbb{R}\times\mathbb{T}^2}}u_0\Vert_{l^q_\gamma L^{p}_{x,y,t}(\mathbb{R}\times\mathbb{T}^2\times [2\pi\gamma,2\pi(\gamma+1)])}\lesssim \left(\frac{M}{N}\right)^{\theta}N^{\left(\frac{3}{2}-\frac{5}{p}\right)}\Vert u_0\Vert_{L^2_{x,y}(\mathbb{R}\times\mathbb{T}^2)}.
\end{equation}
where $\theta=1/2-2/p$.
\end{remark}
This follows directly from \eqref{Adj} after one refines the diagonal contribution \eqref{DiagDual} with
\begin{equation*}
\Vert \int_{\mathbb{R}}e^{-is\Delta_{\mathbb{R}\times\mathbb{T}^2}}P_{\le N}P_{\le M}^xh(s)ds\Vert_{L^2_{x,y}(\mathbb{R}\times\mathbb{T}^2)}\lesssim \left(\frac{M}{N}\right)^{\frac{1}{2}-\frac{1}{p}}N^{\frac{3}{2}-\frac{5}{p}}\Vert h\Vert_{L^{p^\prime}_{x,y,t}(\mathbb{R}\times\mathbb{T}^2\times [-2\pi,2\pi])}.
\end{equation*}
The claim follows\footnote{Taking into account the better derivative gain in the non-diagonal part (by optimizing the power of $N$ in \eqref{possimp01}), one can improve $\theta$ in \eqref{Strich M refined U^p lemma} into $\theta=1-\frac{4}{p}$ if $4<p<6$ and $1/2-1/p+\varepsilon$ if $p\geq 6$.}.


\section{Local well-posedness and small-data scattering}\label{Section-LWP}

We start by translating the Strichartz estimate in Theorem \ref{Striclem} into an embedding theorem for $U^p_\Delta$ spaces:

\begin{lemma}\label{U^p lemma} For $p>4$ and $q$ as in Theorem \ref{Striclem}, the following estimate holds for any time interval $I\subset \R$ and every cube $Q\subset \R^3$ of size N:

\begin{equation}\label{U^p estimate}
\|\mathbbm{1}_I(t)P_Q u\|_{l_\gamma^q L^p_{x,y,t}( \R\times\T^2\times I_\gamma)} \lesssim N^{\frac{3}{2}-\frac{5}{p}}\|u\|_{U_\Delta^{\min(p,q)}(I; L^2_{x,y}(\R\times \T^2))}
\end{equation}

\end{lemma}
The proof of this lemma is slightly tedious but straightforward. We omit the details.

\endproof

\subsection{Contraction argument} The following quantity will be of use in the local existence theory:
\begin{equation}\label{def of Z'}
\|u\|_{Z'(I)}=\|u\|_{X^1(I)}^{1/4}\|u\|_{Z(I)}^{3/4}
\end{equation}

The local well-posedness theory will follow from the following two lemmas:

\begin{lemma}\label{half nl lemma}
Suppose that $u_i=P_{N_i}u_i$ for $i=1,2,3$ satisfying $N_1 \geq N_2 \geq N_3$. There exists $\delta>0$ such that the following estimate holds for any interval $I \subset \R$:

\begin{equation}\label{half nonlinear}
\|u_1u_2 u_3\|_{L^2_{x,y,t}(\R\times \T^2\times I)} \lesssim \left(\frac{N_3}{N_1}+\frac{1}{N_2}\right)^{\delta} \|u_1\|_{Y^0(I)} \|u_2\|_{Z'(I)}\|u_3\|_{Z'(I)}.
\end{equation}
\end{lemma}

\proof It is sufficient to study the case $I=\R$. The result is a consequence of the following two estimates:
\begin{equation}\label{extra refinement}
\|u_1u_2 u_3\|_{L^2_{x,y,t}(\R\times \T^2\times\mathbb{R})} \lesssim N_2N_3\left(\frac{N_3}{N_1}+\frac{1}{N_2}\right)^{\delta_0} \|u_1\|_{Y^0(\R)} \|u_2\|_{Y^0(\R)}\|u_3\|_{Y^0(\R)}.
\end{equation}
and 
\begin{equation}\label{trilinear Z}
\|u_1u_2 u_3\|_{L^2_{x,y,t}(\R\times \T^2\times\mathbb{R})} \lesssim \|u_1\|_{Y^0(\R)} \|u_2\|_{Z(\R)}\|u_3\|_{Z(\R)}.
\end{equation}
Estimate $\eqref{extra refinement}$ follows as in Proposition 3.5 of \cite{HeTaTz} or Proposition 2.8 of \cite{HeTaTz2}. As for $\eqref{trilinear Z}$, we first notice that, by orthogonality considerations, we may replace $u_1$ by $P_C u_1$ where $C$ is a cube of dimension $N_2$. This allows to estimate:

\begin{align*}
\|(P_C u_1)u_2 u_3\|_{L^2_{x,y,t}(\R\times \T^2\times\mathbb{R})}\lesssim&  \|P_C u_1\|_{l^{\frac{36}{5}}_\gamma L_{x,y,t}^{\frac{9}{2}}} \|u_2\|_{l^{\frac{36}{5}}_\gamma L_{x,y,t}^{9/2}}\|u_3\|_{l^{\frac{9}{2}}_\gamma L_{x,y,t}^{18}}\\
\lesssim& N_2^{\frac{7}{18}}\|P_C u_1\|_{Y^0} \|u_2\|_{l^{\frac{36}{5}}_\gamma L_{x,y,t}^{9/2}}\|u_3\|_{l^{\frac{9}{2}}_\gamma L_{x,y,t}^{18}}\\
\lesssim&  (\frac{N_3}{N_2})^{2/9} \|P_C u_1\|_{Y^0}\|u_2\|_{Z(\R)}\|u_3\|_{Z(\R)}
\end{align*}
which is acceptable.

\endproof
 
 As a result of the above lemma, we have:

\begin{lemma}\label{nl estimate lemma}
For $u_i \in X^1(I)$, $i=1,\ldots, 5$, there holds that
\begin{equation}\label{nl refined}
\|\sum_{K\geq 1} P_{K} u_1\prod_{i=2}^5 P_{\leq CK} \tilde u_i\|_{N(I)} \lesssim_C \|u_1\|_{X^1(I)} \prod_{i=2}^5\|u_i\|_{Z'(I)}.
\end{equation}
and in particular,
\begin{equation}\label{nl estimate}
\|\prod_{i=1}^5 \tilde u_i\|_{N(I)} \lesssim \sum_{j=1}^5 \|u_j\|_{X^1(I)} \prod_{i\neq j}\|u_i\|_{Z'(I)}
\end{equation}
where $\tilde u_i$ is either $u_i$ or $\overline{u_i}$.
\end{lemma}

\begin{proof}
The proof is a fairly standard consequence of Lemma \ref{half nl lemma} (cf. Lemma 3.2 of \cite{IoPa}).
\end{proof}

We are now ready to prove the local existence theory for large data and global existence and scattering theory for small data. The proof of the following statement is very similar from the proofs in \cite[Section 3]{IoPa} given Lemma \ref{nl estimate lemma} and is omitted.

\begin{proposition}[Local Well-posedness]\label{LWP}
Let $E >0$ and suppose that $\|u_0\|_{H^1(\R \times \T^2)}\leq E$. There exists $\delta_0=\delta_0(E)$ $\approx \min(1, E^{-1})$ such that if 
$$
\|e^{it\Delta_{\mathbb{R}\times\mathbb{T}^2}}u_0\|_{Z(I)} \leq \delta\leq \delta_0
$$
on some interval $I\subset \R$ containing $0$, then there exists a unique strong solution $u\in X^1_c(I)$ of $\eqref{NLS}$ satisfying $u(0)=u_0$ and 
\begin{equation}\label{SmalldataCCL}
\|u(t)-e^{it\Delta_{\mathbb{R}\times\mathbb{T}^2}}u_0\|_{X^1(I)}\lesssim E^2\delta^3.
\end{equation}
Moreover, the quantities $E(u)$ and $M(u)$ defined in \eqref{mass and energy} are conserved on $I$.
\end{proposition}

Observe that if $u\in X^1_c(\R)$, then $u$ scatters as $t\to\pm\infty$ as in \eqref{Scat}. Also, if $E$ is small enough, $I$ can be taken to be $\R$ which proves Theorem \ref{Thm1}.

\begin{lemma}[Controlling norm]\label{Z controls}
Suppose that $u \in X_{c,\operatorname{loc}}^1(I)$ is a strong solution on an interval $I\subset \mathbb{R}$ satisfying:
\begin{equation}\label{CondForGE}
\|u\|_{Z(I)}<+\infty.
\end{equation}
\begin{enumerate}
\item If $I$ is finite, then $u$ can be extended as a strong solution \emph{in $ X^1_c(I')$} on a strictly larger interval $I'$, $I\subsetneq I^\prime\subset\mathbb{R}$. In particular, if $u$ blows up in finite time, then the $Z$ norm of $u$ has to blow up.

\item If $I$ is infinite, then $u \in X_c^1(I)$.
\end{enumerate}
\end{lemma}

Finally, we end this section with the following stability result:

\begin{proposition}\label{stability}
Let $I\subset\mathbb{R}$ be an interval, and let $\tilde u \in X^1(I)$ solve the approximate equation:
\begin{equation}\label{approx}
(i\partial_t  +\Delta_{\mathbb{R}\times\mathbb{T}^2}) \tilde u=\rho |\tilde u|^4 \tilde u +e
\end{equation}
where $\rho$ is a constant in $[0,1]$. Assume that the following boundedness estimates holds:
\begin{equation}\label{M bound}
\|\tilde u\|_{Z(I)}+\|\tilde u\|_{L_t^\infty(I, H^1(\R \times \T^2))}\leq M.
\end{equation}
There exists $\epsilon_0=\epsilon_0(M)\in (0,1]$ such that if for some $t_0\in I$:
\begin{equation}\label{smallness}
\|\tilde u(t_0)-u_0\|_{H^1(\R\times \T^2)}+\|e\|_{N(I)}\leq \epsilon <\epsilon_0,
\end{equation}
then there exists a solution $u(t)$ to the exact equation:
\begin{equation}\label{rho NLS}
(i\partial_t  +\Delta_{\mathbb{R}\times\mathbb{T}^2})u=\rho |u|^4u
\end{equation}
with initial data $u_0$ which satisfies
\begin{equation}\label{solution estimates}
\begin{split}
\|u\|_{X^1(I)}+\|\tilde u\|_{X^1(I)} \leq& C(M)\\
\|u-\tilde u\|_{X^1(I)}\leq& C(M) \epsilon.
\end{split}
\end{equation}
\end{proposition}


\section{Nonlinear analysis of profiles}\label{Sec-DesPro}

In this section, we describe and analyze the main profiles that appear in our linear and nonlinear profile decomposition.

\subsection{Euclidean profiles}\label{Euclidean profile analysis}

The Euclidean profiles define a regime where we can compare solutions of quintic NLS on $\R^3$ with those on $\R\times \T^2$. Of course, this analysis is only possible in the small-scale limit, i.e. for solutions that concentrate at a point. Our arguments below follow closely those in \cite{IoPaSt, IoPa}.

We fix a spherically-symmetric function $\eta\in C^\infty_0(\mathbb{R}^3)$ supported in the ball of radius $2$ and equal to $1$ in the ball of radius $1$. Given $\phi\in \dot{H}^1(\mathbb{R}^3)$ and a real number $N\geq 1$ we define
\begin{equation}\label{rescaled}
\begin{split}
\phi_N\in H^1(\mathbb{R}^3),\qquad &\phi_N(x)=N^\frac{1}{2}\eta(N^\frac{1}{2}x)\phi(Nx),\\
f_{N}\in H^1(\mathbb{R}\times\T^2),\qquad &f_{N}(y)=\phi_N(\Psi^{-1}(y)),
\end{split}
\end{equation}
where $\Psi:\{x\in\mathbb{R}^3:|x|<1\}\to O_0\subseteq \mathbb{R}\times\T^2$, $\Psi(x)=x$. Thus $\phi_N$ is an $\dot{H}^1$-invariant rescaling of a smooth compact truncation $\phi$, and $f_{N}$ is the function obtained by transferring $\phi_N$ to a neighborhood of $0$ in $\mathbb{R}\times\T^2$. Notice that 

$$\|f_N\|_{H^1(\R\times \T^2)}\lesssim \|\phi\|_{\dot H^1(\R^3)}.$$

We define also
\begin{equation*}
E_{\mathbb{R}^3}(\phi)=\frac{1}{2}\int_{\mathbb{R}^3}|\nabla_{\R^3}\phi|^2\,dx+\frac{1}{6}\int_{\mathbb{R}^3}|\phi|^6\,dx.
\end{equation*}

We will use the main theorem of \cite[Theorem 1.1]{CKSTTcrit} (see also \cite{KiVi}), in the following form.

\begin{theorem}\label{MainThmEucl}
Assume $\psi\in\dot{H}^1(\mathbb{R}^3)$. Then there is a unique global solution $v\in C(\mathbb{R}:\dot{H}^1(\mathbb{R}^3))$ of the initial-value problem
\begin{equation}\label{clo3}
(i\partial_t+\Delta_{\R^3})v=v|v|^4,\qquad v(0)=\psi,
\end{equation}
and
\begin{equation}\label{clo4}
\|\,|\nabla_{\R^3} v|\,\|_{(L^\infty_tL^2_x\cap L^2_tL^6_x)(\mathbb{R}^3\times\mathbb{R})}\leq \widetilde{C}(E_{\mathbb{R}^3}(\psi)).
\end{equation}
Moreover this solution scatters in the sense that there exists $\psi^{\pm\infty}\in\dot{H}^1(\mathbb{R}^3)$ such that
\begin{equation}\label{EScat}
\Vert v(t)-e^{it\Delta_{\mathbb{R}^3}}\psi^{\pm\infty}\Vert_{\dot{H}^1(\mathbb{R}^3)}\to 0
\end{equation}
as $t\to\pm\infty$. Besides, if $\psi\in H^5(\mathbb{R}^3)$ then $v\in C(\mathbb{R}:H^5(\mathbb{R}^3))$ and
\begin{equation*}
\sup_{t\in\mathbb{R}}\|v(t)\|_{H^5(\mathbb{R}^3)}\lesssim_{\|\psi\|_{H^5(\mathbb{R}^3)}}1.
\end{equation*}
\end{theorem}

Our first result in this section is the following lemma:

\begin{lemma}\label{step1}
Assume $\phi\in\dot{H}^1(\mathbb{R}^3)$, $T_0\in(0,\infty)$, and $\rho\in\{0,1\}$ are given, and define $f_{N}$ as in \eqref{rescaled}. Then the following conclusions hold:

(i) There is $N_0=N_0(\phi,T_0)$ sufficiently large such that for any $N\geq N_0$ there is a unique solution $U_{N}\in C((-T_0N^{-2},T_0N^{-2});H^1(\R \times \T^2))$ of the initial-value problem
\begin{equation}\label{clo5}
(i\partial_t+\Delta)U_N=\rho U_N|U_N|^4,\qquad U_N(0)=f_N.
\end{equation}
Moreover, for any $N\geq N_0$,
\begin{equation}\label{clo6}
\|U_N\|_{X^1(-T_0N^{-2},T_0N^{-2})}\lesssim_{E_{\mathbb{R}^3}(\phi)}1.
\end{equation}

(ii) Assume $\varepsilon_1\in(0,1]$ is sufficiently small (depending only on $E_{\mathbb{R}^3}(\phi)$), $\phi'\in H^5(\mathbb{R}^3)$, and $\|\phi-\phi'\|_{\dot{H}^1(\mathbb{R}^3)}\leq\varepsilon_1$. Let $v'\in C(\mathbb{R}:H^5)$ denote the solution of the initial-value problem
\begin{equation*}
(i\partial_t+\Delta_{\R^3})v'=\rho v'|v'|^4,\qquad v'(0)=\phi'.
\end{equation*}
For $R\geq 1$, and $N\geq 10R$ we define
\begin{equation}\label{clo9}
\begin{split}
&v'_R(x,t)=\eta(x/R)v'(x,t),\qquad\,\,\qquad (x,t)\in\mathbb{R}^3\times(-T_0,T_0),\\
&v'_{R,N}(x,t)=N^\frac{1}{2}v'_R(Nx,N^2t),\qquad\quad\,(x,t)\in\mathbb{R}^3\times(-T_0N^{-2},T_0N^{-2}),\\
&V_{R,N}(y,t)=v'_{R,N}(\Psi^{-1}(y),t)\qquad\quad\,\, (y,t)\in\R\times \T^2\times(-T_0N^{-2},T_0N^{-2}).
\end{split}
\end{equation}
Then there is $R_0\geq 1$ (depending on $T_0$ and $\phi'$ and $\varepsilon_1$) such that, for any $R\geq R_0$ and $N\geq 10R$,
\begin{equation}\label{clo18}
\limsup_{N\to\infty}\|U_N-V_{R,N}\|_{X^1(-T_0N^{-2},T_0N^{-2})}\lesssim_{E_{\mathbb{R}^3}(\phi)}\varepsilon_1.
\end{equation}
\end{lemma}

\begin{proof}
The proof of Lemma \ref{step1} is very similar to the proof of \cite[Lemma 4.2]{IoPaSt}. 

\end{proof}

To understand linear and nonlinear evolutions beyond the Euclidean window we need an additional extinction lemma:

\begin{lemma}\label{Extinction}
Suppose that $\phi\in \dot{H}^1(\mathbb{R}^3)$, $\epsilon>0$, and $I \subset \R$ is an interval. Assume that
\begin{equation}\label{small stric}
\Vert \phi\Vert_{\dot{H}^1(\mathbb{R}^3)}\le 1,\qquad\|\nabla_x e^{it\Delta}\phi\|_{L_t^2 L^6_{x}(I\times \R^3)}\leq \epsilon.
\end{equation} 
For $N \geq 1$, we define as before:
$$
Q_N \phi =\eta(N^{-1/2}x)\phi(x),\;\; \phi_N=N^{1/2}Q_N(Nx)=N^{1/2}\eta(N^{1/2}x)\phi(Nx), \;\;\; f_N(y)=\phi_N(\Psi^{-1}(y)).
$$
Then there exists $N_0=N_0(\phi, \epsilon)$ such that for any $N\geq N_0$,
$$
\| e^{it\Delta} f_N\|_{Z(N^{-2}I)}\lesssim \epsilon^\frac{1}{6}
$$

\end{lemma}

\begin{proof} 
The proof is fairly similar to that of Lemma of 4.3  \cite{IoPa}, so we omit the details for brevity.

\end{proof}

We are now ready to describe the nonlinear solutions of \eqref{NLS} corresponding to data concentrating at a point. Let $\widetilde{\mathcal{F}}_e$ denote the set of renormalized Euclidean frames\footnote{We will later consider a slightly more general class of frames, called {\it{Euclidean frames}}, see Definition \ref{DefPro}. However, for all practical purposes (cf. Lemma \ref{Eq and orth profiles}), it suffices to prove Proposition \ref{GEForEP} under the stronger assumption that $\mathcal{O}$ is a renormalized Euclidean frame.} 
\begin{equation}\label{renframe}
\begin{split}
\widetilde{\mathcal{F}}_e:=\{(N_k,t_k,x_k)_{k\geq 1}:&\,N_k\in[1,\infty),\,x_k\in\mathbb{R}\times\T^2,\,N_k\to\infty,\\
&\text{ and either }t_k=0 \text{ for any }k\geq 1\text{ or }\lim_{k\to\infty}N_k^2|t_k|=\infty\}.
\end{split}
\end{equation}

Given $f \in L^2_{x,y}(\R\times \T^2)$, $t_0\in \R$, and $x_0 \in \R\times\T^2$, we define:
\begin{equation}\label{DefPi}
\pi_{x_0} f= f(x-x_0), \quad \Pi_{(t_0,x_0)} f=(e^{-it_0\Delta_{\mathbb{R}\times\mathbb{T}^2}}f)(x-x_0)=\pi_{x_0} e^{-it_0\Delta_{\mathbb{R}\times\mathbb{T}^2}}f.
\end{equation}
Also for $\phi \in \dot H^1(\R^3)$ and $N\geq 1$, we denote the function obtained in \eqref{rescaled} by:
\begin{equation}\label{DefTNE}
T^e_N \phi:= N^{1/2} \widetilde \phi(N \Psi^{-1}(x))\quad \hbox{where}\quad \widetilde \phi(y):=\eta(y/N^{1/2})\phi(y), 
\end{equation}
and observe as before that $T^e_N: \dot H^1(\R^3) \to H^1(\R\times \T^2)$ with $\|T^e_N \phi\|_{H^1(\R\times \T^2)} \lesssim \|\phi\|_{\dot H^1(\R^3)}$.

\medskip

\begin{proposition}\label{GEForEP}
Assume that $\mathcal{O}=(N_k,t_k,x_k)_k\in\widetilde{\mathcal{F}}_e$, $\phi\in\dot{H}^1(\mathbb{R}^3)$, and let $U_k(0)=\Pi_{t_k,x_k}(T^e_{N_k}\phi)$: 

(i) For $k$ large enough (depending only on $\phi$, $\mathcal{O}$), there is a nonlinear solution $U_k\in X^1_c(\R)$ of the equation \eqref{NLS} satisfying:
\begin{equation}\label{ControlOnZNormForEP}
\Vert U_k\Vert_{X^1(\R)}\lesssim_{E_{\mathbb{R}^3}(\phi)}1.
\end{equation}

(ii) There exists a Euclidean solution $u\in C(\mathbb{R}:\dot{H}^1(\mathbb{R}^3))$ of
\begin{equation}\label{EEq}
\left(i\partial_t+\Delta_{\R^3}\right)u=\vert u\vert^4u
\end{equation}
with scattering data $\phi^{\pm\infty}$ defined as in \eqref{EScat} such that the following holds, up to a subsequence:
for any $\varepsilon>0$, there exists $T(\phi,\varepsilon)$ such that for all $T\ge T(\phi,\varepsilon)$ there exists $R(\phi,\varepsilon,T)$ such that for all $R\ge R(\phi,\varepsilon,T)$, there holds that
\begin{equation}\label{ProxyEuclHyp}
\Vert U_k-\tilde{u}_k\Vert_{X^1(\{\vert t-t_k\vert\le TN_k^{-2}\})}\le\varepsilon,
\end{equation}
for $k$ large enough, where
\begin{equation*}
(\pi_{-x_k}\tilde{u}_k)(x,t)=N_k^\frac{1}{2}\eta(N_k\Psi^{-1}(x)/R)u(N_k\Psi^{-1}(x),N_k^2(t-t_k)).
\end{equation*}
In addition, up to a subsequence,
\begin{equation}\label{ScatEuclSol}
\Vert U_k(t)-\Pi_{t_k-t,x_k}T^e_{N_k}\phi^{\pm\infty}\Vert_{X^1(\{(t-t_k)\geq TN_k^{-2}\})}\le \varepsilon,
\end{equation}
for $k$ large enough (depending on $\phi,\varepsilon,T,R$).
\end{proposition}

\begin{proof} cf. Proposition 4.4 of \cite{IoPa} for a similar argument.
%
%
%
\end{proof}

\subsection{Large Scale Profiles}\label{LS profile analysis}

In this subsection, we analyze the large-scale profiles that appear in the profile decomposition in Section \ref{Sec-ProfileDec}. We need the following notation: given $\psi\in H^{0,1}(\mathbb{R}\times\mathbb{T}^2)$ and $M\le 1$, we define the large-scale rescaling
\begin{equation}\label{LSrescaled}
T_M^{ls}\psi(x,y):=M^\frac{1}{2}\widetilde{\psi^*}(Mx,y)\qquad\text{ where }\qquad \widetilde \psi^*(x,y)=P^x_{\le M^{-1/100}}\psi(x,y).
\end{equation}
It's easy to see that,
\begin{equation*}
T_M^{ls}:H^{0,1}(\mathbb{R}\times\mathbb{T}^2)\to H^1(\mathbb{R}\times\mathbb{T}^2)\text{ is a linear operator with }\|T^{ls}_M\psi\|_{H^1(\R\times\mathbb{T}^2)}\lesssim \|\psi\|_{H^{0,1}(\mathbb{R}\times\mathbb{T}^2)}.
\end{equation*}

The purpose of this section is to understand the behavior of nonlinear solutions $u_k$ of \eqref{NLS} with initial data as above and variations thereof. As we will see these are tightly connected to solutions of the quintic resonant system defined in \eqref{SS} that we recall here for convenience.

\subsubsection{The quintic resonant system}
We consider the quintic resonant system
\begin{equation*}
\begin{split}
\left(i\partial_t+\partial_{xx}\right)u_j=&\sum_{(p_1,p_2,p_3,p_4,p_5)\in\mathcal{R}(j)}u_{p_1}\overline{u}_{p_2}u_{p_3}\overline{u}_{p_4}u_{p_5},\quad j\in\mathbb{Z}^2\\
\mathcal{R}(j)=&\{{\vec p}:=(p_1,p_2,p_3,p_4,p_5)\in(\mathbb{Z}^2)^5: \,\, p_1-p_2+p_3-p_4+p_5-j=0,\\
&\,\,\vert p_1\vert^2-\vert p_2\vert^2+\vert p_3\vert^2-\vert p_4\vert^2+\vert p_5\vert^2-\vert j\vert^2=0\}.
\end{split}
\end{equation*}
with initial data ${\vec u}(0)=\{u_j(0)\}_{j \in \Z^2}\in h^1L^2$, where $h^1L^2$ is defined in \eqref{DefhH}.
In light of the conservation laws \eqref{g energy} proved in the Appendix, the following energy:
\begin{equation}\label{Els}
E_{ls}[\vec u]:=\sum_{p \in \Z^2}[1+|p|^2]\|u_p\|^2_{L_x^2(\mathbb{R})}
\end{equation}
is conserved and hence so is the $h^1 L^2$ norm of any solution of $\eqref{SS}$.
 
We remark that in the special case when only one component of the initial data is nontrivial: $u_j(0)=0,j\ne 0$, this is propagated by the nonlinear flow and \eqref{SS} reduces to the scalar defocusing mass-critical Schr\"odinger equation,
\begin{equation*}
\left(i\partial_t+\partial_{xx}\right)u_0=\vert u_0\vert^4u_0,\quad E_{ls}(u_0)=\Vert u_0\Vert_{L^2_x(\mathbb{R})}^2.
\end{equation*}

\medskip

Based on the local and small data theory that we develop in the appendix and the large-data scattering result in the scalar case \cite{Dod},
we conjecture the following global existence theorem about solutions to \eqref{SS}:
\begin{conjecture}\label{LSConj}
Any initial data ${\vec u_0}$ of finite $E_{ls}$ energy leads to a global solution of \eqref{SS} satisfying
\begin{equation}\label{bound on W}
\Vert {\vec u}\Vert_{{\vec W}}^2:=\sum_{p\in\mathbb{Z}^2}\left[1+\vert p\vert^2\right]\Vert u_p\Vert_{L^6_{x,t}(\mathbb{R}_x\times\mathbb{R}_t)}^2\le \Lambda_{ls}(E_{ls}({\vec u}))
\end{equation}
for some nondecreasing finite function $\Lambda_{ls}$. In addition, this solution scatters in the sense that there exists ${\vec v}^{\pm\infty}\in h^1L^2$ such that
\begin{equation}\label{LSScat}
\sum_{p\in \Z^2}\left[1+\vert p\vert^2\right]\Vert u_p(t)-e^{it\partial_{xx}}v_p^{\pm\infty}\Vert_{L^2_x(\mathbb{R})}^2\to0\quad\hbox{as }t\to\pm\infty.
\end{equation}
\end{conjecture}

As we also show in the appendix, this conjecture is in fact implied by the conclusion of Theorem \ref{MainThm}. In addition, the local wellposedness theory for \eqref{SS} (cf. Appendix) gives that this conjecture is true under the smallness hypothesis $E_{ls}({\vec u})<\delta$ for some $\delta>0$. Finally, the result of Dodson \cite{Dod} shows that this conjecture is again true if one adds the additional assumption that ${\vec u}(0)$ is, in fact scalar.

A direct consequence of  Conjecture \ref{LSConj} and the propagation of regularity part of the local well-posedness proposition \ref{LWPSS} is the following:

\begin{proposition}\label{LocTheoryForSS}
Assume that Conjecture \ref{LSConj} holds true. Suppose that ${\vec u_0}\in h^1L^2$ and that ${\vec u}\in C(\mathbb{R}: h^1L^2)$ is the solution of \eqref{SS} with initial data ${\vec u}_0$ given by Conjecture \ref{LSConj}. Suppose also that ${\vec v}_0\in h^5 H^4$ satisfies
\begin{equation*}
\Vert {\vec u}_0-{\vec v}_0\Vert_{ h^1L^2}\lesssim \varepsilon,
\end{equation*}
and that ${\vec v}(t)$ is the solution to \eqref{SS} with initial data ${\vec v}(0)={\vec v}_0$. Then, it holds that:
\begin{equation*}
\begin{split}
\Vert (1-\partial_{xx})^2\{(1+\vert p\vert^4)v_p\}_p\Vert_{L^\infty_t(h^1L^2)\cap {\vec W}(\mathbb{R})}&\lesssim_{\Vert {\vec u}_0\Vert_{h^1L^2}} \Vert {\vec v}_0\Vert_{h^5H^4}\\
\Vert {\vec u}-{\vec v}\Vert_{L^\infty_t(h^1L^2)\cap {\vec W}(\mathbb{R})}&\lesssim_{\Vert {\vec u}_0\Vert_{h^1L^2}} \varepsilon.
\end{split}
\end{equation*}
and there exists ${\vec w}^\pm\in h^5 H^4$ such that
\begin{equation*}
\begin{split}
\sum_{p\in\mathbb{Z}^2}\left[1+\vert p\vert^2\right]\Vert v_p(t)-e^{it\partial_{xx}}w^\pm_p\Vert_{L^2_x(\mathbb{R})}^2\to 0\hbox{ as }t\to\pm\infty.
\end{split}
\end{equation*}
\end{proposition}

\subsubsection{Description of the large scale profiles}
We are now in a position to describe the behavior of solutions with initial data $f_M=T_M^{ls}\psi$.
We also define for $\vec \psi=\{\psi_p\}_p$ as in \eqref{DiscFourier},
\begin{equation*}
{\vec \psi}_M:=\{\psi^M_p\}_p,\quad \hbox{ where }\psi^M_p(x):=M^\frac{1}{2}\psi_p(Mx),
\end{equation*}
and, more generally, we similarly define ${\vec u}_M$ from ${\vec u}$ for any ${\vec u}\in h^1L^2$.
Clearly, with $f_M$ as in \eqref{LSrescaled},
\begin{equation*}
\Vert {\vec \psi}_M\Vert_{h^1L^2}=\Vert \psi\Vert_{H^{0,1}(\mathbb{R}\times\mathbb{T}^2)};\quad \Vert {\vec \psi}_M-{\vec f_M}\Vert_{h^1L^2}\to 0\hbox{ as }M\to 0.
\end{equation*}

\medskip

Our first result in this direction is the following lemma:

\begin{lemma}\label{LSstep1}
Assume that Conjecture \ref{LSConj} holds true. Let $\psi\in H^{0,1}(\mathbb{R}\times\mathbb{T}^2)$, $T_0\in(0,\infty)$, and $\rho\in\{0,1\}$ be given, and define $f_{M}=T_M^{ls}\psi(x,y)$ as in \eqref{LSrescaled}. The following conclusions hold:

(i) There is $M_0=M_0(\phi,T_0)$ sufficiently small such that for any $M\leq M_0$ there is a unique solution $U_{M}\in C((-T_0M^{-2},T_0M^{-2});H^1(\mathbb{R}\times\T^2))$ of the initial-value problem
\begin{equation}\label{clo5}
(i\partial_t+\Delta_{\mathbb{R}\times\mathbb{T}^2})U_M=\rho \vert U_M\vert^4U_M,\qquad U_M(0)=f_M.
\end{equation}
Moreover, for any $M\leq M_0$,
\begin{equation}\label{clo6}
\|U_M\|_{X^1(-T_0M^{-2},T_0M^{-2})}\lesssim_{E_{ls}(\psi)}1.
\end{equation}

(ii) Assume $\varepsilon_1\in(0,1]$ is sufficiently small (depending only on $E_{ls}(\psi)$), ${\vec v}_0\in h^5H^4$, and $\Vert {\vec \psi}-{\vec v}_0\Vert_{h^1L^2}\leq\varepsilon_1$. Let ${\vec v}\in C(\mathbb{R}:h^5H^4)$ denote the solution of the initial-value problem
\begin{equation*}
(i\partial_t+\partial_{xx})v_j=\rho \sum_{(p_1,p_2,p_3,p_4,p_5)\in\mathcal{R}(j)}v_{p_1}\overline{v_{p_2}}v_{p_3}\overline{v_{p_4}}v_{p_5},\qquad v_j(0)=v_{0,j},\, j\in\mathbb{Z}^2.
\end{equation*}
For $M\geq 1$ we define
\begin{equation}\label{RemormalizedLS}
\begin{split}
&v_{j,M}(x,t)=M^\frac{1}{2}v_{j}(Mx,M^2t),\qquad\qquad\quad\quad(x,t)\in\mathbb{R}\times(-T_0M^{-2},T_0M^{-2}),\\
&V_{M}(x,y,t)=\sum_{q\in\mathbb{Z}^2}e^{-it\vert q\vert^2}e^{i\langle y,q\rangle} v_{q,M}(x,t),\qquad (x,y,t)\in\mathbb{R}\times\mathbb{T}^2\times(-T_0M^{-2},T_0M^{-2}).
\end{split}
\end{equation}
Then
\begin{equation}\label{LSclo18}
\limsup_{M\to0}\|U_M-V_{M}\|_{X^1(-T_0M^{-2},T_0M^{-2})}\lesssim_{E_{ls}({\vec \psi})}\varepsilon_1.
\end{equation}
\end{lemma}

The proof of Lemma \ref{LSstep1} is the equivalent for large-scale profiles of the proof of \cite[Lemma 4.2]{IoPaSt}. Here we handle the error caused by neglecting the non resonant terms by using a normal form transformation. This reduction to the analysis of the resonant part of the system seems to be relevant only for large scale profiles ($M \to 0$) or in very special regimes.

\begin{proof} When $\rho=0$, there is nothing to prove. It suffices to prove part (ii). All implicit constants are allowed to depend on $\|\psi\|_{ H^{0,1}(\mathbb{R}\times\mathbb{T}^2)}$.

First, we remark that by \eqref{bound on W} and Stichartz estimates on $\R$, we have that
\begin{equation}\label{l5l10}
\begin{split}
\sum_{q\in\mathbb{Z}^2}\langle q\rangle^2\Vert v_{q}\Vert_{L^5_tL^{10}_x}^2&\lesssim {E_{ls}({\vec \psi})},\quad\sum_{q\in\mathbb{Z}^2}\langle q\rangle^{10}\Vert v_{q}\Vert_{L^5_tW^{4,10}_x}^2\lesssim_{E_{ls}({\vec \psi})}\Vert {\vec v}\Vert_{h^5H^4}^2.
\end{split}
\end{equation}

We now show that $V_M$ is an almost solution in the sense of Lemma \ref{stability}.
\begin{equation*}
\begin{split}
\left(i\partial_t+\Delta_{\mathbb{R}\times\mathbb{T}^2}\right)V_M-\vert V_M\vert^4V_M&=-\sum_{q\in\mathbb{Z}^2}e^{-it\vert q\vert^2}e^{i\langle y,q\rangle}\sum_{{\vec p}\in\mathcal{NR}(q)}v_{p_1,M}\overline{v_{p_2,M}}v_{p_3,M}\overline{v_{p_4,M}}v_{p_5,M}e^{-it\Phi_{q,{\vec p}}}\\
&=LHS\\
\Phi_{q,\vec{p}}&=\vert p_1\vert^2-\vert p_2\vert^2+\vert p_3\vert^2-\vert p_4\vert^2+\vert p_5\vert^2-\vert q\vert^2\\
\end{split}
\end{equation*}
where
\begin{equation*}
\mathcal{NR}(q)=\{{\vec p}=(p_1,p_2,p_3,p_4,p_5)\,: \,p_1-p_2+p_3-p_4+p_5-q=0\,;\,\Phi_{q,{\vec p}}\ne 0\}.
\end{equation*}

\medskip

We now claim that, for $I=(-T_0M^{-2},T_0M^{-2})$,
\begin{equation}\label{EEst}
\Vert LHS\Vert_{N^1(I)}\lesssim_{\Vert {\vec v}_0\Vert_{h^5H^4}} M.
\end{equation}

\medskip

We first decompose
\begin{equation*}
LHS=P_{\ge 2^{-10}}^xLHS+P_{<2^{-10}}^xLHS=P_{high}LHS+P_{low}LHS.
\end{equation*}
Using Bernstein's estimate, the definition of the $N$-norm and \eqref{X1SimpleEst}, we obtain that
\begin{equation*}
\begin{split}
\Vert P^x_{\ge 2^{-10}}LHS\Vert_{N^1(0,S)}^2&\lesssim \Vert P^x_{\ge 2^{-10}}\partial_x LHS\Vert_{N^1(0,S)}^2=\sum_{q\in \Z^2} \Vert  e^{i \langle y,q\rangle}P^x_{\ge 2^{-10}}\partial_x \FF_y (LHS) (q)\Vert_{N^1(0,S)}^2\\
&\lesssim \sum_{q\in\mathbb{Z}^2}\langle q\rangle^{-2}\left[\langle q\rangle^2\sum_{{\vec p}\in\mathcal{NR}(q)}\Vert \partial_x\left\{v_{p_1,M}\overline{v_{p_2,M}}v_{p_3,M}\overline{v_{p_4,M}}v_{p_5,M}\right\}\Vert_{L^1_tH_x^1}\right]^2.
\end{split}
\end{equation*}
Since $q\lesssim\max(p_1,p_2,p_3,p_4,p_5)$, we see that, for any $q$,
\begin{equation*}
\begin{split}
&\langle q\rangle^2\sum_{\vec p\in\mathcal{NR}(q)} \Vert \partial_x\left\{v_{p_1,M}\overline{v_{p_2,M}}v_{p_3,M}\overline{v_{p_4,M}}v_{p_5,M}\right\}\Vert_{L^1_tH_x^1}\\
&\lesssim M \sum_{{\vec p} \in \mathcal{NR}(q)}\langle p_1\rangle^2\Vert(\partial_xv_{p_1})_M\Vert_{L^5_tW^{1,10}_x}\Pi_{j=2}^5\left\{\langle p_j\rangle^2\Vert v_{p_j,M}\Vert_{L^5_tW^{1,10}}\right\}
\end{split}
\end{equation*}
Thus $P_{high}LHS$ is acceptable once we sum in $q$.

The contribution of $P_{low}LHS$ is bounded by recalling the definition of the $N^1$ norm in $\eqref{def of N norm}$ and integrating  the time integral by parts as follows:
\begin{equation*}
\begin{split}
&\int_0^Se^{i(S-\sigma)\Delta_{\mathbb{R}\times\mathbb{T}^2}}P_{low}LHS(\sigma)d\sigma\\
&=-\sum_{q\in\mathbb{Z}^2}\sum_{{\vec p}\in\mathcal{NR}(q)}e^{-iS[\vert q\vert^2+\Phi_{q,{\vec p}}]}e^{i\langle y,q\rangle}
\int_0^Se^{i\left(S-\sigma\right)\left[\partial_{xx}+\Phi_{q,{\vec p}}\right]}P_{low}(v_{p_1,M}\overline{v_{p_2,M}}v_{p_3,M}\overline{v_{p_4,M}}v_{p_5,M})d\sigma\\
&=\sum_{q\in\mathbb{Z}^2}\sum_{{\vec p}\in\mathcal{NR}(q)}e^{-iS[\vert q\vert^2+\Phi_{q,{\vec p}}]}e^{i\langle y,q\rangle}\times\\
&
\Big\{\left[ie^{i\left(S-\sigma\right)\left[\partial_{xx}+\Phi_{q,{\vec p}}\right]}\left(\partial_{xx}+\Phi_{q,{\vec p}}\right)^{-1}P_{low}(v_{p_1,M}\overline{v_{p_2,M}}v_{p_3,M}\overline{v_{p_4,M}}v_{p_5,M})\right]_0^S\\
&-i\int_0^Se^{i\left(S-\sigma\right)\left[\partial_{xx}+\Phi_{q,{\vec p}}\right]}\left(\partial_{xx}+\Phi_{q,{\vec p}}\right)^{-1}P_{low}\partial_\sigma\left\{v_{p_1,M}\overline{v_{p_2,M}}v_{p_3,M}\overline{v_{p_4,M}}v_{p_5,M}\right\}d\sigma\Big\}.
\end{split}
\end{equation*}

We bound the $S-$boundary limit using \eqref{X1SimpleEst} and the boundedness of $\left(\partial_{xx}+\Phi_{q,{\vec p}}\right)^{-1}P_{low}$ on $L^2_x(\mathbb{R})$ along with Lemma \ref{SumLem}:
\begin{equation*}
\begin{split}
&\Vert\sum_{q\in\mathbb{Z}^2}\sum_{{\vec p}\in\mathcal{NR}(q)}e^{-iS[\vert q\vert^2+\Phi_{q,{\vec p}}]}e^{i\langle y,q\rangle}
\left(\partial_{xx}+\Phi_{q,{\vec p}}\right)^{-1}P_{low}(v_{p_1,M}\overline{v_{p_2,M}}v_{p_3,M}\overline{v_{p_4,M}}v_{p_5,M})\Vert_{X^1(I)}^2\\
&\lesssim \sum_{q\in\mathbb{Z}}\langle q\rangle^2\Vert\sum_{{\vec p}\in\mathcal{NR}(q)}v_{p_1,M}(0)\overline{v_{p_2,M}}(0)v_{p_3,M}(0)\overline{v_{p_4,M}}(0)v_{p_5,M}(0)\Vert_{H^1_x(\mathbb{R})}^2\\
&+\sum_{q\in\mathbb{Z}}\langle q\rangle^2 \Vert(i\partial_t+\partial_{xx})\sum_{{\vec p}\in\mathcal{NR}(q)}v_{p_1,M}\overline{v_{p_2,M}}v_{p_3,M}\overline{v_{p_4,M}}v_{p_5,M}\Vert_{L^1_tH^1_x}^2\\
&\lesssim_{\|\vec v_0\|_{h^5H^4}} M^4.
\end{split}
\end{equation*}
The $0-$boundary limit is bounded more directly as follows:
\begin{equation*}
\begin{split}
&\Vert \sum_{q\in\mathbb{Z}^2}\sum_{{\vec p}\in\mathcal{NR}(q)}e^{iS\Delta_{\mathbb{R}\times\mathbb{T}^2}}e^{i\langle y,q\rangle}\left(\partial_{xx}+\Phi_{q,{\vec p}}\right)^{-1}P_{low}\left[v_{p_1,M}\overline{v_{p_2,M}}v_{p_3,M}\overline{v_{p_4,M}}v_{p_5,M}\right](0)\Vert_{X^1(I)}^2\\
&\lesssim
 \sum_{q\in\mathbb{Z}^2}\langle q\rangle^2\Vert \sum_{{\vec p}\in\mathcal{NR}(q)}(v_{p_1,M}\overline{v_{p_2,M}}v_{p_3,M}\overline{v_{p_4,M}}v_{p_5,M})(0)\Vert_{H^1_x(\mathbb{R})}^2\lesssim_{\|\vec v_0\|_{h^5H^4}} M^4.
 \end{split}
\end{equation*}

Finally, we also have that
\begin{equation*}
\begin{split}
&\bigg\Vert \sum_{q\in\mathbb{Z}^2}\sum_{{\vec p}\in\mathcal{NR}(q)}e^{-iS[\vert q\vert^2+\Phi_{q,{\vec p}}]}e^{i\langle y,q\rangle}\\
&{}\int_0^Se^{i\left(S-\sigma\right)\left[\partial_{xx}+\Phi_{q,{\vec p}}\right]}\left(\partial_{xx}+\Phi_{q,{\vec p}}\right)^{-1}P_{low}\partial_\sigma\left\{v_{p_1,M}\overline{v_{p_2,M}}v_{p_3,M}\overline{v_{p_4,M}}v_{p_5,M}\right\}d\sigma\bigg\Vert_{X^1}^2\\
&\lesssim \sum_{q\in\mathbb{Z}^2}\langle q\rangle^2\Vert \sum_{{\vec p}\in\mathcal{NR}(q)}\partial_\sigma\left\{v_{p_1,M}\overline{v_{p_2,M}}v_{p_3,M}\overline{v_{p_4,M}}v_{p_5,M}\right\}\Vert_{L^1_tH^1_x}^2 \lesssim_{{\|\vec v_0\|_{h^5H^4}}} M^4.
\end{split}
\end{equation*}
This finishes the proof of \eqref{EEst}.

\medskip

Independently, we also have that
\begin{equation*}
\Vert V_M\Vert_{L^\infty_tH^1_{x,y}(\mathbb{R}\times\mathbb{T}^2\times I)}^2\le\sum_{q\in\mathbb{Z}^2}\langle q\rangle^2\Vert v_{q,M}\Vert_{L^\infty_tH^1_x}^2\le 2\Vert {\vec u}(0)\Vert_{h^1L^2}^2+M^2\Vert {\vec v}\Vert^2_{h^1H^1}
\end{equation*}
and that
\begin{equation}\label{X1Control}
\Vert V_M\Vert_{X^1( I)}\lesssim_{\|\vec u\|_{h^1L^2}}1+C(\Vert {\vec v}\Vert_{h^5H^4})M.
\end{equation}
Indeed, using Lemma \ref{SumLem} in the third line,
\begin{equation*}
\begin{split}
\Vert \left(i\partial_t+\Delta_{\mathbb{R}\times\mathbb{T}^2}\right)V_M\Vert_{N^1(I)}^2&\lesssim \sum_{q\in\mathbb{Z}^2}\langle q\rangle^2\Vert\sum_{{\vec p}\in\mathcal{R}(q)}v_{p_1,M}\overline{v_{p_2,M}}v_{p_3,M}\overline{v_{p_4,M}}v_{p_5,M}\Vert_{L^1_tH^1_x}^2\\
&\lesssim \sum_{q\in\mathbb{Z}^2}\left[\sum_{{\vec p}\in\mathcal{R}(q)}\Pi_{k=1}^5\langle p_k\rangle\Vert v_{p_k,M}\Vert_{L^5_tW^{1,10}_x}\times \langle q\rangle\Pi_{k=1}^5\langle p_k\rangle^{-1}\right]^2\\
&\lesssim \left[\sum_{p\in\mathbb{Z}^2}\langle p\rangle^2\Vert v_{p,M}\Vert_{L^5_tW^{1,10}_x}^2\right]^5
\end{split}
\end{equation*}
and appealing to \eqref{l5l10}, we can justify \eqref{X1Control}.

Using Proposition \ref{stability},
we conclude that, for $M$ small enough (depending on ${\vec v}_0$), the solution $U_M$ of \eqref{NLS} with initial data $V_M(0)$ exists on $I$ and that
\begin{equation*}
\Vert U_M-V_M\Vert_{X^1(I)}\lesssim \varepsilon_1+C(\Vert {\vec v}_0\Vert_{h^5H^4})M,
\end{equation*}
which ends the proof.

\end{proof}

To understand linear and nonlinear evolutions beyond the large-scale window we need an additional extinction lemma:

\begin{lemma}\label{LSExtinction}
For any $\psi\in H^{0,1}(\mathbb{R}\times\mathbb{T}^2)$ and any $\varepsilon>0$, there exists $T_0=T(\psi,\varepsilon)$ and $M_0=M_0(\psi, \epsilon)$ such that for any $T\ge T_0$ and any $M\le M_0$,
\begin{equation}\label{ExtEst}
\Vert e^{it\Delta_{\mathbb{R}\times\mathbb{T}^2}}T^{ls}_M\psi\Vert_{Z(\{M^2\vert t\vert\ge T_0\})} \lesssim \varepsilon.
\end{equation}
\end{lemma}

\begin{proof}
By Strichartz estimates on $\mathbb{R}$, and dominated convergence, there exists $T_0=T_0(\psi,\varepsilon)$ such that
\begin{equation}\label{StricR}
\sum_{p\in\mathbb{Z}^2}\langle p\rangle^2\Vert e^{it\partial_{xx}}\psi_p\Vert_{L^6_{x,t}(\mathbb{R}\times\{\vert t\vert\ge T_0\})}^2\le\varepsilon^{1000}.
\end{equation}
Let $I=\{\vert t\vert\ge T_0\}$ and $I_M=\{M^2\vert t\vert\ge T_0\}$.
We have that
\begin{align*}
e^{it\Delta_{\mathbb{R}\times\mathbb{T}^2}}T^{ls}_M\psi=&\sum_{q\in\mathbb{Z}^2}e^{i(\langle q, y\rangle-t\vert q\vert^2)} \left(M^{1/2}e^{iM^2t\partial_{xx}}\tilde{\psi}^M_{q}(Mx)\right),\quad \tilde{\psi}^M=P^x_{\le M^{-1/100}}\psi\\
=&\sum_{q \in \Z^2} e^{i(\langle q, y\rangle -|q|^2t)}v_{q, M}(t,x)
\end{align*}
where we denoted by:
\begin{equation*}
 v_{p,M}(x,t)=M^{1/2}e^{iM^2t\partial_{xx}}\tilde{\psi}^M_{p}(Mx).
\end{equation*}
Using the fact that $e^{i\langle q,y\rangle}v_{p,M}(x,t)$ is supported in Fourier space in the box centered at $q$ of radius 2 and Bernstein's estimate in $y$, we can estimate on $\mathbb{R}\times\mathbb{T}^2\times I_M$:
$$
\|P_N e^{it\Delta_{\R\times \T^2}} T_M^{ls} \psi\|_{L^6_{x,y,t}} \lesssim  N^{2/3}\left\| \left(\sum_{|q| \sim N} |v_{q,M}(x,t)|^2\right)^{1/2}\right\|_{L^6_{x,t}}\lesssim N^{2/3}\left(\sum_{|q|\sim N} \|v_{q,M}\|_{L^6_{x,t}(\R_x\times \R_t)}^2\right)^{1/2}.
$$

Consequently we see that
\begin{equation}\label{Zrevisionadd}
\begin{split}
\sum_{N\ge 1}N^\frac{2}{3}\Vert P_Ne^{it\Delta_{\mathbb{R}\times\mathbb{T}^2}}T^{ls}_M\psi\Vert_{L^6_{x,y,t}(\mathbb{R}\times\mathbb{T}^2\times I_M)}^2\lesssim&  \sum_{N \geq 1}N^2\sum_{|q|\sim N} \|v_{q,M}\|^2_{L^6_{x,t}(\mathbb{R}\times I_M)}\\
\lesssim& \sum_{q\in \Z^2}\langle q \rangle^2 \|e^{it\partial_{xx}}\widetilde \psi_q^M\|_{L^6_{x,t}(\R \times I)}^2\lesssim \varepsilon^{1000},
\end{split}
\end{equation}
if $M$ is chosen large enough depending on $\psi$ and $\varepsilon$. On the other hand, by Strichartz estimates \eqref{Stric1}, we also have that
\begin{equation*}
\sum_{N\ge 1}N^{5-\frac{p_0}{2}}\Vert P_Ne^{it\Delta_{\mathbb{R}\times\mathbb{T}^2}}T^{ls}_M\psi\Vert_{l^\frac{4p_0}{p_0-2}_\gamma L^{p_0}_{x,y,t}}^{p_0}\lesssim\Vert T^{ls}_M\psi\Vert_{H^1(\mathbb{R}\times\mathbb{T}^2)}^{p_0}\lesssim \Vert \psi\Vert_{H^{0,1}(\mathbb{R}\times\mathbb{T}^2)}^{p_0}
\end{equation*}
for $p_0=4.1,1000$. Interpolating this with \eqref{Zrevisionadd} gives \eqref{ExtEst}.

\end{proof}

We are now ready to describe the nonlinear solutions of \eqref{NLS} corresponding to large-scale profiles. In view of the profile analysis in the next section, we need to consider renormalized large-scale frames defined by:

\begin{equation}\label{renframeLS}
\begin{split}
\widetilde{\mathcal{F}_{ls}}:=&\{(M_k,t_k,p_k,\xi_k)_k: M_k\le 1,\, M_k\to0,\, t_k\in\mathbb{R},\,  p_k=(x_k,0)\in\mathbb{R}\times\mathbb{T}^2 \hbox{ and } \xi_k\in\mathbb{R},\, \xi_k\to\xi_\infty\in\mathbb{R}\\
&\hbox{and such that either $t_k\equiv 0$ or $M_k^2t_k\to\pm\infty$, and either $\xi_k\equiv 0$ or $M_k^{-1}\xi_k\to\pm\infty$} \}.
\end{split}
\end{equation}

We can then describe the solution of \eqref{NLS} with initial data given by a large-scale profile as follows:

\begin{proposition}\label{GEForLSP}
Assume that Conjecture \ref{LSConj} holds true. Fix a renormalized large-scale frame, $\mathcal{S}=(M_k,t_k,(x_k,0),\xi_k)_k\in\widetilde{\mathcal{F}}_{ls}$, $\psi\in H^{0,1}(\mathbb{R}\times\mathbb{T}^2)$, and let
\begin{equation*}
U_k(0)=\Pi_{t_k,x_k}e^{i\xi_kx}T_{M_k}^{ls}\psi.
\end{equation*}

(i) For $k$ large enough (depending only on $\psi$, $\mathcal{S}$), there is a nonlinear solution $U_k\in X^1_c(\mathbb{R})$ of the equation \eqref{NLS} satisfying:
\begin{equation}\label{ControlOnZNormForLSP}
\Vert U_k\Vert_{X^1(\R)}\lesssim_{E_{ls}(\psi)}1.
\end{equation}

(ii) There exists a solution ${\vec v}\in C(\mathbb{R}:h^1L^2)$ of \eqref{SS}
with scattering data ${\vec v_0}{}^{\pm\infty}$ such that the following holds, up to a subsequence:
for any $\varepsilon>0$, there exists $T(\psi,\varepsilon)$ such that for all $T\ge T(\psi,\varepsilon)$, there holds that
\begin{equation}\label{ProxySSHyp}
\Vert U_k-W_{k}\Vert_{X^1(\{\vert t-t_k\vert\le TM_k^{-2}\})}\le\varepsilon,
\end{equation}
for $k$ large enough, where
\begin{equation*}
W_k(x,t)=e^{-i\tau\vert\xi_k\vert^2}e^{ix\xi_k}V_{M_k}(x-x_k-2\xi_k\tau,y,\tau),\quad \tau=t-t_k
\end{equation*}
with $V_{M_k}$ defined in terms of ${\vec v}$ as in \eqref{RemormalizedLS}.
In addition, up to a subsequence,
\begin{equation}\label{ScatLSSol}
\Vert U_k(t)-\Pi_{t_k-t,x_k}e^{ix\xi_k}T_{M_k}^{ls}\underline{{\vec v_0}{}^{\pm\infty}}\Vert_{X^1(\{\pm(t-t_k)\geq TM_k^{-2}\})}\le \varepsilon,
\end{equation}
for $k$ large enough (depending on $\psi,\varepsilon,T$).
\end{proposition}

\begin{proof}
Without loss of generality we may assume that $x_k=0$. Using a Galilean transform and the fact that $\xi_k$ is bounded, we may assume that $\xi_k=0$ for all $k$. We first consider the case when $t_k=0$ for all $k$ and we let ${\vec v}$ be the solution of \eqref{SS} with initial data ${\vec \psi}$. Then, by Proposition \ref{LocTheoryForSS}, we see that there exists $T_0=T_0(\psi,\varepsilon)$ such that
\begin{equation}\label{ChoiceOFT0}
\begin{split}
\sup_{t\ge T_0}\Vert {\vec u}(t)-e^{it\partial_{xx}}{\vec v_0}^{+\infty}\Vert_{h^1L^2}+\Vert e^{it\partial_{xx}} {\vec v_0}^{+\infty}\Vert_{\vec W (\mathbb{R}\times\{t\ge T_0\})}&\le\varepsilon,\\
\sup_{t\le -T_0}\Vert {\vec u}(t)-e^{it\partial_{xx}}{\vec v_0}^{-\infty}\Vert_{h^1L^2}+\Vert e^{it\partial_{xx}}{\vec v_0}^{-\infty}\Vert_{\vec W(\mathbb{R}\times\{t\le -T_0\})}&\le\varepsilon,
\end{split}
\end{equation}
Fix $T\ge T_0$. Applying Lemma \ref{LSstep1}, we see that, provided that $k$ is large enough,
\begin{equation*}
\Vert U_k-V_{M_k}\Vert_{\tilde{X}^1(\{\vert t\vert\le T M_k^{-2}\})}\le\varepsilon.
\end{equation*}
This and Strichartz estimates, together with \eqref{ChoiceOFT0}, Lemma \ref{LSExtinction} imply that
\begin{equation*}
\Vert e^{ it\Delta_{\mathbb{R}\times\mathbb{T}^2}}U_k(\pm T M_k^{-2})\Vert_{Z(\pm t\geq M_k^{-2}T)}\le\varepsilon.
\end{equation*}
At this point, Proposition \ref{LWP} shows that $U_k$ extends to a global solution $U_k\in X^1_c(\mathbb{R})$ satisfying \eqref{ScatLSSol}.

We now turn to the case when $M_k^{-2}\vert t_k\vert \to \infty$. For definiteness, we assume that $M_k^{-2}t_k \to +\infty$ and let ${\vec v}$ be the solutions to \eqref{SS} satisfying:
\begin{equation*}
\Vert {\vec v}(t)-e^{it\partial_{xx}}{\vec \psi}\Vert_{h^1L^2}\to 0
\end{equation*}
as $t\to -\infty$. Let $\psi^\prime=\underline{\vec v}(0) \in H^{0,1}(\R\times \T^2)$ and let $V_k$ be the solution of $\eqref{NLS}$ with initial data $T^{ls}_{M_k}\psi^\prime$.
Applying the first case of the proof to the frame $(N_k, 0,0,0)$ and the family $V_k$ we conclude that:
\begin{equation*}
\Vert V_k(-t_k)-\Pi_{t_k,0}T^{ls}_{M_k}\psi\Vert_{H^1(\mathbb{R}\times \mathbb{T}^2)} \to 0
\end{equation*}
as $k \to \infty$. The conclusion of the proof now follows from Proposition \ref{stability} and by inspecting the behavior of $V_k$.
\end{proof}


\section{Profile Decomposition}\label{Sec-ProfileDec}

Recall the definitions \eqref{DefPi}, \eqref{DefTNE} and \eqref{LSrescaled}. We consider three different kind of profiles corresponding to different frames.

\begin{definition}[Frames and Profiles]\label{DefPro}
\begin{enumerate}
\item We define a frame to be sequence $(N_k,t_k,p_k, \xi_k)_k$ in $2^{\Z} \times \R \times (\R\times \T^2)\times \R$. We distinguish the following three different types of frames: 
\begin{enumerate}

\item A Euclidean frame is a sequence $\mathcal{F}_e=(N_k,t_k,p_k,0)_k$ with $N_k\ge 1$, $N_k\to+\infty$, $t_k\in\mathbb{R}$, $p_k\in\mathbb{R}\times\mathbb{T}^2$. 

\item A large-scale frame is a sequence $\mathcal{F}_{ls}=(M_k,t_k,p_k,\xi_k)_k$ with $M_k\le 1$, $M_k\to0$, $t_k\in\mathbb{R}$, $ p_k=(x_k,0)\in\mathbb{R}\times\mathbb{T}^2$ and $\xi_k\in\mathbb{R}$, $\xi_k\to\xi_\infty\in\mathbb{R}$.

\item A Scale-$1$ frame is a sequence $\mathcal{F}_{1}=(1,t_k,p_k,0)_k$ with $t_k\in\mathbb{R}$, $p_k\in\mathbb{R}\times\mathbb{T}^2$.
\end{enumerate}

\item We say that two frames $(N_k,t_k,p_k,\xi_k)_k$ and $(M_k,s_k,q_k,\eta_k)_k$ are orthogonal if
\begin{equation*}
\lim_{k\to+\infty} \left(\left\vert \ln\frac{N_k}{M_k}\right\vert+N_k^2\vert t_k-s_k\vert +N_k^{-1} \vert \xi_k-\eta_k\vert+N_k\left \vert (p_k-q_k)-2(t_k-s_k)\xi_k \right\vert\right)=+\infty.
\end{equation*}
Two frames that are not orthogonal are called equivalent.

\item Associated to each of the above three types of frames, we associate a profile defined as: 
\begin{enumerate}
\item If $\mathcal{O}=(N_k,t_k,p_k,0)_k$ is a Euclidean frame and if $\phi\in \dot{H}^1(\mathbb{R}^3)$, we define the Euclidean profile associated to $(\phi,\mathcal{O})$ as the sequence $\widetilde{\phi}_{\mathcal{O},k}$ with
\begin{equation*}
\widetilde{\phi}_{\mathcal{O},k}(x,y):=\Pi_{t_k,p_k}(T^e_{N_k}\phi)(x,y).
\end{equation*}
\item If $\mathcal{O}=(M_k,t_k,p_k,\xi_k)_k$ is a large scale frame, if $p_k=(x_k,0)$ and if $\psi\in H^{0,1}(\mathbb{R}\times\mathbb{T}^2)$, we define the large scale profile associated to $(\psi,\mathcal{O})$ as the sequence $\widetilde{\psi}_{\mathcal{O},k}$ with
\begin{equation*}
\widetilde{\psi}_{\mathcal{O},k}(x,y):=\Pi_{t_k,p_k}\left[e^{i\xi_k x}T^{ls}_{M_k}\psi(x,y)\right].
\end{equation*}

\item If $\mathcal{O}=(1,t_k,p_k,0)$ is a scale-1 frame and $W\in H^1(\mathbb{R}\times\mathbb{T}^2)$, we define the Scale-1 profile associated to $(W,\mathcal{O})$ as $\widetilde{W}_{\mathcal{O},k}$ with
\begin{equation*}
\widetilde{W}_{\mathcal{O},k}=\Pi_{t_k,p_k}W.
\end{equation*}
\end{enumerate}

\item Finally, we say that a sequence of functions $\{f_k\}_k \subset H^1(\R\times \T^2)$ is absent from a frame $\mathcal O$ if, up to a subsequence:
$$
\langle f_k, \widetilde{\psi}_{\mathcal O,k}\rangle_{H^1\times H^1} \to 0
$$
as $k\to +\infty$ for any profile $\widetilde{\psi}_{\mathcal O,k}$ associated with $\mathcal O$.
\end{enumerate}
\end{definition}

\subsection{Orthogonality of profiles}
This section summarizes some orthogonality and equivalence properties of frames. These are included in the following lemma:

\begin{lemma}[Frame equivalences and orthogonality]\label{Eq and orth profiles}
\begin{enumerate}
\item Suppose that $\OO$ and $\OO'$ are equivalent Euclidean (respectively large-scale, or scale-1) frames, then there exists an isometry $L$ of $\dot H^1(\R^3)$ (resp. $H^{0,1}(\R\times \T^2), H^1(\R\times \T^2)$) such that, for any profile generator $\psi \in \dot H^1(\R^3)$ (resp. $H^{0,1}(\R\times \T^2), H^1(\R\times \T^2)$) it holds that, up to a subsequence:
\begin{equation}\label{feq 1}
\limsup_{k\to +\infty} \|\widetilde{L\psi}_{\OO, k}-\widetilde{\psi}_{\OO',k}\|_{H^1(\R\times \T^2)}=0
\end{equation}

\item Suppose that $\OO$ and $\OO'$ are orthogonal frames and $\widetilde{\psi}_{\OO,k}$ and $\widetilde{\phi}_{\OO',k}$ are two profiles associated with $\OO$ and $\OO'$ respectively. Then
\begin{align*}
\lim_{k \to+\infty} \langle \widetilde{\psi}_{\OO, k}, \widetilde{\phi}_{\OO',k}\rangle_{H^1\times H^1(\R\times \T^2)}=&0,\\
\lim_{k \to+\infty} \langle \vert \widetilde{\psi}_{\OO,k}\vert^3, \vert \widetilde{\phi}_{\OO',k}\vert^3\rangle=&0.
\end{align*}

\item If $\OO$ is a Euclidean frame and $\widetilde{\psi}_{\OO,k}, \widetilde{\phi}_{\OO, k}$ are two profiles associated to $\OO$, then:

\begin{align*}
\lim_{k \to +\infty} \|\widetilde{\psi}_{\OO,k}\|_{L^2} +\|\widetilde{\phi}_{\OO,k}\|_{L^2} =&0,\\
\lim_{k\to +\infty} \langle \widetilde{\psi}_{\OO,k}, \widetilde{\phi}_{\OO, k}\rangle_{H^1\times H^1(\R \times \T^2)}=&\langle \psi, \phi\rangle_{\dot H^1\times \dot H^1(\R^3)}.\\
\end{align*}

\item If $\OO$ is a scale-1 frame and $\widetilde{\psi}_{\OO,k}, \widetilde{\phi}_{\OO, k}$ are two profiles associated to $\OO$, then:

\begin{align*}
\lim_{k\to +\infty} \langle \widetilde{\psi}_{\OO,k}, \widetilde{\phi}_{\OO, k}\rangle_{H^1\times H^1(\R\times \T^2)}=\langle \psi, \phi\rangle_{H^1\times H^1(\R\times \T^2)}
\end{align*}

\item If $\OO$ is a large-scale frame and $\widetilde{\psi}_{\OO,k}, \widetilde{\phi}_{\OO, k}$ are two profiles associated to $\OO$, then:

\begin{align*}
\lim_{k \to +\infty} \| \widetilde{\psi}_{\OO,k}\|_{L^6_{x,y}(\R\times \T^2)}&=0,\\
\lim_{k\to +\infty} \langle \widetilde{\psi}_{\OO,k}, \widetilde{\phi}_{\OO, k}\rangle_{H^1\times H^1(\R\times \T^2)}&=\langle\psi,\phi\rangle_{H^{0,1}\times H^{0,1}(\mathbb{R}\times\mathbb{T}^2)}+\vert\xi_\infty\vert^2\langle \psi,\phi\rangle_{L^2\times L^2(\mathbb{R}\times\mathbb{T}^2)}\simeq&\langle \psi, \phi\rangle_{H^{0,1}\times H^{0,1}(\R\times \T^2)}
\end{align*}
where here $\simeq$ means bounded above and below by $C$ and $C^{-1}$ respectively for some $C\geq 1$.
\end{enumerate}
\end{lemma}

\proof The proof is tedious but straightforward, so we omit the details (cf. \cite{IoPaSt}, \cite{IoPa}).
\endproof

\subsection{Profile Decomposition} We are now in a position to perform a profile decomposition for bounded sequences of functions in $H^1(\R\times \T^2)$.

\begin{proposition}\label{PDProp}

Assume that $\phi_k$ is a sequence satisfying
\begin{equation}\label{BoundsOnPhik}
\begin{split}
\sup_{k\ge 0}\Vert \phi_k\Vert_{L^2(\mathbb{R}\times\mathbb{T}^2)}+\Vert\nabla_{x,y}\phi_k\Vert_{L^2(\mathbb{R}\times\mathbb{T}^2)}\le E<+\infty,
\end{split}
\end{equation}
then there exists a subsequence (which we also denote by $\phi_k$), a family of Euclidean profiles $\widetilde{\varphi}^\alpha_{\mathcal{O}^\alpha,k}$, a family of large scale profiles $\widetilde{\psi}^\beta_{\mathcal{S}^\beta,k}$, a family of scale $1$ profiles $\widetilde{W}^\gamma_{\mathcal{O^\gamma},k}$ such that, for any $A\ge 1$ and any $k\ge 0$ in the subsequence
\begin{equation}\label{PD}
\phi_k(x,y)=\sum_{1\le\alpha\le A}\widetilde{\varphi}^\alpha_{\mathcal{O}^\alpha,k}(x,y)+\sum_{1\le\beta\le A}\widetilde{\psi}^\beta_{\mathcal{S}^\beta,k}(x,y)+\sum_{1\le \gamma\le A}\widetilde{W}^\gamma_{\mathcal{O}^\gamma,k}(x,y)+R^A_k(x,y)
\end{equation}
with
\begin{equation}\label{SmalnessRA}
\lim_{A\to+\infty}\limsup_{k\to+\infty}\Vert e^{it\Delta_{\mathbb{R}\times\mathbb{T}^2}}R^A_k\Vert_{Z(\mathbb{R})}=0.
\end{equation}
In addition, all the frames are pairwise orthogonal and we have the following orthogonality property:
\begin{equation}\label{AdditionOfEnergies}
\begin{split}
M(\phi_k)&=\sum_{1\le\beta\le A}M(\psi^\beta)+\sum_{1\le \gamma\le A}M(W^\gamma)+M(R^A_k)+o_{A,k\to+\infty}(1),\\
\Vert\nabla_{x,y}\phi_k\Vert_{L^2(\mathbb{R}\times\mathbb{T}^2)}^2&=\sum_{1\le \alpha\le A}\Vert \varphi^\alpha\Vert_{\dot{H}^1(\mathbb{R}^3)}^2+\sum_{1\le\beta\le A}\left[\vert\xi^\beta_\infty\vert^2M(\psi^\beta)+\Vert \nabla_y\psi^\beta\Vert_{L^2(\mathbb{R}\times\mathbb{T}^2)}^2\right]\\
&+\sum_{1\le\gamma\le A}\Vert\nabla_{x,y}W^\gamma\Vert_{L^2(\mathbb{R}\times\mathbb{T}^2)}^2
+\Vert \nabla_{x,y}R^A_k\Vert_{L^2(\mathbb{R}\times\mathbb{T}^2)}^2+o_{A,k\to+\infty}(1),\\
\Vert\phi_k\Vert_{L^6(\mathbb{R}\times\mathbb{T}^2)}^6&=\sum_{1\le \alpha\le A}\Vert \varphi^\alpha\Vert_{L^6(\mathbb{R}^3)}^6+\sum_{1\le\gamma\le A}\Vert W^\gamma\Vert_{L^6(\mathbb{R}\times\mathbb{T}^2)}^6+o_{A\to+\infty,k\to+\infty}(1),\\
\end{split}
\end{equation}
where $\xi_\infty^\beta=\lim_{k \to +\infty} \xi_k^\beta$, $o_{A,k\to+\infty}(1)\to0$ as $k\to+\infty$ for each fixed $A$, and $o_{A\to +\infty, k \to+\infty}(1)\to 0$ in the ordered limit  $\lim _{A\to +\infty} \lim_{k \to +\infty}$.
\end{proposition}

As in \cite{IoPa}, this follows from iteration of the following finitary statement:

\begin{lemma}\label{PDLem}
Let $\delta>0$.
Assume that $\phi_k$ is a sequence satisfying
\begin{equation}\label{BoundsOnPhik}
\begin{split}
\sup_{k\ge 0}\Vert \phi_k\Vert_{L^2(\mathbb{R}\times\mathbb{T}^2)}+\Vert\nabla_{x,y}\phi_k\Vert_{L^2(\mathbb{R}\times\mathbb{T}^2)}\le E<+\infty,
\end{split}
\end{equation}
then there exists a subsequence (which we also denote by $\phi_k$), $A=A(E,\delta)$ Euclidean profiles $\widetilde{\varphi}^\alpha_{\mathcal{O}^\alpha,k}$, $A$ large scale profiles $\widetilde{\psi}^\beta_{\mathcal{S}^\beta,k}$, and $A$ scale $1$ profiles $\widetilde{W}^\gamma_{\mathcal{O^\gamma},k}$ such that, for any $k\ge 0$ in the subsequence
\begin{equation}\label{PD}
\phi_k(x,y)=\sum_{1\le\alpha\le A}\widetilde{\varphi}^\alpha_{\mathcal{O}^\alpha,k}(x,y)+\sum_{1\le\beta\le A}\widetilde{\psi}^\beta_{\mathcal{S}^\beta,k}(x,y)+\sum_{1\le \gamma\le A}\widetilde{W}^\gamma_{\mathcal{O}^\gamma,k}(x,y)+R^A_k(x,y)
\end{equation}
with
\begin{equation*}
\limsup_{k\to+\infty}\left[\Vert e^{it\Delta_{\mathbb{R}\times\mathbb{T}^2}}R^A_k\Vert_{Z(\mathbb{R})}+\sup_t\Vert e^{it\Delta_{\mathbb{R}\times\mathbb{T}^2}}R^A_k\Vert_{L^6_{x,y}(\mathbb{R}\times\mathbb{T}^2)}\right]\le\delta.
\end{equation*}
In addition, all the frames are pairwise orthogonal and we have the following orthogonality property:
\begin{equation*}
\begin{split}
M(\phi_k)&=\sum_{1\le\beta\le A}M(\psi^\beta)+\sum_{1\le \gamma\le A}M(W^\gamma)+M(R^A_k)+o_{k\to+\infty}(1),\\
\Vert\nabla_{x,y}\phi_k\Vert_{L^2(\mathbb{R}\times\mathbb{T}^2)}^2&=\sum_{1\le \alpha\le A}\Vert \varphi^\alpha\Vert_{\dot{H}^1(\mathbb{R}^3)}^2+\sum_{1\le\beta\le A}\left[\vert\xi^\beta_\infty\vert^2 M(\psi^\beta)+\Vert\nabla_y\psi^\beta\Vert_{L^2(\mathbb{R}\times\mathbb{T}^2)}^2\right]\\
&+\sum_{1\le\gamma\le A}\Vert\nabla_{x,y}W^\gamma\Vert_{L^2(\mathbb{R}\times\mathbb{T}^2)}^2
+\Vert \nabla_{x,y}R^A_k\Vert_{L^2(\mathbb{R}\times\mathbb{T}^2)}^2+o_{k\to+\infty}(1),\\
\Vert\phi_k\Vert_{L^6(\mathbb{R}\times\mathbb{T}^2)}^6&=\sum_{1\le \alpha\le A}\Vert \varphi^\alpha\Vert_{L^6(\mathbb{R}^3)}^6+\sum_{1\le\gamma\le A}\Vert W^\gamma\Vert_{L^6(\mathbb{R}\times\mathbb{T}^2)}^6\\
&+\Vert R^A_k\Vert_{L^6(\mathbb{R}\times\mathbb{T}^2)}^6+o_{k\to+\infty}(1),\\
\end{split}
\end{equation*}
where $o_{k\to+\infty}(1)\to0$ as $k\to+\infty$.
\end{lemma}

The proof of Lemma \ref{PDLem} will be completed in two steps: first we extract the Euclidean and Scale-1 profiles by studying the defects of compactness of a Besov-type Strichartz/Sobolev estimate (cf \eqref{Besov}). This extraction leaves only sequences whose linear flow has small critical Besov norm (cf. $\eqref{Besov}, \eqref{BesSmall}$) but large $Z(\R)$ norm, from which we extract the large-scale profiles and finish the proof. 

We start with the first step: For a sequence of functions $\{f_k\}$ in $H^1(\mathbb{R}\times\mathbb{T}^2)$, we consider the following functional
\begin{equation}\label{Besov}
\Lambda_\infty(\{f_k\})=\limsup_{k\to+\infty} \Vert e^{it\Delta} f_k \Vert _{L_t^\infty B^{-\frac{1}{2}}_{\infty,\infty}}:=\limsup_{k\to +\infty}\sup_{(N,t,x,y)}N^{-\frac{1}{2}}\left\vert \left(e^{it\Delta}P_Nf_k\right)(x,y)\right\vert,
\end{equation}
where the supremum is taken over all scales $N\ge 1$, times $t\in\mathbb{R}$ and locations $(x,y)\in\mathbb{R}\times\mathbb{T}^2$.

\begin{lemma}[Extraction of the Euclidean and Scale-1 profiles]\label{PDLem1}
Let $\nu>0$.
Assume that $\phi_k$ is a sequence satisfying \eqref{BoundsOnPhik}, then there exists a subsequence of $\phi_k$, $A=O(\nu^{-2})$ Euclidean profiles $\widetilde{\varphi}^\alpha_{\mathcal{O}^\alpha,k}$, and $A$ scale $1$ profiles $\widetilde{W}_{\mathcal{O}^\gamma,k}$ such that, for any $k\ge 0$ in the subsequence
\begin{equation}\label{PD}
\phi_k^\prime(x,y)=\phi_k(x,y)-\sum_{1\le\alpha\le A}\widetilde{\varphi}^\alpha_{\mathcal{O}^\alpha,k}(x,y)-\sum_{1\le \gamma\le A}\widetilde{W}^\gamma_{\mathcal{O}^\gamma,k}(x,y)
\end{equation}
satisfies
\begin{equation}\label{BesSmall}
\Lambda_{\infty}(\{\phi_k^\prime\})<\nu.
\end{equation}
Besides, all the frames involved are pairwise orthogonal and $\phi^\prime_k$ is absent from all these frames.
\end{lemma}

\begin{proof}[Proof of Lemma \ref{PDLem1}]
We first claim that if $\Lambda_\infty(\{f_k\})\geq \nu$, then there exists a frame $\mathcal{O}$ and an associated (Euclidean or Scale-1) profile $\widetilde{\psi}_{\mathcal{O}_k}$ satisfying
\begin{equation}\label{claim11}
\limsup_{k\to+\infty}\Vert\widetilde{\psi}_{\mathcal{O},k}\Vert_{H^1}\lesssim 1
\end{equation}
and
\begin{equation}\label{claim12}
\limsup_{k\to+\infty}\left\vert \langle f_k,\widetilde{\psi}_{\mathcal{O},k}\rangle_{H^1\times H^{1}}\right\vert\gtrsim \nu.
\end{equation}
In addition, if $f_k$ was absent from a family of frames $\mathcal{O}^\alpha$, then $\mathcal{O}$ is orthogonal to all the previous frames $\mathcal{O}^\alpha$.

Let us prove the claim above. By assumption, up to extracting a subsequence, there exists a sequence $(N_k,t_k,(x_k,y_k))_k$ such that, for all $k$
\begin{equation*}
\begin{split}
\frac{\nu}{2}&\le N_k^{-\frac{1}{2}}\left\vert \left(e^{it_k\Delta}P_{N_k}f_k\right)(x_k,y_k)\right\vert=\left\vert \langle N_k^{-\frac{1}{2}}e^{it_k\Delta}P_{N_k}f_k,\delta_{(x_k,y_k)}\rangle_{\mathcal{D}\times\mathcal{D}^\prime}\right\vert\\
&\le\left\vert \langle f_k,N_k^{-\frac{1}{2}}e^{-it_k\Delta}P_{N_k}\delta_{(x_k,y_k)}\rangle_{H^1\times H^{-1}}\right\vert.
\end{split}
\end{equation*}
Now, first assume that $N_k$ remains bounded, then, up to a subsequence, one may assume that $N_k\to N_\infty$ and since $N_k$ is dyadic, we may even assume that $N_k=N_\infty$ for all $k$. In this case, we define the scale-1 profile $\mathcal{O}=(1,t_k,(x_k,y_k), 0)_k$ and
\begin{equation*}
\psi=(1-\Delta)^{-1}N_\infty^{-\frac{1}{2}}P_{N_\infty}\delta_{(0,0)}.
\end{equation*}
Inequalities \eqref{claim12} is direct. The fact that $\psi$ is  bounded in $H^1$ uniformly in $N_\infty$ follows from the limit below and implies \eqref{claim11}.

Now, we assume that $N_k\to+\infty$ and we define the Euclidean frame $\mathcal{O}=(N_k,t_k,(x_k,y_k),0)_k$ and the function
\begin{equation*}
\psi=\mathcal{F}^{-1}_{\R^3}\left(|\zeta|^{-2}[\eta^3(\zeta)-\eta^3(2\zeta)]\right)\in H^1(\mathbb{R}^3),\quad\zeta=(\zeta_1,\zeta_2,\zeta_3)\in\mathbb{R}^3.
\end{equation*}
It is not hard to prove that
\begin{equation*}
\begin{split}
&\lim_{k\to+\infty}\|(1-\Delta_{\mathbb{R}\times\mathbb{T}^2})T_{N_k}\psi-N_k^{-\frac{1}{2}}P_{N_k}\delta_{(0,0)}\|_{L^{6/5}(\R\times\T^2)}=0\\
&\limsup_{k\to+\infty}\Vert T_{N_k}\psi\Vert_{H^1(\mathbb{R}\times\mathbb{T}^2)}\lesssim 1.
\end{split}
\end{equation*}
Thus $\|(1-\Delta_{\mathbb{R}\times\mathbb{T}^2})T_{N_k}\psi-N_k^{-\frac{1}{2}}P_{N_k}\delta_{(0,0)}\|_{H^{-1}(\R\times\T^2)}\to 0$ and we conclude that
\begin{equation*}
\frac{\nu}{2}\lesssim\left\vert\langle f_k,N_k^{-\frac{1}{2}}e^{-it_k\Delta_{\mathbb{R}\times\mathbb{T}^2}}P_{N_k}\delta_{(x_k,y_k)}\rangle_{H^1\times H^{-1}}\right\vert\lesssim \left\vert\langle f_k,(1-\Delta_{\mathbb{R}\times\mathbb{T}^2})\widetilde{\psi}_{\mathcal{O},k}\rangle_{H^1\times H^{-1}}\right\vert.
\end{equation*}
This gives \eqref{claim12}.

The last claim about orthogonality of $\mathcal{O}$ with $\mathcal{O}^\alpha$ follows from Lemma \ref{Eq and orth profiles} and the existence of a nonzero scalar product in \eqref{claim12}.


\medskip

Now, continuing with the sequence $\{f_k\}_k$ as above, if the frame selected in the previous step was a Scale-$1$ frame, $\mathcal{O}=(1,t_k,p_k,0)_k$,  we consider the sequence
\begin{equation*}
g_k(x,y):=e^{it_k \Delta}f_k((x,y)+p_k)=\Pi_{-(t_k, p_k)}f_k.
\end{equation*}
This is a bounded sequence in $H^1(\mathbb{R}\times\mathbb{T}^2)$, thus, up to considering a subsequence, we can assume that it converges weakly to $\varphi\in H^1(\mathbb{R}\times\mathbb{T}^2)$. We then define the profile corresponding to $\mathcal{O}$ as $\widetilde{\varphi}_{\mathcal{O},k}$. By its definition and \eqref{BoundsOnPhik}, $\varphi$ has norm smaller than $E$. Besides, we also have that
\begin{equation*}
\begin{split}
\frac{\nu}{2}&\lesssim\lim_{k\to+\infty}\langle f_k,\widetilde{\psi}_{\OO, k} \rangle_{H^1\times H^1}\lesssim\lim_{k\to+\infty}\langle g_k,\psi\rangle_{H^1\times H^1}=\langle \varphi,\psi\rangle_{H^1\times H^1}.
\end{split}
\end{equation*}
Consequently, we get that
\begin{equation}\label{NonzeroProfile}
\Vert\varphi\Vert_{H^1}\gtrsim\nu.
\end{equation}
We also observe that since $g_k-\varphi$ weakly converges to $0$ in $H^1$, there holds that
\begin{equation}\label{AdditionOfEnergies}
\begin{split}
\Vert A f_k\Vert_{L^2}^2&=\Vert A g_k\Vert_{L^2}^2=\Vert A(g_k-\varphi)\Vert_{L^2}^2+\Vert A\varphi\Vert_{L^2}^2+o_k(1)\\
&=\Vert A(f_k-\varphi_{\mathcal{O},k})\Vert_{L^2}^2+\Vert A\varphi\Vert_{L^2}^2+o_k(1)
\end{split}
\end{equation}
for $A=1$ or $A=\nabla_{x,y}$.

The situation if $\mathcal{O}$ is a Euclidean frame is similar. In this case, for $k$ large enough, we consider
\begin{equation*}
\varphi_k(y)=N_k^{-\frac{1}{2}}\eta^3(y/(10N_k^\frac{1}{2}))\left(\Pi_{(-t_k,-x_k)}f_k\right)(\Psi(y/N_k)),\quad y\in\mathbb{R}^3.
\end{equation*}
This is a sequence of functions bounded in $\dot{H}^1(\mathbb{R}^3)$. We can thus extract a subsequence that converges weakly to a function $\varphi\in\dot{H}^1(\mathbb{R}^3)$ satisfying
\begin{equation*}
\Vert \varphi\Vert_{\dot{H}^1(\mathbb{R}^3)}\lesssim 1.
\end{equation*}
Now, let $\gamma\in C^\infty_0(\mathbb{R}^3)$; for $k$ large enough,
\begin{equation*}
\begin{split}
\langle f_k,\widetilde{\gamma}_{\mathcal{O},k}\rangle_{H^1\times H^1(\mathbb{R}\times\mathbb{T}^2)}&=\langle \Pi_{-(t_k,x_k)}f_k,T^e_{N_k}\gamma\rangle_{H^1\times H^1(\mathbb{R}\times\mathbb{T}^2)}\\
&=\langle \varphi,\gamma\rangle_{\dot{H}^1\times\dot{H}^1(\mathbb{R}^3)}+o_k(1).
\end{split}
\end{equation*}
Form this and \eqref{claim11}, \eqref{claim12}, we conclude that
\begin{equation}\label{NonzeroProfile2}
\Vert\varphi\Vert_{\dot{H}^1(\mathbb{R}^3)}\gtrsim\nu
\end{equation}
and that
\begin{equation*}
h_k=f_k-\widetilde{\varphi}_{\mathcal{O},k}
\end{equation*}
is absent from the frame $\mathcal{O}$. Now, similarly to \eqref{AdditionOfEnergies} and using Lemma \ref{Eq and orth profiles}, we see that
\begin{equation*}
\begin{split}
\Vert h_k\Vert_{L^2}^2&=\Vert f_k\Vert^2_{L^2}+o_k(1)\\
\Vert \nabla_{x,y} h_k\Vert_{L^2}^2&=\Vert \nabla_{x,y} f_k\Vert_{L^2}^2+\Vert\nabla\varphi\Vert_{L^2(\mathbb{R}^3)}^2-2\langle \nabla f_k,\nabla \widetilde{\varphi}_{\mathcal{O},k}\rangle_{L^2\times L^2}\\
&=\Vert \nabla_{x,y} f_k\Vert_{L^2}^2-\Vert\nabla\varphi\Vert_{L^2(\mathbb{R}^3)}^2+o_k(1).
\end{split}
\end{equation*}


\medskip

Defining $f^0_k=\phi_k$ and for each $\alpha$, $f^{\alpha+1}_k=f^\alpha_k-\widetilde{\psi}_{\mathcal{O}^\alpha,k}$ where $\widetilde{\psi}_{\mathcal{O}^\alpha,k}$ is the profile given by the considerations above;
iterating this claim at most $O(\nu^{-2})$ times and replacing $\phi_k$ by
\begin{equation*}
\phi^\prime_k:=\phi_k-\sum_\alpha\widetilde{\psi}_{\mathcal{O}^\alpha,k},
\end{equation*}
we obtain that $\{\phi^\prime_k\}$ satisfies
\begin{equation}\label{EBound}
\limsup_{k\to+\infty}\Vert \phi_k^\prime\Vert_{H^1}\le E<+\infty
\end{equation}
and
\begin{equation}\label{SmallBesov}
\limsup_{k\to+\infty}\sup_{N\ge 1,t,x,y}N^{-\frac{1}{2}}\left\vert \left(e^{it\Delta}P_N\phi^\prime_k\right)(x,y)\right\vert<\nu.
\end{equation}
This proves our lemma.
\end{proof}

We can now finish the proof of Lemma \ref{PDLem}.

\begin{proof}[Proof of Lemma \ref{PDLem}]

First, for $\nu=\nu(\delta,E)$ to be decided later, we run Lemma \ref{PDLem1} and extract some profiles. Then, we replace $\phi_k$ by $\phi_k^\prime$, thus ensuring that \eqref{BesSmall} holds for the sequence $\{\phi^\prime_k\}_k$.
We now consider
\begin{equation*}
\Lambda_0(\{\phi^\prime_k\})=\limsup_{k\to+\infty}\Vert e^{it\Delta}\phi^\prime_k\Vert_{Z(\mathbb{R})}.
\end{equation*}
If $\Lambda_0(\{\phi_k^\prime\})<\delta$, we may set $R^A_k=\phi^\prime_k$ for all $k$ and we get Lemma \ref{PDLem}. Assume otherwise. 

\medskip

We claim that if $\{\phi^\prime_k\}_k$ satisfies $\Lambda_0(\{\phi^\prime_k\})\ge\delta$ and $\{\phi^\prime_k\}$ is orthogonal to a family of frames $\mathcal{O}^\alpha$, $1\le \alpha\le A$, then there exists a frame $\mathcal{O}$ orthogonal to $\mathcal{O}^\alpha$, $1\le \alpha\le A$ and an associated profile $\widetilde{\varphi}_{\mathcal{O},k}$ such that, after passing to a subsequence, we have that
\begin{equation}\label{CCLClaim}
\begin{split}
\limsup_{k\to+\infty}\Vert\widetilde{\varphi}_{\mathcal{O},k}\Vert_{H^1(\mathbb{R}\times\mathbb{T}^2)}\gtrsim_\delta 1,\quad \phi^\prime_k-\widetilde{\varphi}_{\mathcal{O},k}\hbox{ is absent from }\mathcal{O}.
\end{split}
\end{equation}
Once this claim is established, the end of the proof follows by iterating the extraction process as in Lemma \ref{PDLem1}.

\medskip

	Since $\Lambda_0(\{\phi^\prime_k\})\ge\delta$, by H\"older's inequality and Strichartz estimates \eqref{Stric1}, we have that, for $p_0\in\{9/2,18\}$ and $q_0=\frac{4p_0}{p_0-2}$,
\begin{equation*}
\begin{split}
c_{N}^k&=N^{\frac{5}{p_0}-\frac{1}{2}}\left\Vert e^{it\Delta}P_{N}\phi^\prime_k\right\Vert_{l^{q_0}_\gamma L^{p_0}_{x,y,t}(\mathbb{R}\times\mathbb{T}^2\times I_\gamma)}\\
\Vert c_N^k\Vert_{l^{p_0}_N}&\le \Vert c_N^k\Vert_{l^2_N}^\frac{2}{p_0}\Vert c_N^k\Vert_{l^\infty_N}^\frac{p_0-2}{p_0}
\le\left(\sum_{N\ge 1}N^2\Vert P_N\phi^\prime_k\Vert_{L^2}^2\right)^\frac{1}{p_0}\left(\sup_{N}c_N^k\right)^\frac{p_0-2}{p_0}.\\
\end{split}
\end{equation*}
Using \eqref{EBound}, we obtain that there exists a sequence of scales $N_k\ge 1$ such that, for some $p_0$ in $\{9/2,18\}$,
\begin{equation*}
(\delta/2)^\frac{p_0}{p_0-2}<\Lambda_0(\{\phi^\prime_k\})^\frac{p_0}{p_0-2}\le E^\frac{2}{p_0-2}N_k^{\frac{5}{p_0}-\frac{1}{2}}\left\Vert e^{it\Delta}P_{N_k}\phi^\prime_k\right\Vert_{l^{q_0}_\gamma L^{p_0}_{x,y,t}(\mathbb{R}\times\mathbb{T}^2\times I_\gamma)}.
\end{equation*}

\medskip

We conclude that there exists a sequence $h_k\in C^\infty_c(\mathbb{R}_x\times\mathbb{T}^2_y\times\mathbb{R}_t)$ such that
\begin{equation*}
\begin{split}
&1\le \Vert h_k\Vert_{l^{q_0\prime}_\gamma L^{p_0^\prime}_{x,y,t}(\mathbb{R}\times\mathbb{T}^2\times I_\gamma)}\le 2\\
&(\delta/2)^\frac{p_0}{p_0-2}E^{-\frac{2}{p_0-2}}N_k^{\frac{1}{2}-\frac{5}{p_0}}\le \langle h_k, e^{it\Delta}P_{N_k}\phi^\prime_k\rangle_{L^2\times L^2_{x,y,t,\gamma}}.
\end{split}
\end{equation*}

\medskip

Now, for a given threshold $B$, we introduce the partition function
\begin{equation*}
\chi_B(\gamma)=\begin{cases}
1&\hbox{ if }\Vert h_k\Vert_{L^{p_0^\prime}_{x,y,t}(\mathbb{R}\times\mathbb{T}^2\times I_\gamma)}\ge B\\
0&\hbox{ otherwise}
\end{cases}
\end{equation*}
and we  decompose
\begin{equation*}
h_k(x,y,t)=h_k^{>B}+h_k^{<B}=h_k(x,y,t)\chi_B([\frac{t}{2\pi}])+h_k(x,y,t)(1-\chi_B([\frac{t}{2\pi}]))
\end{equation*}
so that
\begin{equation*}
\begin{split}
&\Vert h_k\Vert_{l^{q_0^\prime}_\gamma L^{p_0^\prime}_{x,y,t}(\mathbb{R}\times\mathbb{T}^2\times I_\gamma)}\le \Vert h_k^{>B}\Vert_{l^{q_0^\prime}_\gamma L^{p_0^\prime}_{x,y,t}(\mathbb{R}\times\mathbb{T}^2\times I_\gamma)}+\Vert h_k^{<B}\Vert_{l^{q_0^\prime}_\gamma L^{p_0^\prime}_{x,y,t}(\mathbb{R}\times\mathbb{T}^2\times I_\gamma)}\\
&\sup_\gamma\Vert h_k^{<B}\Vert_{L^{p_0^\prime}_{x,y,t}(\mathbb{R}\times\mathbb{T}^2\times I_\gamma)}\le B.
\end{split}
\end{equation*}

\medskip

Using Strichartz estimates, we have that, for $p_1=17/4<p_0$,
\begin{equation*}
\limsup_{k\to+\infty}N_k^{\frac{5}{p_1}-\frac{1}{2}}\left\Vert e^{it\Delta}P_{N_k}\phi^\prime_k\right\Vert_{l^\frac{4p_1}{p_1-2}_\gamma L^{p_1}_{x,y,t}(\mathbb{R}\times\mathbb{T}^2\times I_\gamma)}\lesssim\limsup_{k\to+\infty}\Vert \phi^\prime_k\Vert_{H^1}\lesssim E.
\end{equation*}
Interpolating with \eqref{SmallBesov}, we obtain that
\begin{equation*}
\limsup_{k\to+\infty}N_k^{\frac{5}{p_0}-\frac{1}{2}}\left\Vert e^{it\Delta}P_{N_k}\phi^\prime_k\right\Vert_{l^{\frac{4p_0}{p_1-2}}_\gamma L^{p_0}_{x,y,t}(\mathbb{R}\times\mathbb{T}^2\times I_\gamma)}\lesssim E^\frac{p_1}{p_0}\nu^\frac{p_0-p_1}{p_0}.
\end{equation*}

\medskip

We first observe that
\begin{equation*}
\left\vert{\rm \hbox{supp}}_\gamma(h_k^{>B})\right\vert\le(2/B)^{q_0^\prime},
\end{equation*}
and consequently, by H\"older's inequality in $\gamma$,
\begin{equation*}
\begin{split}
\langle e^{it\Delta}\phi^\prime_k,h_k^{>B}\rangle&\le \Vert e^{it\Delta}\phi_k^\prime\Vert_{l^{\frac{4p_0}{p_1-2}}_\gamma L^{p_0}_{x,y,t}(I_\gamma)}\Vert h_k^{>B}\Vert_{l^{q_0^\prime}_\gamma L^{p_0^\prime}_{x,y,t}(I_\gamma)} \left[(2/B)^{q_0^\prime}\right]^{\frac{p_0-p_1}{4p_0}}\\
&\lesssim 2\left[2/B\right]^{\frac{p_0-p_1}{3p_0+2}}E^\frac{p_1}{p_0}\nu^\frac{p_0-p_1}{p_0}N_k^{\frac{1}{2}-\frac{5}{p_0}}.
\end{split}
\end{equation*}
Consequently, for any fixed $B>0$, we can choose $\nu=\nu(B,\delta)$ with
\begin{equation*}
\nu^\frac{1}{18}=c E^{-\frac{2}{p_0-2}-\frac{p_1}{p_0}}\delta^\frac{p_0}{p_0-2}B^{\frac{p_0-p_1}{3p_0+2}},
\end{equation*}
for $c>0$ a universal constant sufficiently small so that,
\begin{equation}\label{NonTrivialIP}
\begin{split}
(\delta/2)^\frac{p_0}{p_0-2}E^{-\frac{2}{p_0-2}}N_k^{\frac{1}{2}-\frac{5}{p_0}}&\le 2\langle h_k^{<B}, e^{it\Delta}P_{N_k}\phi^\prime_k\rangle_{L^2\times L^2_{x,y,t,\gamma}}\\
&\le 2EN_k^{-1}\Vert \int_{\mathbb{R}}e^{-is\Delta}P_{N_k}h_k^{<B}(x,y,s)ds\Vert_{L^2_{x,y}(\mathbb{R}\times\mathbb{T}^2)}.
\end{split}
\end{equation}

\medskip

Using Strichartz estimates \eqref{Adj}, we obtain that
\begin{equation*}
\begin{split}
\Vert \int_{\mathbb{R}}e^{-is\Delta}P_{N_k}h_k^{<B}(x,y,s)ds\Vert_{L^2_{x,y}}
&\lesssim N_k^{\frac{3}{2}-\frac{5}{p_0}}\Vert h_k^{<B}\Vert_{l^2_\gamma L^{p_0^\prime}_{x,y,t}(I_\gamma)}+N_k^{1-\frac{3}{p_0}}\Vert h_k^{<B}\Vert_{l^{q_0^\prime}_\gamma L^{p_0^\prime}_{x,y,t}(I_\gamma)}\\
&\lesssim N_k^{\frac{3}{2}-\frac{5}{p_0}}\Vert h_k^{<B}\Vert_{l^{q_0^\prime}_\gamma L^{p_0^\prime}_{x,y,t}(I_\gamma)}^\frac{2p_0}{3p_0+2}B^\frac{p_0+2}{3p_0+2}+N_k^{\frac{p_0-3}{p_0}}\Vert h_k^{<B}\Vert_{l^{q_0^\prime}_\gamma L^{p_0^\prime}_{x,y,t}(I_\gamma)}\\
&\lesssim N_k^{\frac{3}{2}-\frac{5}{p_0}}B^\frac{p_0+2}{3p_0+2}+N_k^{\frac{p_0-3}{p_0}}.
\end{split}
\end{equation*}
Choosing $B$ such that
\begin{equation*}
B^\frac{p_0+2}{3p_0+2}=\varepsilon E^{-\frac{2}{p_0-2}}\delta^\frac{p_0}{p_0-2}
\end{equation*}
for some absolute constant $\varepsilon>0$ small enough and plugging into \eqref{NonTrivialIP}, we obtain that
\begin{equation*}
\delta^\frac{p_0}{p_0-2}\lesssim \varepsilon\delta^\frac{p_0}{p_0-2}+E^\frac{2}{p_0-2}N_k^{\frac{2}{p_0}-\frac{1}{2}}.
\end{equation*}
If $\varepsilon>0$ is small enough, we obtain a uniform bound
\begin{equation}\label{BdNk}
N_k\lesssim \left[E^\frac{2}{p_0-2}\delta^{-\frac{p_0}{p_0-2}}\right]^\frac{2p_0}{p_0-4}.
\end{equation}

\medskip

To sum up, we have proved that if $\Lambda_0(\{\phi_k^\prime\})>\delta$, then there exists a sequence of scales $\{N_k\}$ satisfying \eqref{BdNk} and such that
\begin{equation*}
N_k^{\frac{5}{p_0}-\frac{1}{2}}\Vert P_{N_k}e^{it\Delta}\phi_k^\prime\Vert_{l^{q_0}_\gamma L^{p_0}_{x,y,t}(I_\gamma)}>\delta^\frac{p_0}{p_0-2}E^{-\frac{2}{p_0-2}}.
\end{equation*}
Using again Strichartz estimates, we see that
\begin{equation*}
N_k^{\frac{5}{p_2}-\frac{1}{2}}\Vert P_{N_k}e^{it\Delta}\phi_k^\prime\Vert_{l^\frac{4p_2}{p_2-2}_\gamma L^{p_2}_{x,y,t}(I_\gamma)}\lesssim E
\end{equation*}
for $p_2\in\{4.001,1000\}$ and using H\"older's estimate, we deduce that
\begin{equation*}
\Vert P_{N_k}e^{it\Delta}\phi_k^\prime\Vert_{L^{6}_{x,y,t}(\mathbb{R}\times\mathbb{T}^2\times\mathbb{R})}>c(\delta,E)
\end{equation*}
for some $c(\delta,E)>0$. Here and below, $c(\delta,E)$ denotes a small positive constant depending only on $\delta$ and $E$ which is allowed to change from line to line.

\medskip

At this point, we write
\begin{equation*}
P_{N_k}e^{it\Delta_{\mathbb{R}\times\mathbb{T}^2}}\phi_k^\prime(x,y)=\sum_{z=(z_1,z_2)\in\mathbb{Z}^2;\, \vert z_i\vert\le N_k} e^{-it\vert z\vert^2}e^{i\langle z, y\rangle}\eta^2_{N_k}(z)\left[e^{it\Delta_{\mathbb{R}}}\phi^{\prime\prime}_{z,k}(x)\right]
\end{equation*}
where
\begin{equation*}
\phi^{\prime\prime}_{z,k}(x)=\frac{1}{(2\pi)^2}\int_{\mathbb{T}^2}P_{N_k}\phi^\prime_k(x,y)e^{-i \langle z, y\rangle}dy
\end{equation*}
and, extracting a subsequence, we conclude that there exists $z$ such that, for all $k$,
\begin{equation}\label{BdsIS}
\begin{split}
\hbox{\rm supp}(\mathcal{F}_x\phi^{\prime\prime}_{z,k})&\subset [-3N_k,3N_k]\\
\Vert e^{it\Delta_{\mathbb{R}}}\phi^{\prime\prime}_{z,k}\Vert_{L^{6}_{x,t}(\mathbb{R}\times\mathbb{R})}&>c(\delta,E)\\
\Vert \phi^{\prime\prime}_{z,k}\Vert_{H^1(\mathbb{R})}&\le M.\\
\end{split}
\end{equation}

\medskip

To proceed, we borrow the following proposition from \cite{BegVar, CarKer, TaoIS} (we use the version in \cite{TaoIS}):

\begin{proposition}[\cite{BegVar,CarKer,TaoIS}]\label{TaoIS}
For any $M,c(\delta,E)>0$, there exists a finite set $\mathcal{C}\subset L^2(\mathbb{R})$ of functions satisfying,
\begin{equation}\label{CondForvvv}
\Vert v\Vert_{L^2(\mathbb{R})}=1,\quad\Vert v\Vert_{L^1(\mathbb{R})}\le S(E,M,\delta),\,\,\forall v\in\mathcal{C}
\end{equation}
and $\kappa(M,c(\delta,E))>0$ such that whenever $u\in L^2(\mathbb{R})$ obeys the bounds
\begin{equation*}
\Vert u\Vert_{L^2}\le M,\quad \Vert e^{it\Delta_{\mathbb{R}}}u\Vert_{L^6_{x,t}(\mathbb{R}\times\mathbb{R})}\ge c(\delta,E),
\end{equation*}
then there exists $v\in \mathcal{C}$ and $(\lambda,\xi_0,t_0,x_0)\in\mathbb{R}_+\times\mathbb{R}\times\mathbb{R}\times\mathbb{R}$ such that
\begin{equation*}
\langle u,v^\prime\rangle_{L^2\times L^2(\mathbb{R})}\ge\kappa,\quad v^\prime(x)=\lambda^{\frac{1}{2}}e^{ix\xi_0}\left[e^{it_0\Delta_{\mathbb{R}}}v\right]\left(\lambda(x-x_0)\right).
\end{equation*}
\end{proposition}

This proposition gives, after extraction, a function $v\in L^2(\mathbb{R})$ satisfying \eqref{CondForvvv} and a sequence $(\lambda_k,\xi_k,t_k,x_k)\in\mathbb{R}_+\times\mathbb{R}\times\mathbb{R}\times\mathbb{R}$ such that
\begin{equation}\label{NontrivialIP}
\langle \phi^{\prime\prime}_{z,k},v_k\rangle\ge\kappa,\quad v_k(x)=\lambda_k^{\frac{1}{2}}e^{ix\xi_k}\left[e^{it_k\Delta_{\mathbb{R}}}v \right]\left(\lambda_k(x-x_k)\right).
\end{equation}
We may assume that $v$ has a compactly supported Fourier transform.

\medskip

We now claim that $\lambda_k$ and $\vert\xi_k\vert$ remain bounded. Indeed, we compute that
\begin{equation*}
\begin{split}
\mathcal{F}_{\mathbb{R}}v_k(\xi)&=e^{i\xi_kx_k}e^{-ix_k\xi}\lambda_k^{-\frac{1}{2}}\left[\mathcal{F}_{\mathbb{R}}e^{it_k\Delta_{\mathbb{R}}}v\right](\frac{\xi-\xi_k}{\lambda_k})\\
\Vert \langle\xi\rangle^{-1}\mathcal{F}_{\mathbb{R}}v_k(\xi)\Vert_{L^2([-4N_k,4N_k])}&\le \sqrt{\pi}\lambda_k^{-\frac{1}{2}}S(M,E,\delta)\\
\Vert v_k\Vert_{{H}^{-1}(\mathbb{R})}^2&\sim \int_{\R}\left[\frac{1}{1+\vert \xi_k+\lambda_k\eta\vert}\right]^2\left\vert\mathcal{F}_{\mathbb{R}}v\right\vert^2(\eta)d\eta.
\end{split}
\end{equation*}
The second line, together with \eqref{BdsIS} and \eqref{NontrivialIP} implies that
\begin{equation*}
\kappa\le \sqrt{\pi}MS(M,E,\delta)\lambda_k^{-\frac{1}{2}}
\end{equation*}
which forces $\lambda_k\le C_\kappa(E,M,\delta)=\pi M^2\kappa^{-2}S(M,E,\delta)$. Similarly, the third line, together with \eqref{BdsIS} and \eqref{NontrivialIP}, along with the uniform bound on $\lambda_k$ and the compactness of the support of $\FF_\R v$, forces $\vert\xi_k\vert\le C_\xi(E,M,\delta)$.

\medskip

Assume first that $\lambda_k$ remains bounded from below. Then, up to extracting a subsequence, we may assume that $\lambda_k\to\lambda_\infty\in(0,\infty)$. Similarly, we may assume that $\xi_k\to\xi_\infty$. Then, setting
\begin{equation*}
\tilde{\psi}(x,y)=e^{i\langle z,y\rangle}e^{ix\xi_\infty}\lambda_\infty^\frac{1}{2}v(\lambda_\infty x),\quad  t^\prime_k=-\lambda_k^{-2}t_k,\quad x^\prime_k=x_k+2t^\prime_k\xi_k
\end{equation*}
and defining the frame $\mathcal{O}=\{(1,t^\prime_k,(x^\prime_k,0),0)_k\}$, we see from \eqref{NontrivialIP} that
\begin{equation*}
\begin{split}
\kappa&\le\langle P_{N_k}\phi^\prime_k, e^{i\langle z,y\rangle}e^{ix\xi_k}\lambda_k^\frac{1}{2}\left[e^{-it_k\Delta_{\mathbb{R}}}v\right](\lambda_k (x-x_k))\rangle\\
&\le\langle \phi^\prime_k, e^{i(x_k \xi_\infty-t^\prime_k(\vert z\vert^2-\vert \xi_\infty\vert^2))}P_{N_k}e^{-it^\prime_k\Delta_{\mathbb{R}\times\mathbb{T}^2}}\tilde{\psi}(x-x^\prime_k)\rangle+o_k(1)
\end{split}
\end{equation*}
Since $1 \leq N_k \lesssim 1$ is dyadic, we may assume (up to a subsequence) that $N_k=N_\infty$. As a result, setting $\psi=P_{N_\infty}\tilde{\psi}$, we see that the Scale-1 profile $\widetilde{\psi}_{\mathcal{O},k}$ satisfies both \eqref{claim11} and \eqref{claim12}. We also conclude that $\mathcal{O}$ is orthogonal to $\mathcal{O}^\alpha$, $1\le \alpha\le A$. In addition, proceeding as in the proof of Lemma \ref{PDLem}, we find $\varphi$ satisfying \eqref{CCLClaim}.
\medskip

Assume now that $\lambda_k\to 0$. Let $M_k$ be a dyadic number such that $1\le\lambda_k^{-1}M_k\le 2$ and consider the sequence
\begin{equation*}
\Phi_k(x,y)=M_k^{-\frac{1}{2}}e^{-it_k\partial_{xx}}\left[e^{i\xi_k x/M_k}\phi^\prime_k(x_k+\frac{x}{M_k},y)\right].
\end{equation*}
We have that
\begin{equation*}
\Vert \Phi_k\Vert_{L^2(\mathbb{R}\times\mathbb{T}^2)}^2+\Vert\nabla_y\Phi_k\Vert_{L^2(\mathbb{R}\times\mathbb{T}^2)}^2\le M^2+E^2,\\
\end{equation*}
hence, up to extracting a subsequence, we may assume that
$
\Phi_k\rightharpoonup \Phi$ in $H^{0,1}(\mathbb{R}\times\mathbb{T}^2)$.
We define
\begin{equation*}
t^\prime_k=-M_k^{-2}t_k,\quad \xi_k^\prime=-\xi_k, \quad x^\prime_k=x_k+2t_k^\prime\xi_k^\prime,
\end{equation*}
and $\mathcal{O}=(M_k,t_k^\prime,(x^\prime_k,0),\xi^\prime_k)$. Then we obtain from the existence of a nonzero inner product \eqref{NontrivialIP} that $\mathcal{O}$ is orthogonal to $\mathcal{O}^\alpha$, $1\le \alpha\le A$ and we see from the definition of $\Phi_k$ that \eqref{CCLClaim} holds with $\varphi=e^{i\theta_\infty}\Phi$ for some $\theta_\infty \in \R/\Z$. This finishes the proof.

\end{proof}

\section{Induction on energy}\label{Sec-IOE}

\subsection{Proof of Theorem \ref{MainThm}}

We are now ready to prove Theorem \ref{MainThm}. We follow an induction on energy method formalized in \cite{KeMe,KeMe2}. Define
\begin{equation*}
\Lambda(L)=\sup\{\Vert u\Vert_{Z(I)}^2:u\in X^1_{loc}(I),E(u)+\frac{1}{2}M(u)\le L\}
\end{equation*}
where the supremum is taken over all strong solutions of full energy less than $L$. By the local well-posedness theory in Theorem \ref{Thm1}, this is sublinear in $L$ and finite for $L$ sufficiently small. In addition, define
\begin{equation*}
L_{max}=\sup\{L:\Lambda(L)<+\infty\}.
\end{equation*}
Our goal is to prove that $L_{max}=+\infty$. As in \cite{TaViZh}, the key proposition is:
\begin{proposition}\label{KeyProp}
Assume that $L_{max}<+\infty$ and that Conjecture \ref{SSConj} holds for $L_{max}$. Let $\{t_k\}_k$, $\{a_k\}_k$, $\{b_k\}_k$ be arbitrary sequences of real numbers and $\{u_k\}$ be a sequence of solutions to \eqref{NLS} such that $u_k\in X^1_{c,loc}(t_k-a_k,t_k+b_k)$ and satisfying
\begin{equation}\label{CondForComp}
L(u_k)\to L_{max},\quad \Vert u_k\Vert_{Z(t_k-a_k,t_k)}\to+\infty,\quad\Vert u_k\Vert_{Z(t_k,t_k+b_k)}\to+\infty.
\end{equation}
Then, passing to a subsequence, there exists a sequence $x_k\in\mathbb{R}$ and $w\in H^1(\mathbb{R}\times\mathbb{T}^2)$ such that
\begin{equation}\label{Compactness}
w_k(x,y)=u_k(x-x_k,y,t_k)\to w
\end{equation}
strongly in $H^1(\mathbb{R}\times\mathbb{T}^2)$.
\end{proposition}
The proof of this key proposition is postponed to the end of this section.

\begin{corollary}\label{CorComp}
Assume that $L_{max}<+\infty$ and that Conjecture \ref{SSConj} holds for some $E^{ls}_{max}\ge L_{max}$. Then, there exists $u\in X^1_{loc}(\mathbb{R})$ solving \eqref{NLS} and a Lipschitz function $\underline{x}:\mathbb{R}\to\mathbb{R}$ such that $L(u)=L_{max}$ and
\begin{equation}\label{FlowComp}
\begin{split}
\sup_{t\in\mathbb{R}}\vert\underline{x}^\prime(t)\vert&\lesssim 1\\
\{u(x-\underline{x}(t),y,t):t\in\mathbb{R}\}&\hbox{ is precompact in }H^1(\mathbb{R}\times\mathbb{T}^2).
\end{split}
\end{equation}
\end{corollary}

\begin{proof}[Proof of Corollary \ref{CorComp}]

Assuming that $L_{max}<+\infty$, we can find a sequence of solutions of \eqref{NLS} $u_k$ satisfying \eqref{CondForComp}. Applying Proposition \ref{KeyProp}, we can extract a subsequence and obtain a sequence $x_k$ such that \eqref{Compactness} holds for some $w\in H^1(\mathbb{R}\times\mathbb{T}^2)$. Clearly $L(w)=L_{max}$. Let $U\in C(I:H^1)$ be the (maximal) strong solution of \eqref{NLS} with initial data $w$, defined on $I=(-a_\infty,b_\infty)$. A consequence of the local wellposedness theory and \eqref{CondForComp} is that
\begin{equation}\label{Contra1}
\Vert W\Vert_{Z(-a_\infty,0)}=\Vert W\Vert_{Z(0,b_\infty)}=+\infty.
\end{equation}

\medskip

We first claim that there exists $\kappa>0$ such that for all $t$ in $I$,
\begin{equation}\label{Contra2}
\Vert U\Vert_{Z(t-2\kappa,t+2\kappa)}\le 2.
\end{equation}
In particular, $U$ is global, $a_\infty=b_\infty=\infty$.

Assume \eqref{Contra2} is not true. Then there exists a sequence $t_k\in I$ such that
\begin{equation*}
\Vert U\Vert_{Z(t_k-\frac{1}{k},t_k+\frac{1}{k})}\ge 2.
\end{equation*}
We can apply Proposition \ref{KeyProp} to the sequence $U(t_k)$ and obtain that, up to a subsequence, there exists $x_k$ such that $U_k(x,y):=U(t_k,x-x_k,y)\to\omega$ in $H^1$. Let $W$ be the nonlinear solution of \eqref{NLS} with initial data $\omega$. By the local wellposedness theory in Proposition \ref{LWP}, there exists $\kappa^\ast>0$ such that
\begin{equation*}
\Vert W\Vert_{Z(-\kappa^\ast,\kappa^\ast)}\le 1
\end{equation*}
and still by the stability theory Proposition \ref{stability}, we obtain that, for $k$ large enough,
\begin{equation*}
\Vert U\Vert_{Z(t_k-\kappa^\ast,t_k+\kappa^\ast)}\le 2
\end{equation*}
which gives a contradiction for $k$ large enough.
 
\medskip

We now prove \eqref{FlowComp}. We define the sequence of times $\{t_k\}=\{k\kappa\}$ and for each $t_k$, we define $\underline{x}_k$ and $R_k$ such that
\begin{equation}\label{RK}
\frac{1}{2}\int_{\{\vert x-\underline{x}_k\vert\le R_k\}}\int_{\mathbb{T}^2}\left[\vert u(t_k,x,y)\vert^2+\vert\nabla u(t_k,x,y)\vert^2+\frac{1}{3}\vert u(t_k,x,y)\vert^6\right]dxdy=99/100L_{max}
\end{equation}
and $R_k$ is minimal with this property.
While ${\underline x}_k$ is not necessarily unique, we claim that there exists $D$ such that for all $k$,
\begin{equation}\label{Contra3}
R_k\le D,\quad \vert \underline{x}_k-\underline{x}_{k+1}\vert\le D
\end{equation}
and that
\begin{equation}\label{Precomp}
\{u(t_k+s,x-\underline{x}_k),k\in\mathbb{Z},s\in (-\kappa,\kappa)\}\quad \text{is precompact in $H^1(\mathbb{R}\times\mathbb{T}^2)$.}
\end{equation}

The fact that the $R_k$ are uniformly bounded is a direct consequence of compactness up to translation of $\{u(t_k)\}_k$. Assume now that $\{v_k(x,y)=u(x-\underline{x}_k,y,t_k)\}_k$ was not precompact in $H^1(\mathbb{R}\times\mathbb{T}^2)$. Then, we could find $\varepsilon>0$ and a subsequence $k^\prime$ such that for all $k_1^\prime,k_2^\prime$,
\begin{equation}\label{Imp1}
\Vert v_{{k_1}^\prime}-v_{{k_2}^\prime}\Vert_{H^1(\mathbb{R}\times\mathbb{T}^2)}>\varepsilon.
\end{equation}
On the other hand, applying once again Proposition \ref{KeyProp}, we see that there exists a sequence $\overline{x}^k$ and a subsequence of $k^\prime$ such that
\begin{equation*}
v_{k^{\prime\prime}}(x-\overline{x}^{k^{\prime\prime}},y)\to w(x,y)\quad \hbox{strongly in }H^1(\mathbb{R}\times\mathbb{T}^2).
\end{equation*}
Necessarily, from \eqref{RK}, $\{\overline{x}^k\}_k$ remains bounded, so that the convergence of $\{v_{k^{\prime\prime}}\}_k$ contradicts \eqref{Imp1}. Using \eqref{Contra2} and the precompactness of $\{v_k\}_k$, we obtain \eqref{Precomp}. For the same reason as above, this implies the second statement in \eqref{Contra3}. Choosing $\underline{x}(t)$ to be a Lipschitz function satisfying $\underline{x}(t_k)=\underline{x}_k$, we obtain \eqref{FlowComp}. This ends the proof.
\end{proof}

Now, we can finish the proof of our main theorem.

\begin{proposition}
Assume that $u$ satisfies the conclusions of Corollary \ref{CorComp}, then $u=0$. In particular, $L_{max}\ge E^{ls}_{max}$. 
\end{proposition}

\begin{proof}
Assume that $u\ne 0$. Then, from the compactness property \eqref{FlowComp}, we see that there exists $\rho>0$ such that
\begin{equation}\label{RHO}
\inf_{t\in\mathbb{R}}\min(\Vert u(t)\Vert_{L^6_{x,y}(\mathbb{R}\times\mathbb{T}^2)},\Vert u(t)\Vert_{L^2_{x,y}(\mathbb{R}\times\mathbb{T}^2)})\ge\rho.
\end{equation}
The proof we use is adapted from an argument of Duyckaerts-Holmer-Roudenko \cite{DuHoRo} (see also \cite{Pa2}). The fundamental tool is a monotonicity formula derived from the conserved momentum
\begin{equation*}
\hbox{Mom}(u)=\hbox{Im}\int_{\mathbb{R}\times\mathbb{T}^2}\overline{u}(x,y,t)\nabla u(x,y,t)dxdy.
\end{equation*}
First, replacing $u$ by\footnote{$\xi_0$ is well defined since $M(u)\ne 0$ by \eqref{RHO}.}
\begin{equation*}
\begin{split}
&v(z,t)=e^{-i\vert\xi_0\vert^2t+i\langle z,\xi_0\rangle}u(z-2\xi_0t,t),\quad M(u)\xi_0=-\hbox{Mom}(u)\\
&M(v)=M(u),\quad E(u)=E(v)+\vert \xi_0\vert^2M(u)
\end{split}
\end{equation*}
we may assume that $u$ satisfies
\begin{equation}\label{ZM}
\hbox{Mom}(u)=0
\end{equation}
in addition to \eqref{FlowComp}.

Define the Virial action by
\begin{equation*}
\begin{split}
A_R(t)=\int_{\mathbb{R}\times\mathbb{T}^2}\chi_R(x-\underline{x}(t))\cdot (x-\underline{x}(t))\cdot\hbox{Im}\left[\overline{u}(x,y,t)\partial_xu(x,y,t)\right]dxdy
\end{split}
\end{equation*}
for $\chi_R(x)=\chi(x/R)$ and $\chi$ satisfies $\chi(x)=1$ when $\vert x\vert\le 1$ and $\chi(x)=0$ when $\vert x\vert\ge 2$. Clearly
\begin{equation}\label{AR}
\sup_t\vert A_R(t)\vert\lesssim R.
\end{equation}
Independently, we compute that
\begin{equation*}
\begin{split}
\frac{d}{dt}A_R=&-\underline{x}^\prime(t)\hbox{Im}\int_{\mathbb{R}\times\mathbb{T}^2}\overline{u}(x,y,t)\partial_xu(x,y,t)dxdy\\
&-\underline{x}^\prime(t)\int_{\mathbb{R}\times\mathbb{T}^2}\left\{(\chi^\prime)_R(x-\underline{x}(t))\frac{x-\underline{x}(t)}{R}-(1-\chi_R(x-\underline{x}))\right\}\hbox{Im}\left[\overline{u}(x,y,t)\partial_xu(x,y,t)\right]dxdy\\
&+\int_{\mathbb{R}\times\mathbb{T}^2}\chi_R(x-\underline{x}(t))(x-\underline{x}(t))\cdot\partial_t\hbox{Im}\left[\overline{u}(x,y,t)\partial_xu(x,y,t)\right]dxdy.
\end{split}
\end{equation*}
The first line vanishes thanks to assumption \eqref{ZM}. The second line can be bounded by
\begin{equation*}
\int_{\{\vert x-\underline{x}(t)\vert\ge R\}}\int_{\mathbb{T}^2}\left[\vert u(x,y,t)\vert^2+\vert\nabla u(x,y,t)\vert^2\right]dxdy=O_R(t)
\end{equation*}
where, thanks to \eqref{FlowComp},
\begin{equation}\label{OR}
\sup_t O_R(t)\to 0\,\hbox{ as }R\to+\infty.
\end{equation}
Since
\begin{equation*}
\partial_t\hbox{Im}\left[\overline{u}(x,y,t)\partial_xu(x,y,t)\right]=\partial_x\Delta\frac{\vert u\vert^2}{2}-2\hbox{div}\left\{\hbox{Re}\{\partial_x\overline{u}\nabla u\} \right\}-\frac{4}{6}\partial_x\vert u\vert^6,
\end{equation*}
we get that
\begin{equation*}
\begin{split}
\frac{d}{dt}A_R&=4\int_{\mathbb{R}\times\mathbb{T}^2}\chi_R(x-\underline{x}(t))\left\{\frac{1}{2}\vert \partial_xu(x,y,t)\vert^2+\frac{1}{6}\vert u(x,y,t)\vert^6\right\}dxdy\\
&+4\int_{\mathbb{R}\times\mathbb{T}^2}\left(\chi^\prime\right)_R(x-\underline{x}(t))\frac{x-\underline{x}(t)}{R}\left\{\frac{1}{2}\vert \partial_xu(x,y,t)\vert^2+\frac{1}{6}\vert u(x,y,t)\vert^6\right\}dxdy\\
&-
\int_{\mathbb{R}\times\mathbb{T}^2}\frac{\vert u(x,y,t)\vert^2}{2}\partial_x^3\left[(x-\underline{x}(t))\chi_R(x-\underline{x}(t))\right]dxdy+O_R(t)\\
&=4\int_{\mathbb{R}\times\mathbb{T}^2}\left\{\frac{1}{2}\vert \partial_xu(x,y,t)\vert^2+\frac{1}{6}\vert u(x,y,t)\vert^6\right\}dxdy+\tilde{O}_R(t)
\end{split}
\end{equation*}
where $\tilde{O}_R(t)$ satisfies the same bounds as $O_R(t)$ in \eqref{OR}. Integrating this equality, we obtain
\begin{equation*}
\vert A_R(t)-A_R(0)\vert\ge 4t\rho-t\sup_t\tilde{O}_R(t).
\end{equation*}
Taking $R$ sufficiently large and using \eqref{OR}, we obtain, when $t$ is sufficiently large a contradiction with \eqref{AR}. This finishes the proof.
\end{proof}

\subsection{Proof of Proposition \ref{KeyProp}}\label{ProofOfKP}

\begin{proof}

Without loss of generality, we may assume that $t_k=0$. We apply Proposition \ref{PDProp} to the sequence $\{u_k(0)\}_k$ which is indeed bounded in $H^1(\mathbb{R}\times\mathbb{T}^2)$. This way we obtain, for all $J$,
\begin{equation*}
u_k(0)=\sum_{1\le\alpha\le J}\widetilde{\varphi}_{\mathcal{O}^\alpha,k}^\alpha+\sum_{1\le\beta\le J}\widetilde{\psi}^\beta_{\mathcal{S}^\beta,k}+\sum_{1\le\gamma\le J}\widetilde{\omega}^\gamma_{\mathcal{O}^\gamma,k}+R^J_k.
\end{equation*}

\medskip

{\bf Case I}: There are no profiles. Then, taking $J$ sufficiently large, we see from \eqref{SmalnessRA} that
\begin{equation*}
\begin{split}
\Vert e^{it\Delta}u_k(0)\Vert_{Z(\mathbb{R})}&=\Vert e^{it\Delta}R^J_k\Vert_{Z(\mathbb{R})}\le \delta_0/2
\end{split}
\end{equation*}
for $k$ sufficiently large,
where $\delta_0=\delta_0(L_{max})$ is given in Proposition \ref{LWP}. Then we see from Proposition \ref{LWP} that $u_k$ can be extended on $\mathbb{R}$ and that
\begin{equation*}
\lim_{k\to+\infty}\Vert u_k\Vert_{Z(\mathbb{R})}\le \delta_0
\end{equation*}
which contradicts \eqref{CondForComp}.

\medskip

Hence, we see that there exists at least one profile. Using Lemma \ref{Eq and orth profiles} and passing to a subsequence, we may renormalize every Euclidean profile, that is, up to passing to an equivalent profile, we may assume that every Euclidean frame $\mathcal{E}^\alpha$ belongs to $\widetilde{\mathcal{F}}_e$, every large-scale frame $\mathcal{S}^\beta$ belongs to $\widetilde{\mathcal{F}}_{ls}$, and that every Scale-1 frame $\mathcal{O}^\gamma$ belong to $\widetilde{\mathcal{F}}_1$, where $\widetilde{\mathcal{F}}_e$ was defined in \eqref{renframe}, $\widetilde{\mathcal{F}}_{ls}$ was defined in \eqref{renframeLS}, and 
\begin{equation}\label{renframeS1}
\widetilde{\mathcal{F}}_1:=\{(1,t_k,p_k,0)_k : \hbox{either $\vert t_k\vert \to +\infty$ or $t_k\equiv 0$}\}.
\end{equation}

Besides, passing to a subsequence once again and applying Lemma \ref{Eq and orth profiles}, we may assume that for every $\alpha\ne\beta$,
either $N^\alpha_k/N^\beta_k+N^\beta_k/N^\alpha_k\to+\infty$ as $k\to+\infty$ or $N^\alpha_k=N^\beta_k$ for all $k$ and in this case, either $t^\alpha_k=t^\beta_k$ as $k\to+\infty$ or $(N^\alpha_k)^2\vert t^\alpha_k-t^\beta_k\vert \to+\infty$ as $k\to+\infty$.

Now for every linear profile, we define the associated nonlinear profile as the maximal solution of \eqref{NLS} with the same initial data. More precisely,
\begin{enumerate}
\item If $\mathcal{E}^\alpha$ is a Euclidean frame, this is given by Lemma \ref{GEForEP}. This nonlinear profile is noted $U^{E,\alpha}_k$.
\item If $\mathcal{S}^\beta$ is a large-scale profile, this is given by Lemma \ref{GEForLSP}. This nonlinear profile is noted $U^{ls,\beta}_k$.
\item If $\mathcal{O}^\gamma=(1,t^\gamma,x_k^\gamma)_k$ is a Scale-1 frame, then letting $W^\gamma$ be the maximal strong solution with initial data $\omega^\gamma\in H^1(\mathbb{R}\times\mathbb{T}^2)$ if $t^\gamma_k\equiv 0$, and  $W^\gamma$ be the unique nonlinear solution which scatters to $e^{it\Delta_{\mathbb{R}\times\mathbb{T}^2}}\omega^\gamma$ as $t\to\pm\infty$ if $t_k\to\mp\infty$, we have that
\begin{equation*}
U^{1,\gamma}_k(x,t)=W^\gamma(x-x^\gamma_k,t-t^\gamma_k).
\end{equation*}

\end{enumerate}

From \eqref{AdditionOfEnergies} and defining
\begin{equation*}
\begin{split}
L_E(\alpha)&:=\lim_{k\to+\infty}\left(E(\widetilde{\varphi}^\alpha_{\mathcal{E}^\alpha,k})+\frac{1}{2}M(\widetilde{\varphi}^\alpha_{\mathcal{E}^\alpha,k})\right)=E_{\mathbb{R}^3}(\varphi^\alpha)\in(0,L_{max}],\\
L_{ls}(\beta)&:=\lim_{k\to+\infty}\left(E(\widetilde{\psi}^\beta_{\mathcal{S}^\beta,k})+\frac{1}{2}M(\widetilde{\psi}^\beta_{\mathcal{S}^\beta,k})\right)=\Vert \psi^\beta\Vert_{H^{0,1}(\mathbb{R}\times\mathbb{T}^2)}^2\in(0,L_{max}],\\
L_1(\gamma)&:=\lim_{k\to+\infty}\left(E(\widetilde{\omega}^\gamma_{\mathcal{O}^\gamma,k})+\frac{1}{2}M(\widetilde{\omega}^\gamma_{\mathcal{O}^\gamma,k})\right)=E(W^\gamma)+\frac{1}{2}M(W^\gamma)\in(0,L_{max}],\\
\end{split}
\end{equation*}
we see that
\begin{equation}\label{SumOfL}
\begin{split}
\lim_{J\to+\infty}\big[\sum_{1\le\alpha,\beta,\gamma\le J}\left[L_E(\alpha)+L_{ls}(\beta)+L_{1}(\gamma)\right]+\lim_{k\to+\infty}L(R_k^J)\big]\le L_{max}.
\end{split}
\end{equation}
The numbers $L_E(\alpha)$, $L_{ls}(\beta)$, $L_1(\gamma)$ and $\lim_{k\to+\infty}L(R_k^J)$ are all well defined up to taking a subsequence. Up to relabelling the profiles, we can assume that for all $\alpha$, $L_\mu(1)\ge L_\mu(\alpha)$, $\mu\in\{E,ls,1\}$.

\medskip

{\bf Case IIa:} $L_E(1)=L_{max}$, there is only one Euclidean profile, that is
\begin{equation*}
u_k(0)=\widetilde{\varphi}_{\mathcal{E},k}+o_k(1)
\end{equation*}
in $H^1(\mathbb{R}\times\mathbb{T}^2)$ (see \eqref{SumOfL}), where $\mathcal{E}$ is a Euclidean frame. In this case, since from \eqref{ControlOnZNormForEP} the corresponding nonlinear profile $U_k$ satisfies
\begin{equation*}
\Vert U_k\Vert_{Z(\mathbb{R})}\lesssim_{E_{\mathbb{R}^3}(\varphi)} 1\quad\text{and}\quad\lim_{k\to +\infty}\Vert U_k(0)-u_k(0)\Vert_{H^1(\mathbb{R}\times\mathbb{T}^2)}\to 0
\end{equation*}
we may use Proposition \ref{stability} to deduce that
\begin{equation*}
\Vert u_k\Vert_{Z(\mathbb{R})}\lesssim \Vert u_k\Vert_{X^1(\mathbb{R})}\lesssim_{L_{max}} 1
\end{equation*}
which contradicts \eqref{CondForComp}.

\medskip

{\bf Case IIb:} $L_{ls}(1)=L_{max}$, there is only one large scale profile, that is
\begin{equation*}
u_k(0)=\widetilde{\psi}_{\mathcal{S},k}+o_k(1)
\end{equation*}
in $H^1$ (see \eqref{SumOfL}), where $\mathcal{S}$ is a large-scale frame. In this case, since from \eqref{ControlOnZNormForLSP} the corresponding nonlinear profile $U_k$ satisfies
\begin{equation*}
\Vert U_k\Vert_{Z(\mathbb{R})}\lesssim_{\Vert \psi\Vert_{H^{0,1}}} 1\quad\text{and}\quad\lim_{k\to +\infty}\Vert U_k(0)-u_k(0)\Vert_{H^1(\mathbb{R}\times\mathbb{T}^2)}\to 0
\end{equation*}
we may use Proposition \ref{stability} to deduce that
\begin{equation*}
\Vert u_k\Vert_{Z(\mathbb{R})}\lesssim \Vert u_k\Vert_{X^1(\mathbb{R})}\lesssim_{L_{max}} 1
\end{equation*}
which contradicts \eqref{CondForComp}.

\medskip

{\bf Case IIc:} $L_1(1)=L_{max}$ and, using again \eqref{SumOfL}, we have that
\begin{equation*}
u_k(0)=\widetilde{\omega}_{\mathcal{O},k}+o_k(1)\quad\text{ in }H^1(\mathbb{R}\times\mathbb{T}^2),
\end{equation*}
where $\mathcal{O}=\{(1,t_k,x_k,0)\}$ is a Scale-1 frame. If $t_k\equiv 0$, this is precisely conclusion \eqref{Compactness}. If $t_k\to+\infty$, then, we observe that
\begin{equation*}
\Vert e^{it\Delta_{\mathbb{R}\times\mathbb{T}^2}}\widetilde{\omega}_{\mathcal{O}_k}\Vert_{Z(a_k,0)}\le \Vert e^{it\Delta_{\mathbb{R}\times\mathbb{T}^2}}\widetilde{\omega}_{\mathcal{O}_k}\Vert_{Z(-\infty,0)}=\Vert e^{it\Delta_{\mathbb{R}\times\mathbb{T}^2}}\omega\Vert_{Z(-\infty,-t_k)}
\end{equation*}
which goes to $0$ as $k\to+\infty$. Using Proposition \ref{LWP}, we see that, for $k$ large enough,
\begin{equation*}
\Vert u_k\Vert_{Z(-\infty,0)}\le\delta_0
\end{equation*}
which again contradicts \eqref{CondForComp}. The case $t_k\to+\infty$ is similar.

\medskip

{\bf Case III:}
$L_\mu(1)<L_{max}$ for all $\mu\in\{E,ls,1\}$. In this case, we construct an approximate finite-$Z-$norm solution of $\eqref{NLS}$ with initial data $u_k(0)$ and derive a contradiction by invoking Proposition \ref{stability}. Since there exists $\eta>0$ such that for all $\alpha\ge 1$, $\mu\in\{E,ls,1\}$, $L_\mu(\alpha)<L_{max}-\eta$, we have that all nonlinear profiles are global and satisfy, for any $k$, $\alpha\ge 1$ and $\mu\in\{E,ls,1\}$ (after extracting a subsequence)
\begin{equation*}
\Vert U^{\mu,\alpha}_k\Vert_{Z(\mathbb{R})}\le 2\Lambda(L_{max}-\eta/2)\lesssim 1,
\end{equation*}
where from now on all the implicit constants are allowed to depend on $\Lambda(L_{max}-\eta/2)$.
 Using Proposition \ref{stability} it follows that
\begin{equation}\label{BddX1}
\|U^{\mu, \alpha}_k\|_{X^1(\mathbb{R})}\lesssim 1.
\end{equation}

For $J,k\geq 1$ we define
\begin{equation}\label{DefUProfJ}
U^J_{prof,k}:=\sum_{\alpha=1}^J\sum_{\mu\in\{E,ls,1\}} U^{\mu,\alpha}_k.
\end{equation}
We show first that there is a constant $Q\lesssim 1$ such that
\begin{equation}\label{bi1}
\Vert U^J_{prof,k}\Vert_{X^1}^2+\sum_{1\le\alpha\le J}\sum_{\mu\in\{E,ls,1\}}\|U^{\mu,\alpha}_k\|_{X^1}^2\leq Q^2,
\end{equation}
uniformly in $J$, for all $k$ sufficiently large. Let $\delta_0=\delta_0(2L_{max})$ defined in Proposition \ref{LWP}. From \eqref{SumOfL}, we know that there are only finitely many profiles such that $L(\alpha)\geq\delta_0/2$. Without loss of generality, we may assume that for all $\alpha\ge A$, $L(\alpha)\leq\delta_0$. Using \eqref{AdditionOfEnergies}, \eqref{BddX1}, and \eqref{SmalldataCCL} we then see that
\begin{equation*}
\begin{split}
\Vert U^J_{prof,k}\Vert_{X^1(\mathbb{R})}&=\Vert \sum_{1\le\alpha\le J}\sum_{\mu\in\{E,ls,1\}}U^{\mu,\alpha}_k\Vert_{X^1(\mathbb{R})}\\
&\le \Vert \sum_{1 \leq \alpha \leq A} \sum_{\mu\in\{E, ls, 1\}} U^{\mu, \alpha}_k\Vert_{X^1(\R)}+\Vert \sum_{A\le\alpha\le J}\sum_{\mu\in\{E,ls,1\}}(U^{\mu,\alpha}_k-e^{it\Delta}U^{\mu,\alpha}_k(0))\Vert_{X^1(\R)}\\
&+\Vert e^{it\Delta}\sum_{A\le\alpha\le J}\sum_{\mu\in\{E,ls,1\}}U^{\mu,\alpha}_k(0)\Vert_{X^1(\R)}\\
&\le \sum_{1\le\alpha\le A}\sum_{\mu\in\{E,ls,1\}}\Vert U^{\mu,\alpha}_k\Vert_{X^1(\R)}+\sum_{A\le\alpha\le J}\sum_{\mu\in\{E,ls,1\}}L_\mu(\alpha)^\frac{3}{2}\\
&+\Vert\sum_{A\le\alpha\le J}\sum_{\mu\in\{E,ls,1\}}U^{\mu,\alpha}_k(0)\Vert_{H^1(\mathbb{R}\times\mathbb{T}^2)}\\
&\lesssim A+\sum_{A\le\alpha\le J}\sum_{\mu\in\{E,ls,1\}}L_\mu(\alpha)^\frac{3}{2}+\Vert\sum_{A\le\alpha\le J}\sum_{\mu\in\{E,ls,1\}}U^{\mu,\alpha}_k(0)\Vert_{H^1(\mathbb{R}\times\mathbb{T}^2)}
\lesssim 1.
\end{split}
\end{equation*}
The bound on $\sum_{\al=1}^J\sum_{\mu\in\{E,ls,1\}}\|U^{\mu,\al}_k\|_{X^1(\R)}^2$ is similar (in fact easier), which gives \eqref{bi1}.

We are now ready to construct our approximate solution. Let $F(z)=|z|^4z$ and introduce the notation:
\begin{equation}\label{def of F'}
F'(G)u:= 3|G|^4u +2G^2|G|^2 \overline{u}
\end{equation}

For each $B$ and $J$, we define $g^{B,J}_k$ to be the solution of the initial value problem:

\begin{equation}\label{g eqn}
i\partial_t g +\Delta g - F'(U^B_{prof, k}) g =0, \;\;\; g(0)=R^J_k.
\end{equation}

The solution $g^{B,J}_k$ is well defined on $\R$ for $k \geq k_0(J)$ and satisfies:
\begin{equation}\label{g bound}
\|g^{B,J}_k\|_{X^1(\R)}\leq Q'.
\end{equation}
for some constant $Q'$ independent of $J$ and $B$. This follows by splitting $\R$ into $O(Q)$ intervals $I_j$ over which $\|U_{prof, k}^B\|_{Z(I_j)}$ is small and applying the local theory in Section \ref{Section-LWP} (in particular Lemma \ref{nl estimate lemma}) on each subinterval. 

We now define the approximate solution: for $A$ to be chosen shortly, we let 

\begin{equation}\label{U approx}
U_k^{app, J}=U^A_{prof,k} +g_k^{A, J}+U^{>A}_{prof, k} \;\;\; \hbox{ where }\;\;\; U^{>A}_{prof, k} =\sum_{A<\alpha\leq J} \sum_{\mu\in\{E, ls, 1\}} U^{\mu, \alpha}_k
\end{equation}
which has $u_k(0)$ as its initial data and satisfies, for any $1\leq A \leq J$, the bound:
$$
\|U_k^{app, J}\|_{X^1(\R)}\leq 3(Q+Q')
$$
for all $k \geq k_0(J)$. The stability result in Proposition \ref{stability} with $M = 6(1+Q+Q')$ gives us an $\epsilon_1=\epsilon_1(M) \leq \frac{1}{K(1+Q+Q')}$ for some $K$ sufficiently large, such that if the error term defined below in $\eqref{def of ekJ}$ satisfies $\|e_k^J \|_{N(\R)}\leq \epsilon_1$, then we can upgrade the uniform $X^1(\R)$ bounds on $U_k^{app, J}$ into a uniform bound on $\|u_k\|_{Z(\R)}$ thus deriving a contradiction. To do this, we first choose $A$ so that:
\begin{equation}\label{sAe}
\Vert U^{>A}_{prof, k}\|_{X^1(\R)}^2+\sum_{A+1\leq \alpha\leq J} \sum_{\mu\in\{E, ls, 1\}} \|U^{\mu, \alpha}_k\|^2_{X^1(\R)} \leq \epsilon_1^{10} \;\;\;\hbox{ for any $J\geq A$ and $k$ sufficiently large},
\end{equation}
which is possible by arguing as in \eqref{bi1}. With $A$ fixed as above, the rest of the proof depends on Lemma \ref{orthnlprof} and Lemma \ref{gZeroLimit} proved in the next subsection to bound the error term $e_k^J$ resulting from the approximate solution $U^{app, J}_k$ as follows:

\begin{align}
e_k^J =&(i\partial_t +\Delta_{\mathbb{R}\times\mathbb{T}^2})U_k^{app, J} -F(U^{app,J}_k)\label{def of ekJ}\\
=& -F(U^A_{prof, k} +g^{A, J}_k+ U^{>A}_{prof, k}) +\sum_{\substack{1\leq \alpha \leq J\\ \mu\in\{E, ls, 1\}}} F(U_k^{\mu, \alpha}) +F'(U^A_{prof,k})g^{A,J}_k\\
=& -F(U^A_{prof, k} +g^{A, J}_k+ U^{>A}_{prof, k}) +F(U^A_{prof,k} +g^{A, J}_k) +F(U^{>A}_{prof, k})\label{ekI}\\ 
&-F(U^A_{prof,k} +g^{A, J}_k) +F(U^A_{prof,k})+F'(U^A_{prof,k})g^{A,J}_k\label{ekII}\\
&-F(U^A_{prof,k})+\sum_{\substack{1\leq \alpha \leq A\\ \mu\in\{E, ls, 1\}}} F(U_k^{\mu, \alpha}) \label{ekIII}\\
&-F(U^{>A}_{prof, k})+\sum_{\substack{A+1\leq \alpha \leq J\\ \mu\in \{E, ls, 1\}}} F(U_k^{\mu, \alpha})\label{ekIV} 
\end{align}

By Lemma \ref{nl estimate lemma}, \eqref{sAe}, and \eqref{bi1}, we estimate:
\begin{equation*}
\Vert\eqref{ekI}\Vert_{N(\R)} \lesssim \left(\Vert U^A_{prof,k}+g^{A,J}_k\Vert_{X^1(\R)} +\Vert U^{>A}_{prof,k}\Vert_{X^1(\R)}\right)^4 \Vert U^{>A}_{prof,k}\Vert_{X^1(\R)}< \epsilon_1/4
\end{equation*}
for $k$ large enough. By Lemma \ref{nl estimate lemma}, Lemma \ref{gZeroLimit} we bound:\footnote{The last term in $\eqref{ekII}$ was introduced in $\eqref{g eqn}$ precisely to guarantee that only quadratic in $g^{A,J}_k$ factors appear in $\eqref{ekII}$; otherwise one is left with terms like $|U^{\mu,\alpha}|^4 g^{A,J}_k$ that cannot be bounded using Lemmas \ref{orthnlprof} and \ref{gZeroLimit}.}
\begin{equation*}
\Vert \eqref{ekII}\Vert_{N(\R)}\lesssim \left(\Vert U^A_{prof,k}\Vert_{X^1(\R)}+\Vert g^{A,J}_k\Vert_{X^1(\R)}\right)^4\|g^{A,J}_k\|_{Z'(\R)} \lesssim (Q+Q')^4\|g^{A,J}_k\|_{Z'(\R)}<\epsilon_1/4
\end{equation*}
if $J$ is sufficiently large and $k \geq k_0(J)$ for some $k_0(J)$. Turning to $\eqref{ekIII}$, by Lemma \ref{orthnlprof}:
\begin{equation*}
\Vert \eqref{ekIII}\Vert_{N(\R)}\lesssim_A \sum_{\substack{(\alpha_j, \mu_j)\in \{1, \ldots, A\}\times \{E, ls, 1\},\,\, j=1,2 \\(\alpha_1, \mu_1)\neq (\alpha_2, \mu_2)}} \Vert \widetilde U_k^{\mu_1, \alpha_1}\widetilde U_{k}^{\mu_2,\alpha_2}\prod_{i=3}^5 \widetilde U_k^{\mu_i, \alpha_i} \Vert_{N(\R)} \leq \epsilon_1/4
\end{equation*}
if $k$ is large enough. Finally, by $\eqref{sAe}$, 
\begin{align*}
\Vert \eqref{ekIV}\Vert_{N(\R)}\lesssim& \Vert F(U^{>A}_{prof,k})\Vert_{N(\R)}+\sum_{\substack{A<\alpha\leq J\\ \mu \in \{E, ls, 1\}}} \Vert F(U^{\mu, \alpha}_{prof,k})\Vert_{N(\R)}\\
\lesssim& \Vert U^{>A}_{prof,k}\Vert_{X^1(\R)}^5+\sum_{\substack{A<\alpha\leq J\\ \mu \in \{E, ls, 1\}}} \Vert U^{\mu, \alpha}_{prof,k}\Vert_{X^1(\R)}^5\lesssim \epsilon_1/4.
\end{align*}

The upshot is that $\Vert e_k^J \|_{N(\R)}< \epsilon_1$ if $J$ is sufficiently large and $k \geq k_0(J)$. Invoking the conclusion of Proposition \ref{stability} we get that $u_k$ extends as a solution in $X^1_c(\R)$ satisfying:
\begin{equation*}
\Vert u_k \Vert_{Z(\R)} < +\infty, 
\end{equation*}
which is in contradiction with $\eqref{CondForComp}$.
\end{proof}

We are thus left with proving the following two lemmas:

\begin{lemma}\label{orthnlprof}

Assume that $U_k^{\alpha}, U_k^{\beta}, U_k^{\gamma_1}, U_k^{\gamma_2}, U_k^{\gamma_3} $ are five nonlinear profiles from the set $\{U_k^{\mu, \alpha}:1\leq \alpha \leq A,  \mu\in \{E, ls, 1\}\}$ such that $U_k^\alpha$ and  $U_k^{\beta}$ correspond to orthogonal frames. Then:
\begin{equation}\label{onlp0}
\limsup_{k \to +\infty} \|\widetilde U_k^{\alpha} \widetilde U^{\beta}_k \prod_{i=1}^3 \widetilde U_k^{\gamma_i}\|_{N(\R)}=0 \;\;\;\hbox{ where, for $\delta \in \{\alpha, \beta, \gamma_1, \gamma_2, \gamma_3\}$, } \widetilde U_k^{\delta}\in \{U^{ \delta}_k, \overline{U^{\delta}_k}\} 
\end{equation}
\end{lemma}

\begin{lemma}\label{gZeroLimit}
For any fixed $A$, it holds that:
\begin{equation*}
\limsup_{J \to +\infty}\limsup_{k \to +\infty} \Vert g^{A,J}_k\Vert_{Z(\R)}=0.
\end{equation*}
\end{lemma}

\subsection{Proof of Lemma \ref{orthnlprof}}

We start with some notational simplifications: We will often use the notation $U^\alpha_k, U^\beta_k$ instead of $\widetilde U_{k}^{\alpha}, \widetilde U^{\beta}_k$ due to the fact that all our estimates are invariant under complex conjugation of any of the arguments. 

Based on the analysis in Lemma \ref{Eq and orth profiles}, we only need to consider non-linear profiles generated by renormalized frames. Recall that these were defined in \eqref{renframe}, \eqref{renframeLS}, \eqref{renframeS1} for Euclidean, large-scale, and scale-1 frames respectively. Note that any frame is equivalent to a re-normalized frame by Lemma \ref{Eq and orth profiles}.


Let $\theta>0$ be an arbitrary small parameter. If $\OO^\gamma$ is a scale-1 profile, then the nonlinear profile $U_k^\gamma=U_k^{s1, \alpha}$ associated to the linear profile $\widetilde \omega^\gamma_{\OO^\gamma, k}$ can be decomposed as follows: there exists parameters $T^\gamma$ and $R^\gamma$ depending on $\theta$ such that:

\begin{equation}\label{dec s1}
\begin{split}
U^{s1,\gamma}_k(t, x) &=\Pi_{(t_k^\gamma,p_k^\gamma)}\left[\mathbbm1_{t\in[T^\gamma, +\infty)} e^{it\Delta}\phi^\gamma_{+}+ \mathbbm 1_{t\in (-\infty, -T^\gamma)}e^{it\Delta}\phi^\gamma_-
 +V^{\gamma}+\rho^\gamma\right]\\
 &=\omega^{\gamma,+\infty}_k+\omega_k^{\gamma, -\infty}+ \omega^\gamma_k +\rho^\gamma_k\\
&\|\rho^\gamma\|_{X^1(\mathbb{R})} +\|e^{it\Delta}\phi^\gamma_{\pm}\|_{Z'(\pm [T^\gamma, +\infty))} + \Vert (1-P^x_{\geq (R^\gamma)^{-1}}P_{\leq R^\gamma})V^\gamma\Vert_{X^1(\R)} \leq \theta\\
&\Vert V^\gamma \Vert_{X^1(\R)}+\Vert\phi_{\pm}^\gamma\Vert_{H^1} \lesssim 1,\\
&|D_{x,y}^m V^\gamma| \lesssim R^\gamma \mathbbm 1_{S^\gamma_k} \quad \hbox{ for }0\leq m\leq 3\\
&\operatorname{Supp} V^\gamma \subset S_k^\gamma:= \{(x,y,t) \in \R\times \T^2\times\mathbb{R}: |t|\leq T^\gamma,\; |x|\leq R^\gamma\}, \; \quad \phi^{\gamma}_ {\pm} =P_{\leq R^\gamma}\phi^\gamma_{\pm}.
\end{split}
\end{equation}  

If $\mathcal{E}^\gamma$ is a normalized Euclidean profile, then the analysis in Section \ref{Euclidean profile analysis} allows us to find $R^\gamma$ and $T^\gamma$ depending on $\theta$ such that
\begin{equation}\label{dec E}
\begin{split}
U^{E,\gamma}_k(t, x) &=\Pi_{(t_k^\gamma, p_k^\gamma)}\left[\mathbbm1_{t\in[({N_k^\gamma})^{-2}T^\gamma, +\infty)}e^{it\Delta}T_{{N_k^\gamma}}^e\phi^\gamma_{+}+ \mathbbm 1_{t\in(-\infty, -{(N_k^\gamma)}^{-2}T^\gamma)}e^{it\Delta}T^e_{{N_k^\gamma}}\phi^\gamma_-+V_k^{\gamma}+\rho_k^\gamma\right]\\
&=\omega^{\gamma,+\infty}_k+\omega_k^{\gamma, -\infty}+ \omega_k^\gamma +\rho^\gamma_k\\
&\|\rho_k^\gamma\|_{X^1(\R)} +\|\omega_k^{\gamma, \pm \infty}\|_{Z'(\R)} + \Vert (1-P_{ N^\gamma_k(R^\gamma)^{-1}\leq |\nabla|\leq N_k^\gamma R^\gamma})V_k^\gamma\Vert_{X^1(\R)} \leq \theta\\
&\Vert V_k^\gamma \Vert_{X^1(\R)}+\Vert\phi_{\pm}^\gamma\Vert_{\dot H^1(\R^3)} \lesssim 1\\
&|D_{x,y}^m V^\gamma_k | \lesssim R^\gamma (N_k^\gamma)^{1/2+m}\mathbbm 1_{S^\gamma_k} \quad \hbox{ for }0\leq k\leq 3\\
&\operatorname{Supp} V_k^\gamma \subset S_k^\gamma:= \{(x,y,t) \in \R\times \T^2\times\R: |t|\leq {(N_k^\gamma)}^{-2}T^\gamma,\; |(x,y)|\leq {(N_k^\gamma)}^{-1}R^\gamma\} \\
& \phi^\gamma_{\pm}=P_{(R^\gamma)^{-1}\leq \cdot\leq R^\gamma} \tilde{\phi}^\gamma_{\pm},\quad\Vert\tilde{\phi}^\gamma_{\pm}\Vert_{L^1(\mathbb{R}^3)}\le R^\gamma.\\
\end{split}
\end{equation}  
Finally if $\mathcal{L}^\gamma$ is a large-scale profile, then the analysis in Section \ref{LS profile analysis} allows us to decompose:
\begin{equation}\label{dec ls}
\begin{split}
U^{ls,\gamma}_k(t, x) &=\Pi_{(t_k^\gamma, (x_k^\gamma,0))}\bigg[\mathbbm1_{t\in[({N_k^\gamma})^{-2}T^\gamma, +\infty)}e^{it\Delta}e^{i\xi_k .x}T_{{N_k^\gamma}}^{ls}\phi^\gamma_{+}+ \mathbbm 1_{t\in(-\infty, -{(N_k^\gamma)}^{-2}T^\gamma)}e^{it\Delta}e^{i\xi_k.x}T^{ls}_{{N_k^\gamma}}\phi^\gamma_-\\
&{}\quad+e^{i(\xi_k.x-|\xi_k|^2t)}V_k^{\gamma}+\rho_k^\gamma\bigg]=\omega^{\gamma,+\infty}_k+\omega_k^{\gamma, -\infty}+ \omega_k^\gamma +\rho^\gamma_k\\
&\|\rho_k^\gamma\|_{X^1(\R)} +\|\omega_k^{\gamma, \pm \infty}\|_{Z'(\R)} + \Vert (1-P_{ N^\gamma_k(R^\gamma)^{-1}\leq  |\partial_x|\leq N_k^\gamma R^\gamma}P_{|\nabla_y| \leq R^\gamma})V_k^\gamma\Vert_{X^1(\R)} \leq \theta\\
&\Vert V_k^\gamma \Vert_{X^1(\R)}+\Vert\phi_{\pm}\Vert_{H^{0,1}(\R\times \T^2)} \lesssim 1\\
&|D_x^mD_y^k V^\gamma_k| \lesssim R^\gamma (N_k^\gamma)^{1/2+m}\mathbbm 1_{S^\gamma_k} \quad \hbox{ for }0\leq m,k \leq 3\\
&\operatorname{Supp} V_k^\gamma \subset S_k^\gamma:=\{(x,y,t) \in \R\times \T^2\times\R: |t|\leq {(N_k^\gamma)}^{-2}T^\gamma,\; |x-2\xi_k^\gamma t|\leq {(N_k^\gamma)}^{-1}R^\gamma\} \\
& \phi^\gamma_{\pm}=P_{\leq R^\gamma} \phi^\gamma_{\pm},\\
\end{split}
\end{equation}
for some $R^\gamma$ and $T^\gamma$ depending on $\theta$.

The decomposition for Euclidean and large-scale profiles follow from Propositions \ref{GEForEP} and \ref{GEForLSP}. The further frequency and space support properties are either consequences of the latter propositions or are obtained a posteriori by truncating and incurring an error smaller than $\theta$ in $X^1(\R)$ which is absorbed in $\rho_k^\gamma$. In all three cases above, we denote by $\mathcal{S}_k^\gamma$ the translate of $S_k^\gamma$ by $(t_k^\gamma, p_k^\gamma)$, so that $\operatorname{supp} \omega_k^\gamma \subset \mathcal{S}_k^\gamma$.

\begin{proof}[Proof of Lemma \ref{orthnlprof}]
For each $\delta \in \{\alpha, \beta, \gamma_1, \gamma_2, \gamma_3\}$, we decompose $U_k^\delta$ as above with a smallness parameter $\theta$. Expanding the product $U_k^\alpha U_k^\beta \prod_{i=1}^3U_k^{\gamma_i}$ into its constituents, we directly notice that by Lemma \ref{nl estimate lemma} any monomial containing at least one factor of $\rho_k^\delta$ contributes $\lesssim \theta$ to $\eqref{onlp0}$, and similarly for any monomial with more than one factor of type $\omega_k^{\delta, \pm \infty}$.

There remains only two possibilities: either there are two different bulk terms $\omega^\alpha_k$ and $\omega^\beta_k$ and we can use Lemma \ref{onlp lemma1} below, or there are four copies of the same, up to conjugation, bulk (say $\omega^\alpha_k$) and a linear term $\omega^{\beta,\pm\infty}$ and we can use Lemma \ref{onlp lemma2} below. Since $\theta$ is arbitrary, we are thus only left with proving Lemmas \ref{onlp lemma1} and \ref{onlp lemma2}.

\end{proof}

\begin{lemma}\label{onlp lemma1}
With the notation as above, suppose that $f_k^{1}, f_k^2,$ and $f_k^3 \in X^1_c(\R)$ with unit $X^1$ norm. For any $\theta>0$ and $1\le \alpha,\beta\le A$ that correspond to orthogonal frames, there holds:
\begin{equation}\label{onlp 3}
\limsup_{k \to +\infty} \Vert \omega_k^\alpha \omega_k^\beta \prod_{j=1}^3f_k^j\Vert_{N(\R)}=0.
\end{equation} 
\end{lemma}

\begin{lemma}\label{onlp lemma2}
With the notation above, for any $\theta>0$, and any $1\le\alpha,\beta\le A$ that correspond to orthogonal frames,
\begin{equation}\label{onlp 4}
\limsup_{k \to +\infty} \Vert \mathfrak{O}_{4,1}(\omega_k^\alpha, \omega_k^{\beta, \pm \infty})\Vert_{N(\R)}=0.
\end{equation}
\end{lemma}

Recall that $\mathfrak{O}_{4,1}(a,b)$ represents an expression which is $4$-linear in $\{a,\overline{a}\}$ and linear in $\{b,\overline{b}\}$.

\begin{proof}[Proof of Lemma \ref{onlp lemma1}]
Recall that by Lemma \ref{Eq and orth profiles} we may assume that either $N_k^\alpha / N_k^\beta \to +\infty$ or $N_k^\alpha=N_k^\beta$. In the latter case, either $(N_k^\alpha)^{2}|t_k^\alpha-t_k^\beta|+(N_k^\alpha)^{-1}|\xi_k^\alpha -\xi_k^\beta| \to +\infty$ or $t_k^\alpha=t_k^\beta$ and $\xi_k^\alpha=\xi_k^\beta$.

We fix $\kappa>0$.
We start by assuming that $N_k^\alpha/N_k^\beta\to+\infty$ and let $I^\alpha_k=(t^\alpha_k-T^\alpha(N^\alpha_k)^{-2},t^\alpha_k+T^\alpha(N^\alpha_k)^{-2})$. Assume also that $\omega^\beta_k$ corresponds to a large scale profile or a scale-1 profile. In this case, there exists $D=D(\kappa)$ such that
\begin{equation*}
\Vert P_{\ge D}\omega^\beta_k\Vert_{X^1(\R)}\le \kappa
\end{equation*}
and therefore,
\begin{equation*}
\Vert \omega^\alpha_k(P_{\ge D}\omega^\beta_k)f^1_kf^2_kf^3_k\Vert_{N(I_\alpha)}\le\kappa.
\end{equation*}
Independently, using \eqref{nl refined} we also see that
\begin{equation*}
\begin{split}
\Vert \omega^\alpha_k(P_{\le D}\omega^\beta_k)f^1_kf^2_kf^3_k\Vert_{N(I_\alpha)}
&=\Vert \omega^\alpha_k(\mathbbm{1}_{I^\alpha_k}P_{\le D}\omega^\beta_k)f^1_kf^2_kf^3_k\Vert_{N(I^\alpha_k)}\lesssim_{\theta,\kappa} \Vert \mathbbm{1}_{I^\alpha_k}\omega^\beta_k\Vert_{Z(\mathbb{R})}\\
&\lesssim_{\theta,\kappa} o_k(1)
\end{split}
\end{equation*}
where the last line follows from the estimate in the fifth line in \eqref{dec s1} and \eqref{dec ls}.

It remains to consider the case when both profiles are Euclidean profiles. In this case, there exists $D=D(\kappa)$ such that
\begin{equation}\label{onlp 5}
\Vert P_{\leq D^{-1}N_k^\alpha}\omega_k^\alpha\Vert_{X^1(\R)} \lesssim \kappa \hbox{ whereas } \Vert P_{\geq D N_k^\beta}\omega_k^\beta\Vert_{X^1(\R)}\lesssim \kappa
\end{equation}
and using $\eqref{onlp 5}$ and $\eqref{nl refined}$:
\begin{align*}
\Vert \omega_k^\alpha \omega_k^\beta \prod_{j=1}^3f_k^j\Vert_{N(\R)} &\le&& \Vert (P_{\leq D^{-1}N_k^\alpha} \omega_k^\alpha) \omega_k^\beta \prod_{j=1}^3f_k^j\Vert_{N(\R)}+\Vert (P_{\geq D^{-1}N_k^\alpha} \omega_k^\alpha) \cdot (P_{\geq D N_k^\beta}\omega_k^\beta) \prod_{j=1}^3f_k^j\Vert_{N(\R)}\\
&&&+\Vert (P_{\geq D^{-1}N_k^\alpha} \omega_k^\alpha)\cdot (P_{\leq D N_k^\beta} \omega_k^\beta) \prod_{j=1}^3f_k^j\Vert_{N(\R)}\\
&\le&& 2\kappa+\Vert (P_{\geq D^{-1}N_k^\alpha} \omega_k^\alpha)\cdot (\mathbbm{1}_{I^\alpha_k}P_{\leq D N_k^\beta} \omega_k^\beta) \prod_{j=1}^3f_k^j\Vert_{N(\R)},
\end{align*}
and once again, since for $k$ large enough, $D^{-1}N^\alpha_k>> DN^\beta_k$, using \eqref{nl refined} and the fifth line in \eqref{dec E},
\begin{equation*}
\Vert (P_{\geq D^{-1}N_k^\alpha} \omega_k^\alpha)\cdot (\mathbbm{1}_{I^\alpha_k}P_{\leq D N_k^\beta} \omega_k^\beta) \prod_{j=1}^3f_k^j\Vert_{N(\R)}\lesssim_\theta \Vert \mathbbm{1}_{I^\alpha_k} \omega_k^\beta\Vert_{Z(\mathbb{R})}=o_k(1).
\end{equation*}
This finishes the proof of the case $N_k^\alpha/N_k^\beta \to +\infty$. 

Now assume that $N_k^\alpha=N_k^\beta$. If $N_k^2|t_k^\alpha-t_k^\beta| \to +\infty$, then by inspecting the time supports, we get that $\omega_k^\alpha \omega_k^\beta=0$ for large enough $k$. As a result, to continue we may assume $t_k^\alpha\equiv t_k^\beta$.

Assume now that we have two large scale profiles and that $(N_k^\alpha)^{-1} |\xi_k^\alpha -\xi_k^\beta| \to +\infty$. In this case, if $(x,y,s)\in \mathcal{S}^\alpha_k\cap \mathcal{S}^\beta_k$, then
\begin{equation*}
(N^\alpha_k)^{-1}R^\alpha\ge\vert x-2\xi^\alpha_ks\vert
=\vert x-2\xi^\beta_ks-2(\xi^\alpha_k-\xi^\beta_k)s\vert\ge N^\alpha_k\vert s\vert\cdot\left[(N^\alpha_k)^{-1}\vert\xi^\alpha_k-\xi^\beta_k\vert\right]-(N^\alpha_k)^{-1}R^\beta_k.
\end{equation*}
as a result, we see that there exists an interval $I_k$ of length satisfying $(N^\alpha_k)^2\vert I_k\vert\to0$ as $k\to+\infty$ such that
\begin{equation*}
\omega^\alpha_k\omega^\beta_k=(\mathbbm{1}_{I_k}\omega^\alpha_k)\cdot(\mathbbm{1}_{I_k}\omega^\beta_k),
\end{equation*}
and using the fifth line in \eqref{dec ls}, we have that
\begin{equation*}
\Vert \mathbbm{1}_{I_k}\omega^\alpha_k\Vert_{Z(\mathbb{R})}+\Vert\mathbbm{1}_{I_k}\omega^\beta_k\Vert_{Z(\mathbb{R})}\to0\hbox{ as }k\to+\infty.
\end{equation*}
Using Lemma \ref{nl refined}, we deduce that
\begin{equation*}
\Vert \omega^\alpha_k\omega^\beta_k f^1_kf^2_kf^3_k\Vert_{N(\mathbb{R})}\to0.
\end{equation*}
Finally, we are left with the case when $N_k^\alpha\equiv N_k^\beta, t_k^\alpha\equiv t_k^\beta, $ and $\xi_k^\alpha\equiv \xi_k^\beta$, but $(N_k^\alpha)|x_k^\alpha-x_k^\beta|\to +\infty$. Here, $\eqref{onlp 3}$ follows directly since $\omega_k^\alpha \omega_k^\beta =0$ for large enough $k$. Since $\kappa>0$ was arbitrary, the proof is complete.
\end{proof}

\begin{proof}[Proof of lemma \ref{onlp lemma2}]

Using that
\begin{equation*}
\nabla \mathfrak{O}_{4,1}(\omega^\alpha_k,\omega^{\beta,\pm\infty}_k)=\mathfrak{O}_{3,1,1}(\omega^\alpha_k,\nabla\omega^\alpha_k,\omega^{\beta,\pm\infty}_k)+\mathfrak{O}_{4,1}(\omega^\alpha_k,\nabla\omega^{\beta,\pm\infty}_k),
\end{equation*}
we see that
\begin{equation}\label{Estim0.1CoreLin}
\begin{split}
&\Vert \mathfrak{O}_{4,1}(\omega^\alpha_k,\omega^{\beta,\pm\infty}_k)\Vert_{L^1_t(I^\alpha_k,H^1)}\\
&\lesssim_\theta (N^\alpha_k)^\frac{1}{2}\left[(1+N^\alpha_k)\Vert\mathfrak{O}_{1,1}(\omega^\alpha_k,\omega^{\beta,\pm\infty}_k)\Vert_{L^2_{x,y,t}(\mathcal{S}^\alpha_k)}+\Vert \mathfrak{O}_{1,1}(\omega^\alpha_k,\nabla\omega^{\beta,\pm\infty}_k)\Vert_{L^2_{x,y,t}(\mathcal{S}^\alpha_k)}\right].
\end{split}
\end{equation}
where  $I^\alpha_k=(t^\alpha_k-T^\alpha(N^\alpha_k)^{-2},t^\alpha_k+T^\alpha(N^\alpha_k)^{-2})$. Since we also have that
\begin{equation}\label{Estim1CoreLin}
\begin{split}
\Vert \partial^p\omega^{\beta,\pm\infty}_k\Vert_{L^\infty_{x,y,t}}\lesssim_\theta (1+(N^\beta_k)^{\vert p\vert})(N^\beta_k)^\frac{1}{2},
\quad \Vert \omega^\alpha_k\Vert_{L^2_{x,y,t}}&\lesssim_\theta
\begin{cases}
(N^\alpha_k)^{-2}&\hbox{ if }N^\alpha_k\ge1\\
(N^\alpha_k)^{-1}&\hbox{ if }N^\alpha_k\to0.
\end{cases}\\
\end{split}
\end{equation}
we get that when $N^\alpha_k/N^\beta_k\to+\infty$,
\begin{equation*}
\limsup_{k\to+\infty}\Vert \mathfrak{O}_{4,1}(\omega^\alpha_k,\omega^{\beta,\pm\infty}_k)\Vert_{L^1_tH^1_{x,y}}=0.
\end{equation*}
Assume now that $N^\alpha_k\equiv N^\beta_k$. If $(N^\alpha_k)^2\vert t^\alpha_k-t^\beta_k\vert\to+\infty$, we may simply use the $t^{-1/2}$ decay of the propagator $e^{it\Delta}$ and replace the first bound in \eqref{Estim1CoreLin} by
\begin{equation*}
\Vert \partial^p\omega^{\beta,\pm\infty}_k\Vert_{L^\infty_{t}(I^\alpha_k:L^\infty_{x,y})}\lesssim_\theta (1+(N^\alpha_k)^{\vert p\vert})(N^\alpha_k)^\frac{1}{2}\left[ (N^\alpha_k)^2\vert t^\alpha_k-t^\beta_k\vert\right]^{-1/2}
\end{equation*}
and conclude similarly. If $t^\alpha_k\equiv t^\beta_k$ but $(N^\alpha_k)^{-1}\vert\xi^\alpha_k-\xi^\beta_k\vert+N^\alpha_k\vert (x^\alpha_k-x^\beta_k)-2(\xi_k^\alpha-\xi_k^\beta)t_k^\alpha|\to+\infty$, then, it is easy to see that
\begin{equation*}
\mathbbm{1}_{\mathcal{S}^\alpha_k}\mathfrak{O}_{4,1}(\omega^\alpha_k,\omega^{\beta,\pm\infty}_k)\to0\hbox{ in }L^1_t(\mathbb{R},H^1_{x,y}(\mathbb{R}\times\mathbb{T}^2)).
\end{equation*}

It remains to consider the case $N^\beta_k/N^\alpha_k\to+\infty$.  Fix $\kappa>0$. There exists $\delta=\delta(\kappa,\theta)>0$ such that
\begin{equation*}
\Vert P_{\le\delta N^\beta_k}^x\omega^{\beta,+\infty}_k\Vert_{X^1(\R)}\le\kappa
\end{equation*}
so this gives an acceptable contribution. Similar considerations hold for $(-\infty,0)$. Independently, estimating $\omega^\alpha_k$ in $L^\infty_{x,y,t}$ in \eqref{Estim0.1CoreLin} and using that
\begin{equation*}
\Vert P^x_{\ge\delta N^\beta_k}\nabla\omega^{\beta,\pm\infty}_k\Vert_{L^\infty_{x}L^2_{y,t}}\lesssim \delta^{-1/2}(N^\beta_k)^{-1/2},
\end{equation*}
which follows from:
\begin{equation}\label{1D smoothing}
\|P_{\geq K} e^{it\Delta}\phi\|_{L_x^\infty L_{t,y}^2}\lesssim K^{-1/2} \|\phi\|_{L^2(\R\times \T^2)}
\end{equation}
(Plancherel in $(y,t)$ and change variables, cf. (2.13) in \cite{IoPa}), we also find
\begin{equation*}
\Vert \mathfrak{O}_{4,1}(\omega^\alpha_k,P^x_{\ge \delta N^\beta_k}\omega^{\beta,\pm\infty}_k)\Vert_{L^1_t(I^\alpha_k,H^1)}\lesssim_{\kappa,\theta}\delta^{-3/2}(N^\alpha_k/N^\beta_k)^\frac{1}{2}
\end{equation*}
Since $\kappa>0$ is arbitrary, this finishes the proof.
\end{proof}

\subsection{Proof of Lemma \ref{gZeroLimit}}
\begin{proof}
We would like to show that for any fixed $A$ the solution $g_k^{A, J}$ of:
\begin{equation}\label{g eqn}
(i\partial_t +\Delta_{\mathbb{R}\times\mathbb{T}^2})g_k^{J,A}-F'(U_{prof,k}^A)g_k^{A,J}=0, \quad g_k^{A,J}(0)=R_k^{J}
\end{equation}
with $F(G)u$ being defined in $\eqref{def of F'}$ and $U^A_{prof, k}$ in \eqref{DefUProfJ}, satisfies:
\begin{equation}\label{goal 0.5}
\limsup_{J \to +\infty}\limsup_{k \to +\infty} \Vert g_k^{A,J}\Vert_{Z(\R)}=0.
\end{equation}

The key point here is that $R_k^{J}$ satisfies the following:
\begin{equation}\label{R Z small}
\sup_{J,k} \Vert R_k^J\Vert_{H^1(\R\times\T^2)}\lesssim 1 \quad \hbox{ and } \limsup_{J \to +\infty}\limsup_{k \to +\infty}\Vert e^{it\Delta_{\mathbb{R}\times\mathbb{T}^2}}R_k^{J}\Vert_{Z(\R)}=0.
\end{equation}

Throughout this proof we will keep using the following claim that follows easily from a simple analysis similar to that in Proposition \ref{LWP}:

{\bf Claim:} Suppose that for each $1\leq i\leq 4$, $u_i, u^\prime_i$ satisfy:
\begin{equation*}
\sum_{i=1}^4 \Vert u_i \Vert_{X^1(\R)} +\Vert u^\prime_i\Vert_{X^1(\R)} \leq C, \quad \|h\|_{N(\R)}<+\infty
\end{equation*}
then for any $g_0\in H^1(\R\times \T^2)$, the solution $g(t)$ of the initial value problem:
$$
(i\partial_t +\Delta)g +\left(\prod_{i=1}^4 u_i \right)g+\left(\prod_{i=1}^4 u^\prime_i\right) \overline g=h, \quad g(0)=g_0
$$
exists for all time in $X^1_c(\R)$ and satisfies:
\begin{equation}\label{energy est 0.5}
\|g\|_{X^1(\R)}\lesssim_C \|g_0\|_{H^1(\R\times \T^2)}+\|h\|_{N(\R)}.
\end{equation}

As a result of this and the fact that $U_k^\alpha$ are uniformly bounded in $X^1_c(\R)$, we obtain that $g_k^{A,J}$ is well-defined on $\mathbb{R}$ and satisfies:
\begin{equation}\label{X^1 g bound}
\sup_{J,k \in \N} \|g_k^{A,J}\|_{X^1(\R)} \lesssim 1.
\end{equation}

Let $\theta>0$ be given. We will prove the lemma by showing that the left-hand side of $\eqref{goal 0.5}$ is $\lesssim \sqrt{\theta}$. This will be accomplished in three steps:

\medskip
{\bf Step 1: (Reduction to profile cores)} With $\theta$ fixed as above,  we decompose each profile $U_k^\alpha$ for $1\leq \alpha \leq A$ as $\eqref{dec s1}$, \eqref{dec E}, or \eqref{dec ls} with smallness parameter $\theta$. Let $h_k^J$ be the solution of the initial value problem:
\begin{equation}\label{h eqn}
(i\partial_t +\Delta_{\mathbb{R}\times\mathbb{T}^2})h_k^{J}-\sum_{1\leq \alpha \leq A} F^\prime(\omega_k^{\alpha}) h_k^{J}=0, \quad h(0)=R_k^J.
\end{equation}
By the above claim we have that $h_k^J\in X_c^{1}(\R)$ and satisfies:
$$
\sup_{J,k}\|h_k^{J}\|_{X^1(\R)} \lesssim 1.
$$
We will reduce our problem to proving:
\begin{equation}\label{goal modified}
\limsup_{J \to +\infty}\limsup_{k \to \infty}\|h_k^J\|_{Z(\R)}\lesssim \theta.
\end{equation}
In fact, assuming $\eqref{goal modified}$, we show that:
\begin{equation}\label{h and g}
\limsup_{J \to +\infty} \limsup_{k \to +\infty} \|g_k^{A, J}-h_k^J\|_{X^1(\R)}\lesssim \theta^{1/2}.
\end{equation}
To see this, define $\zeta_k=g_k^{A,J}-h_k^J$ which satisfies the following equation:
\begin{align*}
(i\partial_t +\Delta_{\mathbb{R}\times\mathbb{T}^2})\zeta_k=& F'(U^A_{prof, k}) \zeta_k +\left (F'(U^A_{prof,k})h_k^{J}-\sum_{1\leq \alpha \leq A} F'(U_k^\alpha)h_k^{J}\right)\\
&+\sum_{1\leq \alpha \leq A}\left( F'(U_k^\alpha)h_k^J -F'(\omega_k^\alpha)h_k^J\right)
\end{align*}

Estimate $\eqref{h and g}$ would follow directly (for $k$ large enough) using $\eqref{energy est 0.5}$ once we show that:
\begin{align}
\label{1rst term l0.5}\limsup_{k \to  +\infty} \left \|F'(U^A_{prof,k})h_k^{J}-\sum_{1\leq \alpha \leq A} F'(U_k^\alpha)h_k^{J}\right \|_{N(\R)}\lesssim \theta\\
\label{2nd term l0.5}\limsup_{J \to +\infty}\limsup_{k \to  +\infty}\left\|\sum_{1\leq \alpha \leq A}\left( F'(U_k^\alpha)h_k^J -F'(\omega_k^\alpha)h_k^J\right)\right\|_{N(\R)}\lesssim \theta.
\end{align}
To prove \eqref{1rst term l0.5}, expand
\begin{equation*}
F'(U^A_{prof,k})h_k^{J}-\sum_{1\leq \alpha \leq A} F'(U_k^\alpha)h_k^{J}=\sum_{1\le \alpha,\beta,\gamma,\delta\le A}\tilde{U}^\alpha_k\tilde{U}^\beta_k\tilde{U}^\gamma_k\tilde{U}^\delta_k h^J_k,
\end{equation*}
where the first two profiles differ by more than complex conjugation, then expand each profile as in \eqref{dec s1}, \eqref{dec E}, \eqref{dec ls} with smallness parameter $\theta$. The term involving $\omega^\alpha_k\omega^\beta_k$ is then controlled using Lemma \ref{onlp lemma1}, while the bound for the other terms follows from \eqref{nl estimate} since there are at least two terms which are small in $Z^\prime$.
Similarly, $\eqref{2nd term l0.5}$ follows from \eqref{nl estimate} and the $Z'$-norm smallness of the asymptotic components $\omega_k^{\alpha, \pm \infty}$ which give:
$$
\sum_{1\leq \alpha \leq A}\left\| F'(U_k^\alpha)h_k^J -F'(\omega_k^\alpha)h_k^J\right\|_{N(\R)}\lesssim \theta + \|h_k^J\|_{Z'(\R)}
$$
which is $\lesssim \theta^{1/2}$ once we prove \eqref{goal modified}.

\medskip

{\bf Step 2 (Contribution of low frequencies and high-frequency smoothing directions)}

Let $\kappa>0$ be a small dyadic number to be chosen later in terms of $\theta$. We will decompose $R_k^J$ into three components: a low frequency part, a high frequency smoothing part, and a high frequency non-smoothing part:
\begin{equation*}
R_k^J= P_{\leq \kappa^{-1}} R_k^J + \widetilde{P}^1_\kappa P_{> \kappa^{-1}} R_k^J + (1-\widetilde P^1_\kappa)P_{>\kappa^{-1}} R_k^J
\end{equation*}
where $\widetilde P^1_{\kappa}$ was defined in \eqref{def of tilde P}. In this step we will estimate the contribution of first two terms above: Let $\sigma_k^{J, \kappa}$ be the solution to the initial value problem:

\begin{equation}\label{sigma equation}
(i\partial_t +\Delta_{\mathbb{R}\times\mathbb{T}^2})\sigma_k^{J, \kappa} +\sum_{1\leq \alpha \leq A}F'(\omega_k^\alpha) \sigma_k^{J,\kappa}=0, \quad \sigma_k^{J, \kappa}(0)=R_k^{J,0} :=P_{\leq \kappa^{-1}} R_k^J + \widetilde{P}^1_\kappa P_{> \kappa^{-1}} R_k^J.
\end{equation}

As in \eqref{energy est 0.5}, we have that $\sigma_k^{J, \kappa}$ is well-defined and satisfies $\|\sigma_k^{J,\kappa}\|_{X^1(\R)}\lesssim 1$. Applying $\eqref{energy est 0.5}$ to the difference equation between $\eqref{sigma equation}$ and the linear homogeneous Schr\"odinger equation, we get:
\begin{equation}\label{pcl 000}
\|\sigma_k^{J,\kappa}-e^{it\Delta}R_k^{J,0}\|_{X^1(\R)} \lesssim \sum_{1\leq \alpha \leq A} \| F^\prime (\omega_k^\alpha)e^{it\Delta}\left(P_{\leq \kappa^{-1}\langle N^\alpha_k\rangle} R_k^{J,0} + P_{> \kappa^{-1} \langle N^\alpha_k\rangle} R_k^{J,0} \right)\|_{N(\R)}.
\end{equation}
Using \eqref{nl refined}, we directly see that for any $1\le\alpha\le A$,
\begin{equation}\label{pal 00100}
\begin{split}
\limsup_{J\to+\infty}\limsup_{k\to+\infty}\| F^\prime (\omega_k^\alpha)e^{it\Delta}P_{\leq \kappa^{-1}\langle N^\alpha_k\rangle} R_k^{J,0}\|_{N(\R)}&\lesssim_\kappa \limsup_{J\to+\infty}\limsup_{k\to+\infty}\Vert e^{it\Delta}P_{\leq \kappa^{-1}\langle N^\alpha_k\rangle} R_k^{J,0}\Vert_{Z^\prime(\R)}\\
&=0.
\end{split}
\end{equation}

We now bound the contribution of $P_{> \kappa^{-1}\langle N^\alpha_k\rangle} R_k^{J,0}=\widetilde P^1_\kappa P_{> \kappa^{-1}\langle N^\alpha_k\rangle} R_k^J$: Using the convention that $\nabla^1 h := |h|+|\nabla h|$, we estimate:
\begin{equation}\label{smoothing contr}
\begin{split}
\|F'(\omega_k^\alpha) e^{it\Delta} P_{>\kappa^{-1}\langle N^\alpha_k\rangle}R_k^{J,0} \|_{N(\R)}& \lesssim \|(\omega_k^\alpha)^4 \nabla^1\left( e^{it\Delta} P_{>\kappa^{-1}\langle N^\alpha_k\rangle}R_k^{J,0}\right) \|_{L_t^1 L_{x,y}^2}\\
&+ \|(\omega_k^\alpha)^3 \nabla \omega_k^\alpha e^{it\Delta}P_{>\kappa^{-1}\langle N^\alpha_k\rangle}R_k^{J,0} \|_{L_t^1 L_{x,y}^2}\\
&\lesssim \|(\omega_k^\alpha)^4\|_{L_t^2 L_{x,y}^\infty}\|\mathbbm 1_{\mathcal{S}^\alpha_k} \nabla^1\left( e^{it\Delta} P_\kappa^1 P_{>\kappa^{-1}\langle N^\alpha_k\rangle}R_k^J\right)\|_{ L^2_{x,y,t}(\mathcal{S}^\alpha_k)}\\
&+ \kappa\langle N^\alpha_k\rangle^{-1} \|(\omega_k^\alpha)^3 \nabla \omega_k^\alpha\|_{L_t^2 L^\infty_{x,y}}\|\nabla^1\left( e^{it\Delta} P_\kappa^1 P_{>\kappa^{-1}\langle N^\alpha_k\rangle}R_k^J\right)\|_{L_{x,y,t}^2(\mathcal{S}^\alpha_k)}\\
&\lesssim_\theta N^\alpha_k\|\mathbbm 1_{\mathcal{S}^\alpha_k} \nabla^1\left( e^{it\Delta} P_\kappa^1 P_{>\kappa^{-1}\langle N^\alpha_k\rangle}R_k^J\right)\|_{ L^2_{x,y,t}} \lesssim_\theta \kappa^{\frac{1}{100}},
\end{split}
\end{equation}
which follows from Lemma \ref{smoothing lemma} at the end of this section. Using again \eqref{R Z small} we conclude after combining \eqref{pcl 000}, \eqref{pal 00100} and \eqref{smoothing contr} that, whenever $\kappa \leq \kappa_1(\theta)$,
\begin{equation}\label{step 2 conclusion}
\|\sigma_k^{J, \kappa}\|_{Z(\R)}\leq C(\theta,A) \kappa^{\frac{1}{100}} \leq \theta.
\end{equation}

\medskip

{\bf Step 3: (Contribution of high-frequency non-smoothing directions)} Let $f_k^{J, \kappa}$ denote the solution of the initial value problem:
\begin{equation}\label{f_k equation}
(i\partial_t +\Delta_{\mathbb{R}\times\mathbb{T}^2})f_k^{J,\kappa} -\sum_{\alpha=1}^AF'(\omega_k^\alpha) f_k^{J, \kappa}=0, \quad f_k^{J, \kappa}(0)=R_k^{J,1}:= (1-\widetilde P_\kappa^1)P_{\geq \kappa^{-1}}R_k^J
\end{equation}

We would like to show that if $k$ is large enough and $J \geq A$:
\begin{equation}\label{step 3 goal}
\|f_k^{J,\kappa}\|_{Z(\R)} \lesssim \theta
\end{equation}
whenever $\kappa \leq \kappa_2(\theta)$. This combined with $\eqref{step 2 conclusion}$ and the fact that $h_k^J=\sigma_k^{J, \kappa}+f_k^{J,\kappa}$ gives $\eqref{goal modified}$ thus finishing the proof. 

The idea behind the proof of \eqref{step 3 goal} is to exploit the narrowness  of the angular aperture of the Fourier support of the initial data $R_k^{J,1}$ represented by the smallness of $\kappa$. The key observation is that the smallness  of $\kappa$ implies that of $\|e^{it\Delta}R_k^{J,1}\|_{Z(\R)}$, which essentially follows from \eqref{Strich M refined U^p lemma}. Translating this information from the homogeneous linear Schr\"odinger equation to the inhomogeneous one \eqref{f_k equation} is done via a slightly delicate perturbative argument using an appropriately modified $X^1$-type norm that-unlike $X^1$- captures the gain reflected in $\eqref{Strich M refined U^p lemma}$. We now move to the details:

For dyadic numbers $1\leq M \leq N$, we define:
\begin{align*} 
p_{N,M}(\xi)=&\left[\eta^3(\xi/2N)-\eta^3(\xi/N)\right] \left[ \eta^1(\xi_1/2M)-\eta^1(\xi_1/M)\right] \quad \hbox{if }M\geq 2\\
p_{N,1}(\xi)=&\left[\eta^3(\xi/N)-\eta^3(\xi/2N)\right] \eta^1(\xi_1/2)
\end{align*}
and the corresponding operators:
$$
\FF P_{N,M} f= p_{N, M} \FF f, \quad P_{N, \leq K}=\sum_{1\leq M \leq \min(K, N)} P_{N,M}.
$$

Given $\rho\in (0,1]$, we define the weights:
\begin{equation}\label{DefRho}
c_{\rho, N, M}:=\left[M/N +\rho+(1+\rho N)^{-1}\right]^{\delta_0/100},\quad c_{\rho,N,M}\simeq
\begin{cases}
1&\hbox{ if }N\le\rho^{-1}\\
\max(M/N,1/(N\rho))^\frac{\delta_0}{100}&\hbox{ if }\rho^{-1}\le N\le\rho^{-2}\\
\max(M/N,\rho)^\frac{\delta_0}{100}&\hbox{ if }N\ge\rho^{-2}.
\end{cases}
\end{equation}
where $\delta_0$ is the same constant as in $\eqref{extra refinement}$. For functions $f\in X^1(I)$ on intervals $I \subset \R$ and $g\in H^1(\R\times \T^2)$, we defined the norms:
\begin{equation}\label{X^1_rho}
\begin{split}
\|f\|_{\widetilde X^1_\rho(I)}^2:=& \sum_{N\geq 1}\sum_{1\leq M \leq N}c_{\rho, N, M}^2\left \|  P_{N,M} f \right\|_{X^1(I)}^2\\
\|g\|_{\widetilde H^1_\rho(\R\times \T^2)}^2:= &\sum_{N\geq 1}\sum_{1\leq M \leq N}c_{\rho, N, M}^2\left \| P_{N,M} g \right\|_{H^1(\R\times \T^2)}^2
\end{split}
\end{equation} 
From \eqref{DefRho}, we see that
\begin{equation}\label{X^1_rho and X^1}
\begin{split}
\rho^{\delta_0/100}\|f\|_{X^1(I)} \lesssim& \|f\|_{\widetilde X^1_\rho(I)} \lesssim \|f\|_{X^1(I)}\\
\rho^{\delta_0/100}\|g\|_{H^1(\R\times \T^2)} \lesssim& \|g\|_{\widetilde H_\rho^1(\R\times \T^2)} \lesssim \|g\|_{H^1(\R\times \T^2)}
\end{split}
\end{equation}
It is obvious that the $\widetilde X^1_\rho$ satisfies:
$$
\|e^{it\Delta_{\mathbb{R}\times\mathbb{T}^2}} \phi\|_{\widetilde X^1_\rho(I)}\lesssim \|\phi\|_{\widetilde H^1_\rho(\R\times \T^2)}
$$
for any interval $I$, from which follows the smallness of $\|e^{it\Delta} (1-\tilde{P}^1_\kappa) P_{\geq \kappa^{-1}} \phi\|_{\widetilde X^1_\rho}$ for any $\kappa \leq \rho$. This will be particularly useful since $\widetilde X^1_\rho$ controls the $Z$ norm: for any $\rho \in (0,1]$:
\begin{equation}\label{tilde X and Z}
\|f\|_{Z(I)}\lesssim \|f\|_{\widetilde X^1_\rho(I)}.
\end{equation}
In fact, \eqref{Strich M refined U^p lemma} and Lemma \ref{U^p lemma} imply that for $p_0 \in \{\frac{9}{2}, 18\}$:
\begin{equation}\label{Step 301}
N^{\frac{5}{p_0}-\frac{1}{2}}\|P_{N,M}f\|_{l^{\frac{4p_0}{p_0-2}}_\gamma L^{p_0}_{x,y,t}(\R\times \T^2\times I)}\lesssim \left(\frac{M}{M+N}\right)^{1/20}N\|P_{N,M} f\|_{U^{9/2}_\Delta(I)},
\end{equation}
which gives:
\begin{align*}
\|f\|_{Z(I)} \leq& \sum_{p_0 \in \{9/2,18\}} \left( \sum_{N\geq 1} N^{\frac{10}{p_0}-1}\left[\sum_{1\leq M \leq N} \|P_{N,M}f\|_{l^{\frac{4p_0}{p_0-2}}_\gamma L^{p_0}_{x,y,t}(\R\times \T^2\times I)}\right]^2 \right)^{1/2}\\
\lesssim&  \left( \sum_{N\geq 1} \left[\sum_{1\leq M \leq N} \left(\frac{M}{N}\right)^{1/20}  \|P_{N,M}f\|_{X^1(I)} \right]^2\right)^{1/2} \lesssim \left(\sum_{N\geq 1}\sum_{1\leq M \leq N} \left(\frac{M}{N}\right)^{1/20}\|P_{N,M}f\|_{X^1(I)}^2\right)^{1/2} \\
\lesssim& \|f\|_{\widetilde X_\rho^1(I)}
\end{align*}
which proves \eqref{tilde X and Z}. We will prove the following lemma:
\begin{lemma}\label{tilde WP}
Let $D\geq 1$ be dyadic, $\theta \in (0,1]$, and $f_k \in X_c^1(\R)$ be a solution to the linear equation:
$$
(i\partial_t+\Delta_{\mathbb{R}\times\mathbb{T}^2})f_k -\sum_{1\leq \alpha \leq A}F'(\omega_{k,D}^{\alpha})f_k=0
$$
where $\omega_{k,D}^\alpha$ is $P_{\leq D} \omega_k^\alpha$ if $N_k^\alpha \leq 1$ and $(P_{\leq DN_k^\alpha}-P_{\geq D^{-1}N_k^\alpha})\omega^\alpha_k$ if $N_k^\alpha \to +\infty$. Then there exists $\rho_0=\rho_0(\theta, D)$ such that for all $\rho \leq \rho_0$ and $k \geq k(\rho)$ sufficiently large, it holds that:
\begin{equation}\label{step 302}
\|f_k\|_{\widetilde X^1_\rho(\R)}\lesssim_{D, \theta} \|f_k(0)\|_{\widetilde H^1_\rho(\R\times \T^2)}.
\end{equation}
\end{lemma}

With this lemma in hand, we can finish the proof as follows: by the decompositions in $\eqref{dec s1}$, \eqref{dec E}, and \eqref{dec ls}, there exists a dyadic $D\geq 1$ such that $\|\omega_k^\alpha -\omega_{k, D}^\alpha\|_{X^1(\R)}\leq \theta$. Denoting by $\widetilde f^{J, \kappa}_k$ the solution to the initial value problem:
$$
(i\partial_t +\Delta_{\mathbb{R}\times\mathbb{T}^2})\widetilde f_k^{J,\kappa} -F'(\omega_{k,D}^\alpha) \widetilde f_k^{J, \kappa}=0, \quad \widetilde f_k^{J, \kappa}(0)=R_k^{J,1}:= (1-\widetilde P_\kappa^1)P_{\geq \kappa^{-1}}R_k^J
$$
It follows from $\eqref{energy est 0.5}$ that:
$$
\|f_k^{J, \kappa} -\widetilde f_k^{J,\kappa}\|_{X^1(\R)}\lesssim \theta
$$
for $k$ large enough, whereas it follows from Lemma \ref{tilde WP} and \eqref{tilde X and Z} above that for $\rho \leq \rho(\theta, D)$:
$$
\|\widetilde f^{J, \kappa}_k\|_{Z(\R)}\lesssim \|\widetilde f_k^{J, \kappa}\|_{\widetilde X^1_\rho(\R)}\lesssim_{D, \theta} \|R_k^{J,1}\|_{\widetilde H^1_\rho}
$$
for $k$ large enough. Choosing $\kappa=\rho^{100}$, we have that $\|R_k^{J,1}\|_{\widetilde H^1_\rho} \lesssim \kappa^{\delta_0}$ and $\eqref{step 3 goal}$ now follows by combining the above two inequalities. We are left with proving Lemma \ref{tilde WP}.

\subsection{Proof of Lemma \ref{tilde WP}} As before, we denote for any interval $I \subset \mathbb{R}$:
$$
\|f\|_{\widetilde N_\rho(I)}=\left\| \int_0^t e^{i(t-s)\Delta}f(s) ds\right\|_{\widetilde X^1_\rho(I)}.
$$
Using the time-divisibility of the $Z-$norm and the nature of the decomposition in \eqref{dec s1}, \eqref{dec E}, and \eqref{dec ls}, we can split $\R$ for each $\mu>0$ into $O_{\theta, \mu}(1)$ intervals $I$ for which:
\begin{equation}\label{Z small 303}
\sum_{1\leq \alpha \leq A} \|\omega_{k,D}^\alpha\|_{Z'(I)}\leq \mu.
\end{equation}
We will eventually choose $\mu=\mu(\theta,D,A,Q)$ independent of $\rho$, $\kappa$, which means that it would be enough to prove that the analogue of $\eqref{step 302}$ for each interval $I$ (with $f_k(0)$ replaced by the initial value of $f_k$ on that interval). By \eqref{energy est 0.5}, this would follow if we prove that for each interval $I$ and for every $1\leq \alpha \leq A$,
\begin{equation}\label{step 304}
\|F'(\omega_{k,D}^\alpha) f\|_{\widetilde N_\rho(I)}\leq \frac{1}{2A}\|f\|_{\widetilde X^1_\rho}
\end{equation}
if $k\geq k(\rho)$ is large enough.
To prove this we will need the following refinement of \eqref{extra refinement}: suppose that $1\leq M_1 \leq N_1$, $1\leq M_2 \leq N_2\leq N_1$, and $N_3\leq N_2$ are all dyadic. Then:
\begin{equation}\label{step 305}
\begin{split}
&\|P_{N_1, M_1} u_1 P_{N_2, \leq M_2} u_2P_{N_3} u_3\|_{L^2_{x,y,t}(\R\times \T^2\times I)}\\
&\lesssim \left[\left(\frac{N_2}{N_1}+\frac{1}{N_2}\right)\frac{M_2}{N_2} \frac{M_1}{M_1+N_2}\right]^{\delta_0/2}\|P_{N_1, M_1}u_1\|_{Y^0(I)}\left(N_2\|P_{N_2, \leq M_2} u_2\|_{Y^0(I)}\right)\left(N_3\|P_{N_3} u_3\|_{Y^0(I)}\right).
\end{split}
\end{equation}
This follows by interpolating \eqref{extra refinement} and the following refinement of $\eqref{trilinear Z}$ obtained by using \eqref{Step 301}: if $C$ is any cube in $\R^3$ with side-length $\sim N_2$:
\begin{align*}
&\|P_C P_{N_1,M_1}u_1P_{N_2,\le M_2}u_2 P_{N_3}u_3\|_{L^2_{x,y,t}(\R\times \T^2\times I)}\\
&\lesssim  \|P_C P_{N_1, M_1}u_1\|_{l^{\frac{36}{5}}_\gamma L_{x,y,t}^{\frac{9}{2}}} \|P_{N_2, \leq M_2}u_2\|_{l^{\frac{36}{5}}_\gamma L_{x,y,t}^{9/2}}\|P_{N_3}u_3\|_{l^{\frac{9}{2}}_\gamma L_{x,y,t}^{18}}\\
&\lesssim N_2^{\frac{7}{18}}\left(\frac{M_1}{M_1+N_2}\right)^{1/20}\|P_C P_{N_1, M_1} u_1\|_{Y^0} \|P_{N_2, \le M_2}u_2\|_{l^{\frac{36}{5}}_\gamma L_{x,y,t}^{9/2}}\|P_{N_3}u_3\|_{l^{\frac{9}{2}}_\gamma L_{x,y,t}^{18}}\\
&\lesssim   N_2 N_3\left(\frac{M_2}{N_2}\frac{M_1}{M_1+N_2}\right)^{1/20} \|P_C u_1\|_{Y^0}\|u_2\|_{Y^0}\|u_3\|_{Y^0}.
\end{align*}

To continue we split into two cases:
\medskip

{\bf Case 1: $N_k \leq 1$ for all $k$.} In this case, we have that $\omega_{k,D}^\alpha=P_{\leq D}\omega_k^\alpha$. Writing $f=P_{\leq \rho^{-1}}f +P_{\geq \rho^{-1}}f$, we estimate using \eqref{nl refined}:
$$
\|F'(\omega_{k,D}^\alpha) P_{\leq \rho^{-1}} f\|_{N(I)}\lesssim \mu^3\|P_{\leq \rho^{-1}}f\|_{X^1(I)} \lesssim \mu^3 \|f\|_{\widetilde X^1_\rho(I)}
$$
since $c_{\rho, N, M}\sim 1$ if $N \lesssim \rho^{-1}$. On the other hand, for $N \geq \rho^{-1}$,
$$
P_{N,M}[F'(\omega_{k,D}^\alpha) P_{\geq \rho^{-1}} f]=\sum_{N'\sim_D N, M'\sim_D M}P_{N,M}[F'(\omega_{k,D}^\alpha)  P_{N',M'}P_{\geq \rho^{-1}} f]
$$
and as a result, we can estimate:
\begin{align*}
\|F'(\omega_{k, D}^\alpha)P_{\geq \rho^{-1}} f \|_{\widetilde N_\rho(I)}^2\lesssim& \sum_{N,M}c_{\rho, N,M}^2 \|P_{N,M}[F'(\omega_{k,D}^\alpha) P_{\geq \rho^{-1}} f]\|^2_{N(I)}\\
\lesssim& \sum_{N,M}\sum_{N'\sim_D N,\, M'\sim_D M}c_{\rho, N, M}^2 \|F'(\omega_{k,D}^\alpha) P_{N',M'}P_{\geq \rho^{-1}}f \|^2_{N(I)}\\
\lesssim&\sum_{N,M}\sum_{N'\sim_D N,\, M'\sim_D M}c_{\rho, N', M'}^2 \mu^3 \|P_{N',M'}P_{\geq \rho^{-1}}f \|^2_{X^1(I)}\lesssim_D \mu^3 \|f\|_{\widetilde X^1_\rho(I)}^2,
\end{align*}
which gives \eqref{step 304} in this case provided $\mu$ is chosen small enough depending on $A, D, \theta$.

\medskip
{\bf Case 2: $N_k \to +\infty$:} The analysis in this case is more delicate. We will assume throughout that $N_k^\alpha \geq \rho^{-2}$. The contribution of low frequencies $P_{\leq 2^{10}D N_k^\alpha}$ is easily estimated using $\eqref{nl refined}$ as follows:
\begin{equation}\label{step 305B}
\|F'(\omega_{k,D}) P_{\leq 2^{10}D N_k^\alpha}f\|_{N(I)} \lesssim_D \mu^3 \|P_{\leq 2^{10}D N_k^\alpha}f\|_{Z(I)} \lesssim \mu^3 \|f\|_{\widetilde X^1_\rho(I)}.
\end{equation}

We now turn to the contribution of $P_{\geq 2^{10}D N_k^\alpha}f$: by inspecting the Fourier supports, we notice that:
\begin{equation}
P_{N,M}[F'(\omega_{k,D}^\alpha)P_{\geq 2^{10}DN_k^\alpha} f]=\sum_{N'\sim N, \, M'\sim M}P_{N,M}[F'(\omega_{k,D}^\alpha)P_{N',M'}P_{\geq 2^{10}DN_k^\alpha} f]
\end{equation}
if $M\geq 2^5DN_k^\alpha$, and 
\begin{equation}\label{non diag 307}
P_{N,\leq 2^5 DN_k^\alpha}[F'(\omega_{k,D}^\alpha)P_{\geq 2^{10}DN_k^\alpha} f]=\sum_{N'\sim N, \, M'\leq 2^{5}DN_k^\alpha}P_{N,\leq 2^5 DN_k^\alpha}[F'(\omega_{k,D}^\alpha)P_{N',M'}P_{\geq 2^{10}DN_k^\alpha} f].
\end{equation}

In the case when $M\geq 2^5DN_k^\alpha$, we can proceed as in {\bf Case 1} to get:
\begin{equation}\label{step 306}
\|\sum_{N\geq 1}\sum_{M >2^5DN_k^\alpha}P_{N,M}\left[F'(\omega_{k,D}^\alpha) P_{\geq 2^{10}DN_k^\alpha} f\right]\|_{\widetilde N_\rho(I)}\lesssim \mu^3 \|f\|_{\widetilde X_\rho^1(I)}.
\end{equation}

Moving on to the case when $M \leq 2^5D N_k^\alpha$, we first denote $\tilde f \in \{ P_{\geq 2^{10}D N_k^\alpha}f, \overline{P_{\geq 2^{10}D N_k^\alpha}f}\}$, $\widetilde \omega_{k,D}^\alpha \in \{\omega_{k,D}^\alpha, \overline{\omega_{k,D}^\alpha}\}$ and estimate, for $N' \geq 2^5D N_k^\alpha$, $M^\prime\le 2^5DN^\alpha_k$,  using Proposition \ref{N prop}, and \eqref{step 305}:
\begin{align*}
\|\mathfrak{O}_{4,1}(\omega_{k,D}^\alpha,P_{N', M'}f)\|_{N(I)}=& \sup_{\|h\|_{Y^{-1}(I)}=1} \left|\int_{\R\times \T^2\times I} h\cdot \mathfrak{O}_{4,1}(\omega_{k,D}^\alpha,P_{N', M'}f) \;dxdydt\right|\\
\lesssim& \|\mathfrak{O}_{2,1}(\omega^\alpha_{k,D},P_{N', M'} f)\|_{L_{x,y,t}^2} \sup_{\|h\|_{Y^{-1}(I)}=1} \|\mathfrak{O}_{2,1}(\omega^\alpha_k,(P_{\leq 2N'}- P_{\leq N'/2})h)\|_{L^2_{x,y,t}}\\
\lesssim&_{D,\theta}  \;\mu^4 \left[\left(\frac{N_k^\alpha}{N'}+\frac{1}{N_k^\alpha}\right)\frac{M'}{M'+N_k^\alpha}\right]^{\delta_0/2} \|P_{N^\prime, M^\prime}\tilde f\|_{Y^0(I)}.
\end{align*}
Note that\footnote{Notice that here, we use crucially the gain from \cite[Proposition $3.5$]{HeTaTz}.}:
\begin{align*}
\left[\left(\frac{N_k^\alpha}{N'}+\frac{1}{N_k^\alpha}\right)\frac{M'}{M'+N_k^\alpha}\right]^{\delta_0/4}\lesssim& \left(\frac{M'}{N'}\right)^{\delta_0/4}\lesssim c_{\rho, N',M'} \quad &\hbox{ if } N' \leq (N_k^\alpha)^2, \\
\lesssim& \left(\frac{M'}{(N_k^\alpha)^2}\right)^{\delta_0/4} \lesssim_D \rho^{\delta_0/2} \lesssim c_{\rho, N',M'} \quad &\hbox{ if } N' \geq (N_k^\alpha)^2\geq \rho^{-4}.
\end{align*}
As a result we get that:
\begin{align*}
\sum_{1\leq M' \leq 2^5DN_k^\alpha}\|\mathfrak{O}_{4,1}(\omega_{k,D}^\alpha,P_{N', M'}f)\Vert_{N(I)}\lesssim  \mu^4 \left(\frac{N_k^\alpha}{N'}+\frac{1}{N_k^\alpha}\right)^{\delta_0/4}\|P_{N^\prime}\tilde f\|_{\widetilde X^1_\rho(I)}
\end{align*}
which gives by \eqref{non diag 307} combined with \eqref{step 306} that the contribution of $P_{\geq 2^{10}DN_k^\alpha} f$ is 
$$
\|F'(\omega_{k,D}^\alpha)  P_{\geq DN_k^\alpha} f \|_{N(I)}\lesssim_{D, \theta} \mu^3 \|f\|_{\widetilde X^1_\rho(I)}.
$$
This and \eqref{step 305B} give the desired bound if $\mu\leq \mu_0(D, \theta)$, which finishes the proof of Lemma \ref{tilde WP}.
\endproof

We end this section by proving the following smoothing lemma:

\begin{lemma}\label{smoothing lemma}
Suppose that $\psi \in H^1(\R\times \T^2)$ satisfies the following:
\begin{equation}\label{psi sm lemma}
\|\psi\|_{H^1(\R\times \T^2)}\lesssim 1 \quad \quad \|e^{it\Delta} \psi\|_{Z(\R)}\leq \delta^2
\end{equation}
for some $0<\delta<1$. Then for any $R>0$, there exists $C(R) \geq 1$ such that:
\begin{equation}\label{L^2 smoothing}
N \|\nabla^1\left( e^{it\Delta} \tilde{P}_\delta^1 \psi\right)\|_{ L^2_{x,y,t}(\{|p-p_0|\leq R N^{-1}\}\times \{|t-t_0|\leq RN^{-2}\})}\leq C(R) \delta^{\frac{1}{100}}.
\end{equation}
for any dyadic integer $N\geq 1$ and $t_0\in \R$ and any $p_0 \in \R\times \T^2$.
\end{lemma}

\proof Without any loss of generality, we may assume that $R=1, p_0=0$, and $t_0=0$ thanks to space and time translation invariance. From the definition of the $Z-$norm and $\eqref{psi sm lemma}$, we have:
$$
\|e^{it\Delta} P_K \psi\|_{L^{18}_{x,y,t}(\{|p|\leq N^{-1}\}\times \{|t|\leq N^{-2}\})} \lesssim \delta^2 K^{\frac{2}{9}},
$$
from which it follows by H\"older's inequality that:
\begin{equation}\label{smooth lemma 001}
N\|\nabla^1 (e^{it\Delta}P_K \widetilde P_\delta^1\psi)\|_{L^{2}_{x,y,t}(\{|p|\leq N^{-1}\}\times \{|t|\leq N^{-2}\})} \lesssim \delta^2 N^{-\frac{11}{9}}K^{\frac{11}{9}}.
\end{equation}
On the other hand, estimating differently, we have from the smoothing estimate \eqref{1D smoothing} that:
\begin{equation}\label{smooth lemma 002}
\begin{split}
N\|\nabla^1(e^{it\Delta} P_K P_\delta^1 \psi) \|_{L^2_{x,y,t}(\{|p|\leq N^{-1}\}\times \{|t|\leq N^{-2}\})}&\lesssim N^{1/2}\|\nabla^1(e^{it\Delta} P_K P_\delta^1 \psi) \|_{L_x^\infty L^2_{y,t}(\{|p|\leq N^{-1}\}\times \{|t|\leq N^{-2}\})}\\
&\lesssim N^{1/2}(\delta K)^{-1/2}.
\end{split}
\end{equation}
Combining $\eqref{smooth lemma 001}$ and $\eqref{smooth lemma 002}$ and summing in $K$ yields $\eqref{L^2 smoothing}$.
\end{proof}

\section{Appendix: Local Theory of the Resonant System}\label{Sec-AppSS}

In this appendix, we collect some facts about resonant quintic system \eqref{SS}. This system is the Hamiltonian flow of the positive definite Hamiltonian:
\begin{equation}\label{Hamiltonian}
H[\vec u] =\frac{1}{2}\sum_{q\in\mathbb{Z}^2} \int_\R |\partial_x u_q|^2 dx + \frac{1}{6}\sum_{q\in \Z^2, n \in \mathbb{N}} \int_{\R} | \sum_{\substack{p_1-p_2+p_3=q\\|p_1|^2-|p_2|^2+|p_3|^2=n}} u_{p_1} \overline{u_{p_2}} u_{p_3} |^2dx.
\end{equation}
In addition, \eqref{SS} enjoys the following conservation laws:  for weights $g(p)=g_{a,b,c}( p)= a+bp+c\vert p\vert^2$, the solutions of \eqref{SS} have formally conserved $g$-energy where, 
\begin{equation}\label{g energy}
E_g({\vec u})=\sum_{p\in\mathbb{Z}^2}g( p)\Vert u_p\Vert_{L^2(\R)}^2.
\end{equation}
Indeed,
\begin{equation*}
\begin{split}
\frac{d}{dt}E_g({\vec u})&=\sum_{\substack{ p_0+p_2+p_4=p_1+p_3+p_5\\ \vert p_0\vert^2+\vert p_2\vert^2+\vert p_4\vert^2=\vert p_1\vert^2+\vert p_3\vert^2+\vert p_5\vert^2}}g(p_0)\hbox{Im}\left\{\int_{\mathbb{R}}\overline{u_{p_0}u_{p_2}u_{p_4}}u_{p_1}u_{p_3}u_{p_5}dx\right\}\\
&=\frac{1}{3}\sum_{\substack{ p_0+p_2+p_4=p_1+p_3+p_5\\ \vert p_0\vert^2+\vert p_2\vert^2+\vert p_4\vert^2=\vert p_1\vert^2+\vert p _3\vert^2+\vert p_5\vert^2}}\left[g(p_0)+g(p_2)+g(p_4)\right]\hbox{Im}\left\{\int_{\mathbb{R}}\overline{u_{p_0}u_{p_2}u_{p_4}}u_{p_1}u_{p_3}u_{p_5}dx\right\}\\
&=0.
\end{split}
\end{equation*}

The following proposition proves the local well-posedness and small-data scattering for the quintic system \eqref{SS}. We recall that the norm $\vec W(I)$ is defined as:
$$
\Vert {\vec u}\Vert_{{\vec W}(I)}^2:=\sum_{p\in\mathbb{Z}^2}\left[1+\vert p\vert^2\right]\Vert u_p\Vert_{L^6_{x,t}(\mathbb{R}_x\times I)}^2.
$$

\begin{proposition}[Local well-posedness and small-data scattering for \eqref{SS}]\label{LWPSS}
Let ${\vec u}(0)=\{u_p(0)\}_p\in h^1L^2$ satisfy $\|\vec u_0\|_{h^1 L^2} \leq E$, then: 
\begin{enumerate}
\item there exists an open interval $0\in I$ and a unique solution $\vec u(t)$ of \eqref{SS} in $C^0_t(I:h^1L^2)\cap \vec W(I)$.
\item there exists $E_0$ such that if $E \leq E_0$, $\vec u(t)$ is global and scatters in positive and negative infinite time.
\item propagation of regularity: if ${\vec u}(0)\in h^\sigma H^k$ for some $\sigma\ge 1$ and $k\ge 0$, then ${\vec u}\in C^0_t(I:h^\sigma H^k)$.
\end{enumerate}
\end{proposition}

\begin{proof}

The proof follows from a simple fixed point theorem (and classical arguments) once we have established the fundamental nonlinear estimate. By the scalar Strichartz estimates, we see that
\begin{equation*}
\Vert u_j\Vert_{L^6_{x,t}(\mathbb{R}\times I)}\lesssim \Vert u_j(0)\Vert_{L^2}+\sum_{\mathcal{R}(j)}\Vert u_{p_1}\overline{u_{p_2}}u_{p_3}\overline{u_{p_4}}u_{p_5}\Vert_{L^\frac{6}{5}_{x,t}(\mathbb{R}\times I)},
\end{equation*}
where $\mathcal{R}(j)$ was defined in \eqref{SS}.
Multiplying by $\langle j\rangle$ and square-summing, the first term on the right-hand side is bounded by the square of the $h^1L^2_x$-norm. For the second term, we compute using Lemma \ref{SumLem} below that
\begin{equation*}
\begin{split}
&\sum_{j\in\mathbb{Z}^2}\langle j\rangle^2\left[\sum_{\mathcal{R}(j)}\Vert u_{p_1}\overline{u_{p_2}}u_{p_3}\overline{u_{p_4}}u_{p_5}\Vert_{L^\frac{6}{5}_{x,t}(\mathbb{R}\times I)}\right]^2\\
&\lesssim \sum_{j\in\mathbb{Z}^2}\langle j\rangle^2\left[\sum_{\mathcal{R}(j)}\Pi_{k=1}^5\Vert u_{p_k}\Vert_{L^6_{x,t}(\mathbb{R}\times I)}\right]^2\\
&\lesssim\sum_{j\in\mathbb{Z}^2}\big\{\sum_{\mathcal{R}(j)}\Pi_{k=1}^5\langle p_k\rangle^{2}\Vert u_{p_k}\Vert_{L^6_{x,t}(\mathbb{R}\times I)}^2\times\langle j\rangle^2\sum_{\mathcal{R}(j)}\langle p_1\rangle^{-2}\langle p_2\rangle^{-2}\langle p_3\rangle^{-2}\langle p_4\rangle^{-2}\langle p_5\rangle^{-2}\big\}\\
&\lesssim \sum_{j\in\mathbb{Z}^2}\sum_{\mathcal{R}(j)}\Pi_{k=1}^5\langle p_k\rangle^{2}\Vert u_{p_k}\Vert_{L^6_{x,t}(\mathbb{R}\times I)}^2\lesssim \Vert {\vec u}\Vert_{\vec W (I)}^{10}.
\end{split}
\end{equation*}
Consequently, we obtain
\begin{equation*}
\Vert {\vec u}\Vert_{\vec W(I)}\lesssim \Vert e^{it \partial_{xx}} \vec u_0\Vert_{\vec W(I)}+\Vert {\vec u}\Vert_{\vec W(I)}^5.
\end{equation*}
This and the Strichatz estimate:
$$
\|e^{it \partial_{xx}} \vec u_0\|_{\vec W (\R)}\lesssim \| \vec u_0\|_{h^1 L^2}\lesssim E
$$
allows one to run a classical fixed-point argument in $\vec W(I)\cap C_t (I: h^1L^2)$ provided $I$ or $E$ is small enough. The rest of the proposition follows from standard arguments.
\end{proof}

\begin{lemma}\label{SumLem}
There holds that
\begin{equation*}
\sup_{j\in\mathbb{Z}^2}\left\{\langle j\rangle^2\sum_{(p_1,p_2,p_3,p_4,p_5)\in\mathcal{R}(j)}\langle p_1\rangle^{-2}\langle p_2\rangle^{-2}\langle p_3\rangle^{-2}\langle p_4\rangle^{-2}\langle p_5\rangle^{-2}\right\}\lesssim 1.
\end{equation*}
\end{lemma}

\begin{proof}

Without loss of generality, we may assume that
\begin{equation}\label{WLOG}
\vert p_1\vert\le\vert p_3\vert\le\vert p_5\vert,\quad \vert p_2\vert\le \vert p_4\vert,\quad\max(\vert j\vert,\vert p_4\vert)\sim\vert p_5\vert.
\end{equation}
The first equation in the resonant condition specifies
\begin{equation*}
p_5=p_2+p_4+j-p_1-p_3,
\end{equation*}
while the second condition specifies that $p_3$ lies on a specific circle $\mathcal{C}$, namely
\begin{equation*}
\left\vert p_3-\frac{p_2+p_4+j-p_1}{2}\right\vert^2=\frac{2\left[\vert p_2\vert^2+\vert p_4\vert^2+\vert j\vert^2-\vert p_1\vert^2\right]-3\vert p_2+p_4+j-p_1\vert^2}{4}.
\end{equation*}
We can use Lemma \ref{LemAp1} below to bound the number of possible $p_3$:
\begin{equation*}
\sum_{p_3;\vert p_3\vert\ge\max(\vert p_1\vert,\vert p_2\vert,\vert p_4\vert);\,p_3\in\mathcal{C}}\langle p_3\rangle^{-2}\lesssim\langle \max(\vert p_1\vert,\vert p_2\vert,\vert p_4\vert)\rangle^{-1}.
\end{equation*}
Using this, we obtain that
\begin{equation*}
\begin{split}
S_1&=\sum_{\substack{(p_1,p_2,p_3,p_4,p_5)\in\mathcal{R}(j);\\\vert p_2\vert\le\vert p_4\vert\le\vert p_3\vert;\vert p_1\vert\le \vert p_3\vert\le\vert p_5\vert}}
\langle p_1\rangle^{-2}\langle p_2\rangle^{-2}\langle p_3\rangle^{-2}\langle p_4\rangle^{-2}\frac{\langle j\rangle^2}{\langle p_5\rangle^2}\\
&\lesssim\sum_{p_1,p_2,p_4}\langle p_1\rangle^{-2}\langle p_2\rangle^{-2}\langle p_4\rangle^{-2}\sum_{\substack{p_3:(p_1,p_2,p_3,p_4,p_2+p_4+j-p_1-p_3)\in\mathcal{R}(j),\\ \vert p_3\vert\ge\max(\vert p_1\vert,\vert p_2\vert,\vert p_4\vert)}}\langle p_3\rangle^{-2}\\
&\lesssim \sum_{p_1,p_2,p_4}\langle p_1\rangle^{-2}\langle p_2\rangle^{-2}\langle p_4\rangle^{-2}\langle\vert p_1\vert+\vert p_2\vert+\vert p_4\vert\rangle^{-1}\lesssim 1.
\end{split}
\end{equation*}
The sum when $\vert p_3\vert\le\vert p_4\vert$ is bounded similarly, using Lemma \ref{LemAp1} to bound the sum over $p_4$ instead of the bound over $p_3$.
\end{proof}

We are left with 

\begin{lemma}\label{LemAp1}

For any $P\in \mathbb{R}^2$, $R>0$ and $A>1$ there holds that:
\begin{equation}
\sum_{\substack{\vert p\vert\geq A\\ p \in \mathbb{Z}^2\cap C(P,R)}} \frac{1}{\langle p\rangle^2} \lesssim A^{-1},
\end{equation}
where $C(P,R)$ denotes the circle of radius $R$ centered at $P$. 
\end{lemma}

\begin{proof}
Indeed, it suffices to prove that for any $k\ge 0$,
\begin{equation*}
\left\vert\{p:\,\vert p\vert\in [2^kA,2^{k+1}A];\,p \in C(P,R)\}\right\vert\lesssim  2^kA.
\end{equation*}
When $R\le 2^{k+10}A$, this follows directly from the fact that the integer points are $1$-separated and hence there can be at most $2\pi R$ points on a circle of radius $R$. When $R\ge 2^{k+10}A$, we have that
\begin{equation*}
\left\vert D(0,2^{k+1}A) \cap C(P,R)\right\vert \lesssim 2^kA,
\end{equation*}
where $D(0,K)$ denotes the disc centered at the origin of radius $K$. To see this, let $\alpha$ be the angle of the above arc in $C(P,R)$. This arc has length $R\alpha$ while from the fact that the chord has to fit into the disc, we also obtain that $R\tan\alpha\le 2^{k+2}A$ which gives the result. 
\end{proof}

One can also prove by standard arguments a stability result for \eqref{SS} similar to Lemma \ref{stability}, which will allow to prove:

\begin{lemma}\label{reverse implication}
Assume the conclusion of Theorem \ref{MainThm} holds for all initial data $u_0\in H^1(\R\times \T^2)$ with full energy $L(u_0) < E_{max}$. Then Conjecture \ref{SSConj} holds true for all initial data $\vec v_0 \in h^1 L^2$ satisfying $E_{ls}(\vec v_0) < E_{max}$. In particular, if all finite full-energy (i.e. $H^1$) solutions scatter for $\eqref{NLS}$, then the same holds for finite $E_{ls}$-energy solutions of \eqref{SS}.
\end{lemma}

%
%
%

\end{document}